\documentclass[a4paper,reqno,12pt]{amsart}

\usepackage{cite}
\usepackage{amssymb,amsthm,graphicx,bm,bbm,mathrsfs,mathtools,exscale,amsmath,url,color,extarrows}
\usepackage{geometry,multicol,rotfloat,MnSymbol} 
\usepackage{dcolumn}
\usepackage{comment,verbatim}
\usepackage{bold-extra,longtable,mdwlist}
\usepackage[colorlinks,linkcolor=blue,urlcolor=black]{hyperref}
\usepackage[font=small, labelfont=bf, width=0.8\textwidth]{caption}
\usepackage{physics}
\usepackage{enumitem}
\usepackage{array,booktabs}
\usepackage{accents}
\usepackage{subcaption}
\usepackage[normalem]{ulem}

\usepackage{epstopdf}
\usepackage{tikz}
\usepackage{pstricks} 
\usetikzlibrary{graphs,patterns,decorations.markings,arrows,matrix}

\allowdisplaybreaks
\theoremstyle{definition}
\newtheorem{thm}{Theorem}
\newtheorem{lm}[thm]{Lemma}
\newtheorem{cor}[thm]{Corollary}
\newtheorem{defn}[thm]{Definition}

\newtheorem{remark}[thm]{Remark}
\newtheorem{prop}[thm]{Proposition}
\newtheorem{conj}[thm]{Conjecture}

\numberwithin{thm}{section}

\newcommand{\Rmnum}[1]{uppercase\expandafter{\romannumeral #1}}

\newcommand{\pf}{\mathrm{Pf} }
\newcommand{\comp}[1]{#1^{\mathrm{c}}}
\def\ii{{\rm i}}
\def\dd{{\rm d}}

\def\yalph{Y}

\newcommand*{\bigdot}[1]{%
  \accentset{\mbox{\large\bfseries .}}{#1}}

\numberwithin{equation}{section}

\DeclareMathOperator{\Span}{Span}

\DeclareMathOperator{\Pf}{Pf} 

\def\be{\begin{equation}}
\def\ee{\end{equation}}

\title[Symmetric functions from the six-vertex model in half-space]{Symmetric functions from the six-vertex model \\ in half-space}
\author{Alexandr Garbali, Jan de Gier, William Mead and Michael Wheeler}
\address{School of Mathematics and Statistics, University of Melbourne, Victoria 3010, Australia}
\email{alexandr.garbali@unimelb.edu.au, jdgier@unimelb.edu.au, wmead@student.unimelb.edu.au, wheelerm@unimelb.edu.au}

\begin{document}
\maketitle 

\begin{abstract}
We study the stochastic six-vertex model in half-space with generic integrable boundary weights, and define two families of multivariate rational symmetric functions. Using commutation relations between double-row operators, we prove a skew Cauchy identity of these functions. In a certain degeneration of the right-hand side of the Cauchy identity we obtain the partition function of the six-vertex model in a half-quadrant, and give a Pfaffian formula for this quantity. The Pfaffian is a direct generalization of a formula obtained by Kuperberg in his work on symmetry classes of alternating-sign matrices. 

One of our families of symmetric functions admits an integral (sum over residues) formula, and we use this to conjecture an orthogonality property of the dual family. We conclude by studying the reduction of our integral formula to transition probabilities of the (initially empty) asymmetric simple exclusion process on the half-line. 
\end{abstract}

\setcounter{tocdepth}{1}
\tableofcontents

\section{Introduction}

\subsection{Background}
\label{intro:background}
The connection between symmetric functions, exactly solvable vertex models and stochastic processes is a fertile branch of integrable probability, and is motivated by the rigorous study of the asymptotic behaviour of certain random variables within the KPZ universality class. Cornerstone results are the Macdonald processes \cite{borodin_corwin_2014} and their half-space analogues \cite{barraquand_half-space_2020}, as well as the construction of probability measures on full-space from Cauchy identities of multivariate partition functions within the stochastic six-vertex model \cite{borodin_family_2017,borodin_stochastic_2016}. Generalizations to the setting of higher-spin \cite{corwin2016stochastic,borodin_higher_2018} and higher-rank \cite{borodin_coloured_2022,borodin2020observables} models are also known, and have been very topical in recent literature. 

The study of stochastic vertex models in half-space is less developed and comes with additional complications due the presence of a boundary. Some of the first algebraic results in this setting were obtained in \cite{barraquand_stochastic_2018}, with application to rigorous asymptotics of the asymmetric simple exclusion process (ASEP) on the half-line, albeit for a rather restricted case of hopping rates at the boundary. More recently, the results of \cite{barraquand_stochastic_2018} were upgraded to fully generic boundaries in \cite{he_boundary_2023}. In each of these works the interplay with symmetric functions is always close to the foreground, as both \cite{barraquand_stochastic_2018,he_boundary_2023} rely on symmetric function identities obtained in \cite{betea_refined_2015} and \cite{imamura_solvable_2022} to match half-space vertex-model expectations with quantities in the half-space and free-boundary Schur processes, respectively.

A much wider body of work exists in the context of the ASEP (recovered as a continuous-time limit of the stochastic six-vertex model \cite{aggarwal_convergence_2016}) with open boundaries. The integrable structure of the ASEP on a strip (with open left and right boundaries), arising from commuting double-row transfer matrices \cite{sklyanin_boundary_1988}, was reviewed in \cite{Crampe_2014} with a successive exploration of various integrable boundary conditions in \cite{Crampe_2016}. The dynamic phase diagram based on the Bethe ansatz solution for the spectral gap of the open ASEP was computed in \cite{gier_essler_ASEP}. A more recent approach to spectral gap analysis for the open ASEP uses high genus Riemann surface analysis of the Bethe equations \cite{Godreau_2020}. Integrability of the classical open ASEP is relevant to quantum stochastic systems described by integrable Lindbladians \cite{essler_piroli}.

Substantial literature has also been devoted to the stationary properties of the ASEP with two open boundaries. In the context of symmetric polynomials, the ASEP stationary state is related to Askey--Wilson polyomials \cite{Uchiyama_2004,Corteel2006TableauxCF} and the multivariate Koornwinder polynomials in the case of multi-species, or coloured, models \cite{Cantini_2016,Finn_2017}.

In the setting of the half-line (one non-trivial boundary), distribution functions can be accessed following the initial results of \cite{tracy_asymmetric_2013} for reflecting boundary conditions. Recent works include results for transition probabilities, boundary current fluctuation analysis for special boundary conditions \cite{barraquand_half-space_2020,barraquand_half-space_2022,imamura_solvable_2022,he_boundary_2023} and Markov duality analysis \cite{barraquand_markov_2024}. While analyses of current fluctuations both at the boundary and in the bulk has been studied for last passage percolation \cite{baik_pfaffian_2018}.



The goal of the present text is to further develop our understanding of stochastic vertex models in half-space, viewed through the lens of multivariate symmetric functions. Stated in the simplest possible terms, our aim is to construct half-space analogues of the symmetric functions introduced by Borodin in \cite{borodin_family_2017}, and to study their algebraic properties (Cauchy summation identities, integral formulas and orthogonality). As mentioned above, classical symmetric functions (specifically, Hall--Littlewood and Schur polynomials) have already played an important role in probabilistic results in the half-space setting. In the current paper, however, our rational symmetric functions are new, and they are a key result in their own right.

\medskip
We proceed to a detailed summary of our main results.

\subsection{Six-vertex model with a boundary}
\label{intro:6vm}

Our primary algebraic tool is the \textit{stochastic six-vertex model}. The interpretation of the six-vertex model as a Markov process in the quadrant dates back to a number of earlier works; see, for example, \cite{gwa_six-vertex_1992,borodin_stochastic_2016,borodin_higher_2018}. In the current text we will follow the same conventions as \cite[Chapter 1]{borodin_coloured_2022}. 

We shall assign weights to finite collections of \textit{paths} drawn on the square lattice. Each edge of the lattice (whether horizontal or vertical) supports at most one path, and the vertices obey the \textit{ice rule}, which enforces that the total number of paths entering a vertex is the same as the total number of paths exiting it. Imposing the ice rule, one obtains six possible types of vertices:
\begin{align}
		\label{Stochastic 6VM weights table-intro} 
		\begin{tabular}{ccc}
			\qquad
			\begin{tikzpicture}
				\draw[gray,dashed,line width=1pt,-] (-1,0) -- (1,0);
				\draw[gray,dashed,line width=1pt,-] (0,-1) -- (0,1);
			\end{tikzpicture}
			\qquad
			&
			\qquad
			\begin{tikzpicture}
				\draw[gray,dashed,line width=1pt,-] (-1,0) -- (1,0);
				\draw[red,line width=2pt,->] (0,-1) -- (0,1);
			\end{tikzpicture}
			\qquad
			&
			\qquad
			\begin{tikzpicture}
				\draw[gray,dashed,line width=1pt,-] (0,1) -- (0,0) -- (-1,0);
				\draw[red,line width=2pt,->] (0,-1) -- (0,0) -- (1,0);
			\end{tikzpicture}
			\qquad
			\\
			\qquad
			$1$
			\qquad
			& 
			\qquad
			$\dfrac{q(1-z)}{1-qz}$
			\qquad
			& 
			\qquad
			$\dfrac{1-q}{1-qz}$
			\qquad
			\\
			\vspace{0.1cm}
			\\
			\qquad
			\begin{tikzpicture}
				\draw[gray,dashed,line width=1pt,-] (-1,0) -- (1,0);
				\draw[gray,dashed,line width=1pt,-] (0,-1) -- (0,1);
				\draw[red,line width=2pt,->] (-1,0) -- (-0.2,0) -- (0,0.2) -- (0,1);
				\draw[red,line width=2pt,<-] (1,0) -- (0.2,0) -- (0,-0.2) -- (0,-1);
			\end{tikzpicture}
			\qquad
			&
			\qquad
			\begin{tikzpicture}
				\draw[red,line width=2pt,->] (-1,0) -- (1,0);
				\draw[gray,dashed,line width=1pt,-] (0,-1) -- (0,1);
			\end{tikzpicture}
			\qquad
			&
			\qquad
			\begin{tikzpicture}
				\draw[gray,dashed,line width=1pt,-] (0,-1) -- (0,0) -- (1,0);
				\draw[red,line width=2pt,->] (-1,0) -- (0,0) -- (0,1);
			\end{tikzpicture}
			\qquad
			\\
			\qquad
			$1$
			\qquad
			& 
			\qquad
			$\dfrac{1-z}{1-qz}$
			\qquad
			&
			\qquad
			$\dfrac{z(1-q)}{1-qz}$
			\qquad
			\\
		\end{tabular} 
	\end{align}
 where we have indicated the weight of each vertex underneath its picture. Here $q$ denotes the \textit{quantum deformation parameter} arising from the underlying $U_q\left(\widehat{\mathfrak{sl}}_2\right)$ algebra, while $z$ denotes the \textit{spectral parameter} associated to the vertex. The weight of a lattice configuration is defined as the product of the weights of the individual vertices which comprise it. 
 
 The weights \eqref{Stochastic 6VM weights table-intro} have a number of well known fundamental properties. The first is that they obey the Yang--Baxter equation; see Proposition \ref{prop:YB} of the text. This property endows the model with its rich algebraic structure and facilitates the exact computation of observables. The second is that the weights have a sum-to-unity property; see Proposition \ref{prop:stoch}. This property allows one to assign probability measures to collections of paths in the square grid, and is the key feature that in turn allows reductions to one-dimensional stochastic processes such as the \textit{asymmetric simple exclusion process}.

 It turns out to be possible to extend both of the above properties, namely integrability and Markovian dynamics, to the setting of \textit{boundary vertices}. A boundary vertex consists of a single incoming and outgoing edge, such that the two edges join to form a right angle. Once again, each edge of a boundary vertex supports at most one path. However, in contrast to the bulk vertices \eqref{Stochastic 6VM weights table-intro}, we no longer enforce any ice rule on boundary vertices; this means that the total flux of paths through a boundary vertex need not be conserved. In the absence of an ice rule, there are two possible path assignments to the incoming/outgoing edges, leading to four types of boundary vertices:
 \begin{align}
		\label{Stochastic boundary K-weight table-intro}
		\begin{tabular}{cccc}
			\qquad
			\begin{tikzpicture}[baseline={([yshift=-.5ex]current bounding box.center)}]
				\draw[gray,dashed,line width=1pt,-] (1,0) -- (0,1) -- (1,2);
				\draw[blue,fill=blue] (0,1) circle (0.1cm);
			\end{tikzpicture}
			\qquad
			&
			\qquad
			\begin{tikzpicture}[baseline={([yshift=-.5ex]current bounding box.center)}]
				\draw[gray,dashed,line width=1pt,-] (1,0) -- (0,1);
				\draw[red,line width=2pt,->] (0,1) -- (1,2);
				\draw[blue,fill=blue] (0,1) circle (0.1cm);
			\end{tikzpicture}
			\qquad
			&
			\qquad
			\begin{tikzpicture}[baseline={([yshift=-.5ex]current bounding box.center)}]
				\draw[gray,dashed,line width=1pt,-] (0,1) -- (1,2);
				\draw[red,line width=2pt,-] (1,0) -- (0,1);
				\draw[red,line width=2pt,->] (1,0) -- (0.5,0.5);
				\draw[blue,fill=blue] (0,1) circle (0.1cm);
			\end{tikzpicture}
			\qquad
			&
			\qquad
			\begin{tikzpicture}[baseline={([yshift=-.5ex]current bounding box.center)}]
				\draw[gray,dashed,line width=1pt,-] (0,1) -- (1,2);
				\draw[red,line width=2pt,->] (1,0) -- (0,1) -- (1,2);
				\draw[red,line width=2pt,->] (1,0) -- (0.5,0.5);
				\draw[blue,fill=blue] (0,1) circle (0.1cm);
			\end{tikzpicture}
			\qquad
			\\
			\vspace{0cm}
			\\
			\qquad
			$1-h(x)$
			\qquad
			&
			\qquad
			$h(x)$
			\qquad
			&
			\qquad
			$\dfrac{-h(x)}{ac}$
			\qquad
			&
			\qquad
			$1+\dfrac{h(x)}{ac}$
			\qquad
		\end{tabular}
	\end{align}
where we have defined 
\begin{align}
\label{h-intro}
h(x) = \frac{ac(1-x^2)}{(a-x)(c-x)}.   
\end{align} 
The weights \eqref{Stochastic boundary K-weight table-intro} depend on a spectral parameter $x$, but unlike their bulk counterparts, they no longer have any $q$ dependence. Instead, they acquire dependence on two \textit{boundary parameters} $a$ and $c$, whose values are free but the same for all boundary vertices. 

One may verify that the weights \eqref{Stochastic boundary K-weight table-intro}, together with \eqref{Stochastic 6VM weights table-intro}, satisfy the \textit{reflection equation}, which is the boundary-analogue of the Yang--Baxter equation introduced by Sklyanin \cite{sklyanin_boundary_1988}. It is also clear that these weights have a sum-to-unity property and are non-negative for certain mild choices of $x$ and $a,c$, allowing one to maintain links with probability.

\subsection{Double-row operators}
\label{intro:operators}

A standard algebraic device in the setting of lattice models with a boundary is the \textit{double-row transfer matrix}. Motivated by this, we introduce a family of double-row operators whose matrix elements are computed as partition functions in the model \eqref{Stochastic 6VM weights table-intro}--\eqref{Stochastic boundary K-weight table-intro}.

In the sequel, let $\mu,\nu \subset \mathbb{N}$ be finite subsets of the natural numbers\footnote{That is, both $\mu$ and $\nu$ have finitely many elements of finite size.}, and for each $i \in \mathbb{N}$ introduce the indicator function $\eta_i^{\mu} = \bm{1}_{i \in \mu}$ which assigns a value of $1$ if $i$ is an element of $\mu$, and $0$ otherwise. We then construct the following partition function on a semi-infinite lattice:
\begin{align}
\label{A-intro}
A_{\mu \rightarrow \nu}(x)
:=
\begin{tikzpicture}[baseline={([yshift=-.5ex]current bounding box.center)},scale=0.7]
			\draw[lightgray,line width=1.5pt,->] (8,0) -- (1,0) -- (0,0.5) -- (1,1) -- (8,1);
			\draw[lightgray,line width=1.5pt,->] (8,0) -- (7.5,0);
			\draw[blue,fill=blue] (0,0.5) circle (0.1cm);
			\foreach \x in {2,...,7}
			{\draw[lightgray,line width=1.5pt,->] (\x,-0.5) -- (\x,1.5);}
			\node[right] at (8,0) {$0 \leftarrow x$};
			\node[right] at (8,1) {$0 \rightarrow x^{-1}$};
			\node[below] at (2,-0.5) {$\eta^\mu_1$};
			\node[below] at (3,-0.5) {$\eta^\mu_2$};
			\node[below] at (4,-0.5) {$\eta^\mu_3$};
			\node[below] at (5,-0.5) {$\cdots$};
			\node[below] at (6,-0.5) {$\cdots$};
			\node[below] at (7,-0.5) {$\cdots$};
			\node[below] at (2,-1.2) {$\uparrow$}; 
			\node[below] at (3,-1.2) {$\uparrow$};
			\node[below] at (4,-1.2) {$\uparrow$};  
			\node[below] at (2,-1.8) {$y_1$}; 
			\node[below] at (3,-1.8) {$y_2$};
			\node[below] at (4,-1.8) {$y_3$}; 
			\node[above] at (2,1.5) {$\eta^\nu_1$};
			\node[above] at (3,1.5) {$\eta^\nu_2$};
			\node[above] at (4,1.5) {$\eta^\nu_3$};
			\node[above] at (5,1.5) {$\cdots$};
			\node[above] at (6,1.5) {$\cdots$};
			\node[above] at (7,1.5) {$\cdots$};
		\end{tikzpicture}
  \end{align}
  where the assignment of $1$ to any edge of the lattice means that a path is present there, while the assignment of $0$ means that it is vacant. The vertices used in the top row of \eqref{A-intro} are given by the table \eqref{Stochastic 6VM weights table-intro}, where in the $j$-th vertex in the row (read from left to right) we set $z \mapsto x y_j$. Similarly, the vertices used in the bottom row of \eqref{A-intro} are given by \eqref{Stochastic 6VM weights table-intro} under $90^{\circ}$ counterclockwise rotation, where in the $j$-th vertex in the row we set $z \mapsto x/y_j$. The boundary vertex appearing in \eqref{A-intro} has weights given by \eqref{Stochastic boundary K-weight table-intro}. Note that $A_{\mu \rightarrow \nu}(x)$ depends implicitly on the alphabet $Y=(y_1,y_2,\dots)$, but we suppress this dependence in our notation where there is no potential for confusion. 
  
  One evaluates $A_{\mu \rightarrow \nu}(x)$ by computing the weighted sum over all possible path configurations of the picture \eqref{A-intro}; that is, by computing it as a statistical mechanical partition function. Although we have defined $A_{\mu \rightarrow \nu}(x)$ on a semi-infinite lattice, since $\mu$ and $\nu$ are finite subsets of $\mathbb{N}$ it is easily verified that sufficiently far from the boundary vertex one sees only empty vertices (devoid of paths). The weight of such vertices is $1$, meaning that the quantity $A_{\mu \rightarrow \nu}(x)$ is a finite sum of rational functions in $x,y_j,q,a,c$, for any fixed $\mu,\nu$. 
  
  In view of the sum-to-unity property of the bulk weights \eqref{Stochastic 6VM weights table-intro} and boundary weights \eqref{Stochastic boundary K-weight table-intro}, one finds that $A_{\mu \rightarrow \nu}(x)$ obeys a sum-to-unity property:
\begin{align}
\label{eq:sum-to-unity kernel}
\sum_{\nu \in \mathbb{W}} A_{\mu \rightarrow \nu}(x)
=
1,
\quad
\text{for any fixed}\ \mu\in \mathbb{W},
\end{align}
where $\mathbb{W}$ denotes the set of finite subsets of $\mathbb{N}$. This fact allows us to view $A_{\mu \rightarrow \nu}(x)$ as a Markov kernel and to generate a discrete-time Markov process of paths in half-space, similarly to what has been done in the context of the stochastic six-vertex model in a quadrant \cite{borodin_stochastic_2016,borodin_higher_2018}.

It is then natural to construct an infinite-dimensional Markov matrix $A(x)$ with entries $A_{\mu \rightarrow \nu}(x)$, where $\mu$ is the row index and $\nu$ the column index, which acts in the vector space obtained by taking the formal linear span of all finite subsets of $\mathbb{N}$. Our first result is that these operators commute:
\begin{prop}[Proposition \ref{A commutation prop} below]
	\label{A commutation prop-intro}
	Fix $x_1,x_2 \in \mathbb{C}$ and assume there exists $\rho>0$ such that
	\begin{equation}
		\label{A operator commutation condition-intro}
		\abs{\frac{1-x_iy_k}{1-qx_iy_k}\frac{q(1-x_j/y_k)}{1-qx_j/y_k}}\leq \rho <1,
	\end{equation}
    for all $1 \leq i\neq j \leq 2$ and for all $k\in\mathbb{N}$. One then has
	\begin{equation}
		\label{eq:Acommute-intro}
		A(x_1)A(x_2) = A(x_2) A(x_1),
	\end{equation}
 where the latter identity holds in ${\rm End}(\Span \mathbb{W})$.
\end{prop}

The proof of Proposition \ref{A commutation prop-intro} follows from use of the Yang--Baxter and reflection equations for our model. It proceeds along similar lines to the proof that double-row transfer matrices commute (see, for example, \cite{sklyanin_boundary_1988}), however it requires a careful adjustment to the current setting where our operators act in the infinite-dimensional space $\Span\mathbb{W}$. The condition \eqref{A operator commutation condition-intro} is an artefact of the proof, and ensures that the product of operators $A(x_i)A(x_j)$ converges.

\subsection{Symmetric functions and Cauchy identities}
\label{intro:sym-cauchy}

With our commuting double-row operators in hand, there is a natural passage to the definition of a family of symmetric functions. In particular, for all $\mu,\nu \subset \mathbb{N}$ and a fixed alphabet $(x_1,\dots,x_L)$, we define

\begin{align}
		\label{G partition function defn picture-intro}
		G_{\nu/\mu} (x_1,\dots,x_L) & := 
		\begin{tikzpicture}[baseline={([yshift=-.5ex]current bounding box.center)},scale=0.7]
			\draw[lightgray,line width=1.5pt,->] (8,0) -- (1,0) -- (0,0.5) -- (1,1) -- (8,1);
			\draw[lightgray,line width=1.5pt,->] (8,0) -- (7.5,0);
			\draw[lightgray,line width=1.5pt,->] (8,2) -- (1,2) -- (0,2.5) -- (1,3) -- (8,3);
			\draw[lightgray,line width=1.5pt,->] (8,2) -- (7.5,2);
			\draw[lightgray,line width=1.5pt,->] (8,4) -- (1,4) -- (0,4.5) -- (1,5) -- (8,5);
			\draw[lightgray,line width=1.5pt,->] (8,4) -- (7.5,4);
			\draw[lightgray,line width=1.5pt,->] (8,6) -- (1,6) -- (0,6.5) -- (1,7) -- (8,7);
			\draw[lightgray,line width=1.5pt,->] (8,6) -- (7.5,6);
			\draw[blue,fill=blue] (0,0.5) circle (0.1cm);
			\draw[blue,fill=blue] (0,2.5) circle (0.1cm);
			\draw[blue,fill=blue] (0,4.5) circle (0.1cm);
			\draw[blue,fill=blue] (0,6.5) circle (0.1cm);
			\foreach \x in {2,...,7}
			{\draw[lightgray,line width=1.5pt,->] (\x,-0.5) -- (\x,7.5);}
			\node[right] at (8,0) {$0 \leftarrow x_1$};
			\node[right] at (8,1) {$0 \rightarrow x_1^{-1}$};
			\node[right] at (8,2.5) {$\vdots$};
			\node[right] at (8,4.5) {$\vdots$};
			\node[right] at (8,6) {$0 \leftarrow x_L$};
			\node[right] at (8,7) {$0 \rightarrow x_L^{-1}$};
			\node[above] at (2,7.5) {$\eta^\nu_1$};
			\node[above] at (3,7.5) {$\eta^\nu_2$};
			\node[above] at (4,7.5) {$\eta^\nu_3$};
			\node[above] at (5,7.5) {$\cdots$};
			\node[above] at (6,7.5) {$\cdots$};
			\node[below] at (2,-0.5) {$\eta^\mu_1$};
			\node[below] at (3,-0.5) {$\eta^\mu_2$};
			\node[below] at (4,-0.5) {$\eta^\mu_3$};
			\node[below] at (5,-0.5) {$\cdots$};
			\node[below] at (6,-0.5) {$\cdots$};
			\node[below] at (7,-0.5) {$\cdots$};
			\node[below] at (2,-1.2) {$\uparrow$}; 
			\node[below] at (3,-1.2) {$\uparrow$};
			\node[below] at (4,-1.2) {$\uparrow$};  
			\node[below] at (2,-1.8) {$y_1$}; 
			\node[below] at (3,-1.8) {$y_2$};
			\node[below] at (4,-1.8) {$y_3$}; 
		\end{tikzpicture}
	\end{align}
 which is again interpreted as a statistical mechanical partition function, with vertex weights given by \eqref{Stochastic 6VM weights table-intro} and \eqref{Stochastic boundary K-weight table-intro}. From its definition, it is clear that $G_{\nu/\mu} (x_1,\dots,x_L)$ may be interpreted as the concatenation of $L$ discrete-time Markov kernels \eqref{A-intro}; expressing this algebraically via the double-row operators $A(x_i)$, we have that
\begin{align}
		\label{G partition function defn A eq-intro}
		G_{\nu/\mu} (x_1,\dots,x_L) & = \bra{\mu}A(x_1) \cdots A(x_L) \ket{\nu}.
\end{align}
The functions \eqref{G partition function defn A eq-intro} are a key focal point of this work\footnote{There is at least one other known instance of symmetric functions appearing from stochastic vertex models with a boundary; see \cite{zhong2022stochastic}. However, the weights that we use in the present work (both in the bulk and in the boundary) are more general and our results are otherwise disjoint from the work performed in \cite{zhong2022stochastic}.}. In view of the commutativity \eqref{eq:Acommute-intro} of the double-row operators that are used to define them, it is clear that the functions $G_{\nu/\mu} (x_1,\dots,x_L)$ are symmetric in $(x_1,\dots,x_L)$ (but for generic subsets $\mu$ and $\nu$ they do not exhibit symmetry in the $(y_1,y_2,\dots)$ alphabet). They also satisfy a sum-to-unity property (see Proposition \ref{prop: G stocahsticity}) that is a direct consequence of \eqref{eq:sum-to-unity kernel}.

Given a family of symmetric functions, a general goal is to produce another family which is \textit{dual} to the first. In practice, this means that the original family is orthogonal to the dual one with respect to a certain inner product, or alternatively, the two families should pair together to yield a \textit{Cauchy summation identity}. We are able to solve the latter of these two problems. In particular, for any subsets $\mu,\nu \subset \mathbb{N}$ and a fixed alphabet $(z_1,\dots,z_M)$, we construct a second family of symmetric functions
\begin{align}
\label{F partition function defn B eq-intro}
F_{\mu/\nu} (z_1,\dots,z_M) = \bra{\mu}\bigdot{B}(z_1) \cdots \bigdot{B}(z_M) \ket{\nu}
\end{align}
where $\bigdot{B}(z_i)$ denotes another kind of double-row operator (see Definition \ref{AB row operator defn}). As with the functions \eqref{G partition function defn A eq-intro}, $F_{\mu/\nu} (z_1,\dots,z_M)$ also depends on the secondary alphabet $(y_1,y_2,\dots)$, but we suppress this from our notation. We then observe the following Cauchy identity between the two families \eqref{G partition function defn A eq-intro} and \eqref{F partition function defn B eq-intro}:

\begin{thm}[Theorem \ref{skew Cauchy identity thm} below]
	\label{skew Cauchy identity thm-intro}
	Fix alphabets of complex numbers $(x_1\dots,x_L),(z_1,\dots,z_M)$ and assume there exists $\rho>0$ such that 
	\begin{equation}
		\label{Cauchy identity condition-intro}
		\abs{\frac{1-x_iy_k}{1-qx_iy_k}\frac{q(1-z_j/y_k)}{1-qz_j/y_k}}\leq\rho<1, \hspace{0.5cm} \abs{\frac{1-x_iy_k}{1-qx_iy_k}\frac{1-qz_jy_k}{1-z_jy_k}}\leq\rho<1,
	\end{equation}
	for all $1\leq i\leq L,1\leq j\leq M$ and $k\in\mathbb{N}$. Then the partition functions \eqref{G partition function defn A eq-intro} and \eqref{F partition function defn B eq-intro} satisfy the skew Cauchy identity
	\begin{multline}
		\label{skew Cauchy identity eq-intro}
		\sum_{\kappa} G_{\kappa/\mu}(x_1,\dots,x_L) F_{\kappa/\nu}(z_1,\dots,z_M) \\  = \prod_{i=1}^M\prod_{j=1}^L\left[\frac{x_j-qz_i}{x_j-z_i}\frac{1-z_ix_j}{1-qz_ix_j}\right]\sum_{\lambda}F_{\mu/\lambda}(z_1,\dots,z_M)G_{\nu/\lambda}(x_1,\dots,x_L).
	\end{multline}
    where the sum on the left is over all finite subsets $\kappa \subset \mathbb{N}$, while the sum on the right is over subsets $\lambda \subset \mathbb{N}$ whose elements do not exceed the maximal element of $\mu$.
\end{thm}

The proof of Theorem \ref{skew Cauchy identity eq-intro} is by application of a non-trivial commutation relation between the operators $A(x_i)$ and $\bigdot{B}(z_j)$; see Proposition \ref{AB commutation relation prop}. As \eqref{skew Cauchy identity eq-intro} holds for any fixed $\mu$ and $\nu$, one may examine specific choices of these subsets such that the right hand side summation simplifies. One such choice is to take $\mu = \varnothing$, when the sum on the right hand side of \eqref{skew Cauchy identity eq-intro} collapses to a single term, yielding the following corollary:

\begin{cor}[Corollary \ref{Cauchy identity cor} below]
	\label{Cauchy identity cor-intro}
	With the same set of assumptions as in Theorem \ref{skew Cauchy identity thm-intro}, one has the summation identity 
	\begin{equation}
		\label{Cauchy identity eq-intro}
		\sum_{\kappa} G_\kappa(x_1,\dots,x_L) F_{\kappa/\nu}(z_1,\dots,z_M) = \prod_{i=1}^M h(z_i) \prod_{i=1}^M\prod_{j=1}^L\left[\frac{x_j-qz_i}{x_j-z_i}\frac{1-z_ix_j}{1-qz_ix_j}\right]G_\nu(x_1,\dots,x_L),
	\end{equation}
 where we have introduced the abbreviation $G_\kappa(x_1,\dots,x_L)=G_{\kappa/\varnothing}(x_1,\dots,x_L)$, and where $h(z_i)$ is given by \eqref{h-intro}.
\end{cor}

Equation \eqref{Cauchy identity eq-intro} is likely to have probabilistic utility. Note that, in view of the stochasticity of $G_\kappa(x_1,\dots,x_L)$, one may view the left hand side of \eqref{Cauchy identity cor-intro} as an expectation value $\mathbb{E}(F_{\kappa/\nu})$. Here $\kappa$ is a random variable and $F_{\kappa/\nu}(z_1,\dots,z_M)$ is a family of observables in which both $M$ and $\nu$ are free. It would be interesting to explore the full range of observables that one may access through such a scheme, and to examine the types of explicit formulas that one obtains for these averages, via the right hand side of \eqref{Cauchy identity eq-intro}; similar approaches have previously been successfully followed in \cite{borodin_higher_2018,borodin2020observables}.

\subsection{A half-space analogue of the domain wall partition function}
\label{intro:pfaff}

Motivated by the quest for further simplifications of the Cauchy identity \eqref{Cauchy identity eq-intro}, we were led to consider the partition function $G_{\nu/\mu}(x_1,\dots,x_L)$ in which both $\mu = \nu = \varnothing$; we denote this quantity $G_{\varnothing}(x_1,\dots,x_L)$ in the sequel. In this situation, no paths enter or exit via the external vertical edges of the lattice \eqref{G partition function defn picture-intro}. At first glance, it may seem that this renders $G_{\varnothing}(x_1,\dots,x_L)$ trivial; this turns out not to be the case, since paths may be created/destroyed by the boundary vertices, and may thus trace out non-trivial configurations within the bulk of the lattice. Nevertheless, we do find an unexpected simplification of the function $G_{\varnothing}(x_1,\dots,x_L)$ that brings us into the realm of the six-vertex model in the \textit{half-quadrant}.

\begin{thm}[Theorem \ref{G=Z empty set thm} below]
	\label{G=Z empty set thm-intro}
Fix an integer $L$ and an alphabet $(x_1,\dots,x_L)$ of complex parameters. Introduce the following \textit{triangular partition function}:
\begin{equation}
		\label{Z triangular partition function defn eq-intro}
		Z_L(x_1\dots,x_L) = 
		\begin{tikzpicture}[baseline={([yshift=-.5ex]current bounding box.center)},scale=1]
			\draw[lightgray,line width=1.5pt,->] (1,0) -- (1,1) -- (6,1);
			\draw[lightgray,line width=1.5pt,->] (2,0) -- (2,2) -- (6,2);
			\draw[lightgray,line width=1.5pt,->] (3,0) -- (3,3) -- (6,3);
			\draw[lightgray,line width=1.5pt,->] (4,0) -- (4,4) -- (6,4);
			\draw[lightgray,line width=1.5pt,->] (5,0) -- (5,5) -- (6,5);
			
			\draw[lightgray,line width=1.5pt,->] (1,0) -- (1,0.5);
			\draw[lightgray,line width=1.5pt,->] (2,0) -- (2,0.5);
			\draw[lightgray,line width=1.5pt,->] (3,0) -- (3,0.5);
			\draw[lightgray,line width=1.5pt,->] (4,0) -- (4,0.5);
			\draw[lightgray,line width=1.5pt,->] (5,0) -- (5,0.5);
			
			\draw[blue,fill=blue] (5,5) circle (0.1cm);
			\draw[blue,fill=blue] (4,4) circle (0.1cm);
			\draw[blue,fill=blue] (3,3) circle (0.1cm);
			\draw[blue,fill=blue] (2,2) circle (0.1cm);
			\draw[blue,fill=blue] (1,1) circle (0.1cm);
			
			\node[right] at (6,1) {$0 \rightarrow x_1^{-1}$};
			\node[right] at (6,2) {$0 \rightarrow x_2^{-1}$};
			\node[right] at (6,3) {$0$};
			\node[right] at (6,4) {$0$};
			\node[right] at (6,5) {$0 \rightarrow x_L^{-1}$};
			
			\node[right] at (6.9,3) {$\vdots$};
			\node[right] at (6.9,4) {$\vdots$};
			
			\node[below] at (1,-0.4) {$\uparrow$};
			\node[below] at (2,-0.4) {$\uparrow$};
			\node[below] at (5,-0.4) {$\uparrow$};
			
			\node[below] at (1,-0.9) {$x_1$};
			\node[below] at (2,-0.9) {$x_2$};
			\node[below] at (3,-0.9) {$\cdots$};
			\node[below] at (4,-0.9) {$\cdots$};
			\node[below] at (5,-0.9) {$x_L$};
			
			\node[below] at (1,0) {$0$};
			\node[below] at (2,0) {$0$};
			\node[below] at (3,0) {$0$};
			\node[below] at (4,0) {$0$};
			\node[below] at (5,0) {$0$};
		\end{tikzpicture}
	\end{equation}
 where the vertex in the $i$-th column (counted from the left) and $j$-th row (counted from the bottom) is given by \eqref{Stochastic 6VM weights table-intro} with $z \mapsto x_i x_j$, and the boundary vertices are given by \eqref{Stochastic boundary K-weight table-intro}. One then has that
	\begin{equation}
		G_{\emptyset} (x_1,\dots,x_L) = Z_L(x_1,\dots,x_L).
	\end{equation}
    In particular, $G_\emptyset(x_1,\dots,x_L)$ does not depend on the collection of vertical spectral parameters $Y$ that are implicit in its definition \eqref{G partition function defn picture-intro}.
\end{thm}

Theorem \ref{G=Z empty set thm-intro} is proved by a non-obvious application of the Yang--Baxter equation, and initially came as a complete surprise to us. There is also a version of it that applies to the more general object $G_{\nu}(x_1,\dots,x_L)$; see Lemma \ref{G nu picture lemma} in the text. One of the main points of interest in these results is that they connect our probability measures on the half-line to six-vertex measures in the half-quadrant. The latter have been quite topical in recent years; see, for example, \cite{barraquand_stochastic_2018,he_boundary_2023}.

Another point of interest is that \eqref{Z triangular partition function defn eq-intro} provides a natural two-parameter generalization of the partition function of \textit{off diagonally symmetric alternating-sign matrices} (OSASMs), as introduced by Kuperberg in \cite{kuperberg_symmetry_2002}. In particular, \cite{kuperberg_symmetry_2002} gave a Pfaffian evaluation of the partition function \eqref{Z triangular partition function defn eq-intro}, in a limit where only the middle two boundary vertices in \eqref{Stochastic boundary K-weight table-intro} survive. One result of the current text is that this Pfaffian structure is preserved in the presence of four non-trivial boundary vertices\footnote{After completion of this work we became aware that an analogous result was very recently also presented in \cite{behrend2023diagonally}; we thank Roger Behrend for bringing this to our attention.}: 

 \begin{thm}[Theorem \ref{thm:ZPfaff} below]
\label{thm:ZPfaff-intro}
    When $L$ is even, the triangular partition function \eqref{Z triangular partition function defn eq-intro} admits the Pfaffian formula
    \begin{equation}
        \label{Z single Pfaffian them eq-intro}
        Z_L(x_1,\dots,x_L) = \prod_{1\leq i<j\leq L} \frac{1-x_ix_j}{x_i-x_j} \cdot \pf \left(\frac{x_i-x_j}{1-x_ix_j} Q(x_i,x_j) \right)_{1\leq i,j\leq L}
    \end{equation}
    where $Q$ is a symmetric function in two variables given by
    \begin{equation}
        \label{eq: Q Pfaffian kernel-intro}
        Q(x_i,x_j) = (1-h(x_i))(1-h(x_j)) - \frac{h(x_i)h(x_j)}{ac} \frac{(1-q)x_ix_j}{1-qx_ix_j}.  
    \end{equation}
\end{thm}
A Pfaffian formula is also available in the case where $L$ is odd, but we omit it here as its presentation is less elegant. The reduction of \eqref{Z single Pfaffian them eq-intro} to the OSASM partition function of \cite{kuperberg_symmetry_2002} is obtained by taking $a \to 1$, $c \to -1$; note that in this limit one has that $h(x) \to 1$.

The proof of Theorem \ref{thm:ZPfaff-intro} proceeds by establishing a set of conditions that uniquely determine $Z_L(x_1,\dots,x_L)$ by Lagrange interpolation, and then showing that the right hand side of \eqref{Z single Pfaffian them eq-intro} satisfies these conditions. For the verification step, we show that $Z_L(x_1,\dots,x_L)$ may be recovered from the \textit{shuffle exponential} of a certain linear combination of the partition functions $Z_1$ and $Z_2$, assuming a \textit{shuffle product} that we define; see Section \ref{ssec:shuffle}. The connection with shuffle algebras is non-essential for our proof, but expedites the verification of our recursion relations significantly.

Combining Theorem \ref{thm:ZPfaff-intro} with Corollary \ref{Cauchy identity cor-intro} in the case $\nu = \varnothing$, we then have the following result:
\begin{cor}[Corollary \ref{cor:cauchy-pfaff} below]
\label{cor:cauchy-pfaff-intro}
    With the same set of assumptions as in Theorem \ref{skew Cauchy identity thm-intro}, one has the summation identity
    \begin{multline}
    \label{pfaffian-average}
        \sum_{\kappa} G_\kappa(x_1,\dots,x_L) F_{\kappa}(z_1,\dots,z_M) = \prod_{i=1}^M h(z_i) \prod_{i=1}^M\prod_{j=1}^L\left[\frac{x_j-qz_i}{x_j-z_i}\frac{1-z_ix_j}{1-qz_ix_j}\right]\\
       \times \prod_{1\leq i<j\leq L} \frac{1-x_ix_j}{x_i-x_j} \cdot \pf \left(\frac{x_i-x_j}{1-x_ix_j} Q(x_i,x_j) \right)_{1\leq i,j\leq L}
    \end{multline} 
where $F_{\kappa}(z_1,\dots,z_M) = F_{\kappa/\varnothing}(z_1,\dots,z_M)$ and $Q$ is given by \eqref{eq: Q Pfaffian kernel-intro}.
\end{cor}
Corollary \ref{cor:cauchy-pfaff-intro} has a very similar flavour to the refined Cauchy and Littlewood identities obtained in \cite{betea_refined_2015,wheeler_refined_2016}, in which infinite summations of (products of) symmetric functions are evaluated as partition functions in the six-vertex model. As we have already mentioned, the left hand side of \eqref{pfaffian-average} admits the interpretation of an average; the right hand side of \eqref{pfaffian-average} provides a Pfaffian evaluation of this average, which is likely to be valuable for asymptotic purposes. A similar scheme, in the full-space setting, was recently elaborated in \cite[Appendix C]{aggarwal2023deformed}. 

\subsection{Integral formula for transition probabilities}
\label{intro:integral}

Our next result concerns the evaluation of $G_{\nu}(x_1,\dots,x_L)$, where $\nu$ is non-empty. In this more general setting we lose the Pfaffian structure of $G_{\varnothing}(x_1,\dots,x_L)$, but it still turns out to be possible to obtain a compact multiple-integral formula for the object in question.

\begin{thm}[Theorem \ref{thm: G nu inetgral formula} below]
    \label{thm: G nu inetgral formula-intro}
    In the special case $\mu=\varnothing$, $\nu = \{\nu_1 > \cdots > \nu_n \geq 1\}$, the partition function \eqref{G partition function defn picture-intro}
	can be expressed as the following $n$-fold integral:  
	\begin{multline}
        \label{eq: G nu integral forumla thm-intro}
	    G_\nu(x_1,\dots,x_L) = \oint_{\mathcal{C}_1}\frac{\dd w_1}{2\pi\ii} \cdots \oint_{\mathcal{C}_n}\frac{\dd w_n}{2\pi\ii} Z_{L+n} \left(x_1,\dots,x_L,w_1^{-1},\dots,w_n^{-1}\right) \\
        \times\prod_{i=1}^n\prod_{j=1}^L \left[\frac{q w_i-x_j}{w_i-x_j}\frac{1-w_ix_j}{1-qw_ix_j}\right] \prod_{1\leq i<j\leq n} \left[\frac{w_j-w_i}{qw_j-w_i}\frac{1-qw_iw_j}{1- w_iw_j}\right] \\ \times \prod_{i=1}^n \left[\frac{ac\left(q w_i^2-1\right)}{(w_i-a)(w_i-c)}\frac{y_{\nu_i}}{1-q w_i y_{\nu_i}}\prod_{j-1}^{\nu_i-1} \frac{1-w_iy_j}{1-qw_iy_j}\right],
	\end{multline}
 where $Z_{L+n}$ denotes the partition function \eqref{Z triangular partition function defn eq-intro} in $(L+n)$ variables and $\mathcal{C}_1,\dots,\mathcal{C}_n$ are a certain family of $q$-nested contours that surround the points $(x_1,\dots,x_L)$ (for the precise definition of these contours, see Definition \ref{defn: nested contours} and Figure \ref{fig:contour diagram}).
\end{thm}

Originally, we constructed a sum-over-subsets formula for $G_{\nu}(x_1,\dots,x_L)$ (see Theorem \ref{G nu thm} below) by the application of \textit{Drinfeld twists} to the columns of the partition function \eqref{G partition function defn picture-intro}, similarly to what was done in \cite{wheeler_refined_2016}. Once obtained, it was then straightforward to match this sum-over-subsets formula to a sum-over-residues produced by evaluating the integral \eqref{eq: G nu integral forumla thm-intro}. While this is a totally direct method, the technical details are quite unwieldy; for this reason, we present instead a simpler (though non-constructive) verification argument, that once again relies on checking a set of recursive properties that uniquely determine $G_{\nu}(x_1,\dots,x_L)$.

The remainder of the text is then devoted to studying the special case $c \to \infty$ of \eqref{eq: G nu integral forumla thm-intro}, when the integrand becomes fully factorized. While this loses some of the generality of boundary vertex weights (under this limit, the third vertex in \eqref{Stochastic boundary K-weight table-intro} vanishes) it has the advantage of making our formulas much more tractable without sacrificing the non-trivial injection of paths into the lattice \eqref{G partition function defn picture-intro} that still occurs in this regime.

\subsection{An orthogonality conjecture}
\label{intro:orthog}

Orthogonality with respect to an integral scalar product was a key feature of the multivariate rational functions studied in \cite{borodin_higher_2018,borodin_family_2017,borodin_coloured_2022}. It seems likely that the symmetric functions considered in the current text also have nice orthogonality properties, and we find the first evidence of this in the following conjecture.
\begin{conj}[Conjecture \ref{conj:F-orth} below]
\label{conj:F-orth-intro}
Fix a finite subset $\nu = \{\nu_1 > \cdots > \nu_n \geq 1\}$, and a second finite subset $\kappa$ whose cardinality satisfies $|\kappa| \leq n$. Then in the limit $c \to \infty$, one has that
\begin{multline}
\label{F-int-orthog}
\oint_{\mathcal{C}}\frac{\dd w_1}{2\pi\ii} 
\cdots 
\oint_{\mathcal{C}}\frac{\dd w_n}{2\pi\ii}
\prod_{1\leq i<j\leq n}
\left[\frac{w_j-w_i}{qw_j-w_i}\frac{1-qw_iw_j}{1- w_iw_j}\right]
\\
\times
\prod_{i=1}^{n}
\left[
\frac{w_i-a}{w_i(1-a w_i)}
\frac{1-q w_i^2}{1-w_i^2    }
\frac{y_{\nu_i}}{1-q w_i y_{\nu_i}}\prod_{j-1}^{\nu_i-1} \frac{1-w_iy_j}{1-qw_iy_j}
\right]
F_{\kappa}(w_1,\dots,w_n)
=
\delta_{\kappa,\nu},
\end{multline}
where the contour $\mathcal{C}$ is a small, positively oriented circle surrounding the points $y_j^{-1}$, $j \geq 1$ and no other singularities of the integrand, such that $q\cdot \mathcal{C}$ is disjoint from the interior of $\mathcal{C}$.
\end{conj}
This conjecture has been tested extensively on examples with $n \leq 3$. It was motivated by considering the integral \eqref{eq: G nu integral forumla thm-intro} at $c \to \infty$. In that limit, one may show that the triangular partition function \eqref{Z triangular partition function defn eq-intro} factorizes:
\begin{align*}
\lim_{c \rightarrow \infty}
Z_L(x_1,\dots,x_L)
=
\lim_{c \rightarrow \infty}
\prod_{i=1}^L
(1-h(x_i))
=
\prod_{i=1}^L
\frac{x_i(1-ax_i)}{x_i-a}.
\end{align*}
This has two implications. On the one hand, the right side of the Cauchy identity \eqref{pfaffian-average} factorizes at this point; on the other hand, the integrand of \eqref{eq: G nu integral forumla thm-intro} is itself factorized and one may recognize the Cauchy kernel embedded within it. Expanding the Cauchy kernel via the left side of \eqref{pfaffian-average} and collecting coefficients of $G_{\kappa}(x_1,\dots,x_L)$ on both sides of the equation, one deduces \eqref{F-int-orthog}. We note that Conjecture \ref{conj:F-orth-intro} is true assuming the linear independence of the functions $G_{\kappa}(x_1,\dots,x_L)$, but the proof of such a statement is outside of the scope of the present work.

\subsection{Reduction to the ASEP on the half-line}
\label{intro:ASEP}

It is well known that the stochastic six-vertex model converges, in a certain continuous-time limit, to the asymmetric simple exclusion process (ASEP); for a rigorous proof of this convergence in the full-space setting see, for example, \cite{aggarwal_convergence_2016}. Our final result is an explicit integral formula for transition probabilities of the ASEP on the half-line, obtained by taking an appropriate degeneration of the integral expression \eqref{eq: G nu integral forumla thm-intro}.

In what follows, we consider the ASEP on the half-line with bulk hopping rates $q$ and $1$ associated to left and right moves, respectively, and incoming/outgoing hopping rates parameterized as
\begin{align*}
    \alpha = \frac{a c(1-q)}{(1-a)(1-c)}, \qquad \gamma = - \frac{1-q}{(1-a)(1-c)},
\end{align*}
respectively. For more details on the definition of this system, see Section \ref{ASEP construction section}.

\begin{thm}
Let $\mathbb{P}_t(\emptyset\to\nu)$ denote the probability that the ASEP on the half-line has particles at positions $\nu \subset \mathbb{N}$ at time $t$, given that it is initially empty. Under the limit $c\to\infty$ ($\gamma\to0$), and assuming that $\alpha+q\neq 1$, one has that
    \begin{multline}
    \label{half-line-trans}
    \lim_{\gamma \to 0}
	    \mathbb{P}_t(\emptyset\to\nu) = \alpha^n\mathrm{e}^{-\alpha t}\oint_{\mathcal{D}_1}\frac{\dd w_1}{2\pi\ii} \cdots \oint_{\mathcal{D}_n}\frac{\dd w_n}{2\pi\ii} \prod_{1\leq i<j\leq n} \left[\frac{w_j-w_i}{qw_j-w_i}\frac{1-qw_iw_j}{1- w_iw_j}\right] \\
        \times\prod_{i=1}^n \left[\frac{1-q w_i^2}{w_i(q+\alpha-1-\alpha w_i)(1-qw_i)} \left(\frac{1-w_i}{1-qw_i}\right)^{\nu_i-1} \exp(\frac{(1-q)^2 w_i t}{(1-w_i)(1-qw_i)})\right],
    \end{multline}
     where the contours $\mathcal{D}_1,\dots,\mathcal{D}_n$ surround the point $1$ and satisfy the conditions of Definition \ref{defn: ASEP contours}.
\end{thm}
In the case of closed (or hard-wall) boundary conditions, integral formulas for the half-line ASEP were previously obtained in \cite{tracy_asymmetric_2013}. Our formula \eqref{half-line-trans} appears to be new, however, as we deal with more generic boundary hopping rates than were considered in \cite{tracy_asymmetric_2013}. In particular, \eqref{half-line-trans} allows the non-trivial injection of particles into the half-line (indeed, this is necessary to make this probability non-zero), whereas in \cite{tracy_asymmetric_2013} no particle may ever enter or exit the system.

\subsection{Future prospects}
\label{intro:future}

A number of natural directions are suggested by the current work, and we plan to pursue at least a few of the following topics in later texts:

\begin{itemize}
\item The stochastic six-vertex model has a \textit{higher-spin} analogue \cite{borodin_higher_2018,borodin_family_2017,corwin2016stochastic} that may be obtained via the \textit{fusion procedure} of \cite{kulish1981yang}. In this more general model, vertical lattice lines may now accommodate any number of paths, while horizontal lines still permit at most one path. 

It is quite straightforward to push many of the formulas in the current text through the machinery of fusion, yielding a theory of higher-spin double-row symmetric functions. We expect that these functions will have many interesting properties, including simpler Cauchy identities and links with known families of functions, such as $BC$-symmetric Hall--Littlewood polynomials.

\item In another direction, the stochastic six-vertex model may also be generalized to ensembles of \textit{coloured paths}, following \cite{borodin_coloured_2022}. Algebraically, this corresponds to lifting the underlying quantum group $U_q\left(\widehat{\mathfrak{sl}}_2\right)$ to $U_q\left(\widehat{\mathfrak{sl}}_{n+1}\right)$, with $n \geq 1$. We expect that, in this coloured lattice model, our double-row operators will satisfy non-trivial commutation relations and that one should observe a family of \textit{non-symmetric} multivariate rational functions that transform nicely under the action of the \textit{Hecke algebra}. Evidence of such a construction, albeit in a slightly different model, has already been obtained in \cite{zhong2022stochastic}.

\item It is interesting to further explore the combinatorial implications of the Pfaffian formula \eqref{Z single Pfaffian them eq-intro}. In view of the generic boundary vertex weights \eqref{Stochastic boundary K-weight table-intro}, the summation set of the partition function \eqref{Z triangular partition function defn eq-intro} is equivalent to \textit{diagonally symmetric} alternating-sign matrices (with a generic diagonal), whose enumeration was unknown for twenty years since Kuperberg's work \cite{kuperberg_symmetry_2002} but very recently resolved in \cite{behrend2023diagonally}.
\end{itemize}

\subsection{Acknowledgments}
We warmly thank Roger Behrend, Vadim Gorin, Leonid Petrov, Jeremy Quastel, Travis Scrimshaw and Ole Warnaar for helpful discussions. The authors gratefully acknowledge financial support of the Australian Research Council. AG was partially supported by the ARC DECRA DE210101264. WM was supported by an Australian Government Research Training Program Scholarship. MW was partially supported by the ARC Future Fellowship FT200100981.

\section{Vertex model with boundaries}
\label{Vertex model definition section}
In this section we will outline a stochastic vertex model which will be the focus of much of this text. Explicit constructions of multi-parameter symmetric functions will be provided using these vertex models on the square lattice with non-trivial boundaries. These constructions have generic boundary parameters which will allow for a relation to the half-line open ASEP in Section \ref{ASEP construction section}.

\subsection{Bulk vertex weights}
Here we define the weights and relations of the vertex models that will be used throughout this work. First, we provide a definition of the stochastic six-vertex model \cite{gwa1992bethe} and its $U_q\left(\widehat{\mathfrak{sl}}_2\right)$ $R$-matrix. We will follow the conventions of \cite{borodin_coloured_2022}.
\begin{defn}
	\label{R-matrix defn}
	The \emph{stochastic six-vertex model} is an assignment of a rational function to a vertex 
	\begin{align}
        \label{R-matrix defn pic}
		 R_{y/x}(i,j;k,l) & =
		\begin{tikzpicture}[baseline={([yshift=-.5ex]current bounding box.center)}]
			\draw[lightgray,line width=1.5pt,->] (-1,0) -- (1,0);
			\draw[lightgray,line width=1.5pt,->] (0,-1) -- (0,1);
			\node[below] at (0,-1) {$i$};
			\node[left] at (-1,0) {$x\to j$};
			\node[above] at (0,1) {$k$};
			\node[right] at (1,0) {$l$};
			\node[below] at (0,-1.4) {$\uparrow$};
			\node[below] at (0,-1.9) {$y$};
		\end{tikzpicture};
    \hspace{0.5cm} i,j,k,l\in\{0,1\}.
	\end{align} 
	At any given edge a 1 indicates the presence of a path while a 0 indicates an absence thereof. The values of the weights \eqref{R-matrix defn pic} are given pictorially in the following table where $z = y/x$. Any weight which does not appear in the table is defined to be equal to zero.
	\begin{align}
		\label{Stochastic 6VM weights table} 
		\begin{tabular}{ccc}
			\qquad
			\begin{tikzpicture}
				\draw[gray,dashed,line width=1pt,-] (-1,0) -- (1,0);
				\draw[gray,dashed,line width=1pt,-] (0,-1) -- (0,1);
			\end{tikzpicture}
			\qquad
			&
			\qquad
			\begin{tikzpicture}
				\draw[gray,dashed,line width=1pt,-] (-1,0) -- (1,0);
				\draw[red,line width=2pt,->] (0,-1) -- (0,1);
			\end{tikzpicture}
			\qquad
			&
			\qquad
			\begin{tikzpicture}
				\draw[gray,dashed,line width=1pt,-] (0,1) -- (0,0) -- (-1,0);
				\draw[red,line width=2pt,->] (0,-1) -- (0,0) -- (1,0);
			\end{tikzpicture}
			\qquad
			\\
			\qquad
			$1$
			\qquad
			& 
			\qquad
			$\dfrac{q(1-z)}{1-qz}$
			\qquad
			& 
			\qquad
			$\dfrac{1-q}{1-qz}$
			\qquad
			\\
			\vspace{0.1cm}
			\\
			\qquad
			\begin{tikzpicture}
				\draw[gray,dashed,line width=1pt,-] (-1,0) -- (1,0);
				\draw[gray,dashed,line width=1pt,-] (0,-1) -- (0,1);
				\draw[red,line width=2pt,->] (-1,0) -- (-0.2,0) -- (0,0.2) -- (0,1);
				\draw[red,line width=2pt,<-] (1,0) -- (0.2,0) -- (0,-0.2) -- (0,-1);
			\end{tikzpicture}
			\qquad
			&
			\qquad
			\begin{tikzpicture}
				\draw[red,line width=2pt,->] (-1,0) -- (1,0);
				\draw[gray,dashed,line width=1pt,-] (0,-1) -- (0,1);
			\end{tikzpicture}
			\qquad
			&
			\qquad
			\begin{tikzpicture}
				\draw[gray,dashed,line width=1pt,-] (0,-1) -- (0,0) -- (1,0);
				\draw[red,line width=2pt,->] (-1,0) -- (0,0) -- (0,1);
			\end{tikzpicture}
			\qquad
			\\
			\qquad
			$1$
			\qquad
			& 
			\qquad
			$\dfrac{1-z}{1-qz}$
			\qquad
			&
			\qquad
			$\dfrac{z(1-q)}{1-qz}$
			\qquad
			\\
		\end{tabular} 
	\end{align}	
\end{defn}

\begin{prop}[Stochasticity]
\label{prop:stoch}
    \label{prop: stochasticity bulk}
    The weights of the stochastic six-vertex model satisfy a sum to unity property
    \begin{equation}
        \sum_{k,l\in\{0,1\}}  R_{z}(i,j;k,l) = 1,
    \end{equation}
    for any fixed $i,j\in\{0,1\}$.
\end{prop}
\begin{defn}
    The \emph{$R$-matrix} of the stochastic six-vertex model is $R_{12}(z)\in\mathrm{End}(V_1\otimes V_2)$ for $V_1,V_2\cong \mathbb{C}^2$ given by
\begin{equation}
    \label{R-matrix form}
    R_{12}(z) = \sum_{i,j,k,l\in\{0,1\}} R_{z}(i,j;k,l) E_1^{(j,l)} \otimes E_2^{(i,k)}  ,
\end{equation}
where $E_m^{(i,j)}\in\mathrm{End}(V_m)$ is the $2\times2$ elementary matrix with a 1 in row $i$ and column $j$.
\end{defn}
We will be interested in $R$-matrices acting on (possibly infinite-dimensional) spaces $V_1\otimes V_2\otimes \cdots$ where each $V_i\cong \mathbb{C}^2$. We will denote by $R_{jk}$ the $R$-matrix which acts non-trivially in the space $V_j$ aligned horizontally and $V_k$ aligned vertically as in the picture \eqref{R-matrix defn pic}.
\begin{prop}[Factorization]
	\label{lm:Rfactor}
	At the special value $z=1$, the $R$-matrix \eqref{R-matrix form} satisfies the relation $ R_{z=1}(i,j;k,l) = \delta_{i,l}\delta_{j,k}$. This has pictoral representation
	\be
	\label{eq:Rfactor}
	 R_{x/x}(i,j;k,l) =
	\begin{tikzpicture}[baseline={([yshift=-.5ex]current bounding box.center)}]
		\draw[lightgray,line width=1.5pt,->] (-1,0) -- (1,0);
		\draw[lightgray,line width=1.5pt,->] (0,-1) -- (0,1);
		\node[below] at (0,-1) {$i$};
		\node[left] at (-1,0) {$x\to j$};
		\node[above] at (0,1) {$k$};
		\node[right] at (1,0) {$l$};
		\node[below] at (0,-1.4) {$\uparrow$};
		\node[below] at (0,-1.9) {$x$};
	\end{tikzpicture}
	=
	\begin{tikzpicture}[baseline={([yshift=-.5ex]current bounding box.center)}]
		\draw[lightgray,line width=1.5pt,->] (-1,0) arc (-90:0:1);
		\draw[lightgray,line width=1.5pt,->] (0,-1) arc (180:90:1);
		\node[below] at (0,-1) {$i$};
		\node[left] at (-1,0) {$x\to j$};
		\node[above] at (0,1) {$k$};
		\node[right] at (1,0) {$l$};
		\node[below] at (0,-1.4) {$\uparrow$};
		\node[below] at (0,-1.9) {$x$};
	\end{tikzpicture}.
	\ee
\end{prop}

\begin{prop}[Yang\textendash Baxter equation]
\label{prop:YB}
	The $R$-matrix \eqref{R-matrix form} satisfies the \emph{Yang\textendash Baxter equation}
	\begin{equation}
		\label{Yang-Baxter eq R-matrices}
		R_{12}(y/x) R_{13}(z/x) R_{23} (z/y) = R_{23}(z/y) R_{13}(z/x) R_{12}(y/x).
	\end{equation}
	For fixed $i_1,i_2,i_3,j_1,j_2,j_3\in\{0,1\}$, this is represented pictorially as
	\[
    \sum_{k_1,k_2,k_3 \in \{0,1\}} 
	\begin{tikzpicture}[baseline={([yshift=-.5ex]current bounding box.center)}]
		\draw[lightgray,line width=1.5pt,->] (0,1) -- (1,0) -- (3,0);
		\draw[lightgray,line width=1.5pt,->] (0,0) -- (1,1) -- (3,1);
		\draw[lightgray,line width=1.5pt,->] (2,-0.5) -- (2,1.5);
		\node[left] at (0,1) {$x\rightarrow i_1$};
		\node[left] at (0,0) {$y\rightarrow i_2$};
		\node[below] at (2,-0.5) {$i_3$};
		\node[right] at (3,0) {$j_1$};
		\node[right] at (3,1) {$j_2$};
		\node[above] at (2,1.5) {$j_3$};
		\node[below] at (1,0) {$k_1$};
		\node[above] at (1,1) {$k_2$};
		\node[right] at (2,0.5) {$k_3$};
		\node[below] at (2,-0.9) {$\uparrow$};
		\node[below] at (2,-1.4) {$z$};
	\end{tikzpicture}
    =
	\sum_{k_1,k_2,k_3 \in \{0,1\}} 
	\begin{tikzpicture}[baseline={([yshift=-.5ex]current bounding box.center)}]
		\draw[lightgray,line width=1.5pt,->] (0,1) -- (2,1) -- (3,0);
		\draw[lightgray,line width=1.5pt,->] (0,0) -- (2,0) -- (3,1);
		\draw[lightgray,line width=1.5pt,->] (1,-0.5) -- (1,1.5);
		\node[left] at (0,1) {$x\rightarrow i_1$};
		\node[left] at (0,0) {$y\rightarrow i_2$};
		\node[below] at (1,-0.5) {$i_3$};
		\node[right] at (3,0) {$j_1$};
		\node[right] at (3,1) {$j_2$};
		\node[above] at (1,1.5) {$j_3$};
		\node[above] at (2,1) {$k_1$};
		\node[below] at (2,0) {$k_2$};
		\node[right] at (1,0.5) {$k_3$};
		\node[below] at (1,-0.9) {$\uparrow$};
		\node[below] at (1,-1.4) {$z$};
	\end{tikzpicture}.	
	\]
\end{prop}
\begin{prop}[$R$-matrix unitarity]
	\label{R-matrix unitrarity prop}
	The $R$-matrix \eqref{R-matrix form} satisfies the unitarity condition
	\begin{equation}
		\label{R-matrix unitrarity prop eq}
		R_{21}\left(x/y\right) R_{12}\left(y/x\right) = \mathsf{id},
	\end{equation}
        where $\mathsf{id}$ is the identity within $\mathrm{End}(V_1\otimes V_2)$. For fixed $i_1,i_2,j_1,j_2\in\{0,1\}$, this is represented pictorially as
	\[
	\sum_{k_1,k_2\in\{0,1\}}
	\begin{tikzpicture}[baseline={([yshift=-.5ex]current bounding box.center)},scale=0.8]
		\draw[lightgray,line width=1.5pt,-] (1,0) -- (1,1.7);
		\draw[lightgray,line width=1.5pt,-] (1,1.7) arc (180:90:0.3);
		\draw[lightgray,line width=1.5pt,->] (1.3,2) -- (3,2);
		\draw[lightgray,line width=1.5pt,-] (0,1) -- (1.7,1);
		\draw[lightgray,line width=1.5pt,-] (1.7,1) arc (-90:0:0.3);
		\draw[lightgray,line width=1.5pt,->] (2,1.3) -- (2,3);
		\node[below] at (1,0) {$i_1$};
		\node[left] at (1,2) {$k_1$};
		\node[right] at (3,2) {$j_1$};
		\node[left] at (0,1) {$y \rightarrow i_2$};
		\node[right] at (2,1) {$k_2$};
		\node[above] at (2,3) {$j_2$};
		\node[below] at (1,-0.6) {$\uparrow$};
		\node[below] at (1,-1.2) {$x$};
	\end{tikzpicture} =	
	\begin{tikzpicture}[baseline={([yshift=-.5ex]current bounding box.center)},scale=0.8]
		\draw[lightgray,line width=1.5pt,->] (0,1) arc (-60:-30:5.196);
		\draw[lightgray,line width=1.5pt,->] (1,0) arc (150:120:5.196);
		\node[below] at (1,0) {$i_1$};
		\node[right] at (3,2) {$j_1$};
		\node[left] at (0,1) {$y \rightarrow i_2$};
		\node[above] at (2,3) {$j_2$};
		\node[below] at (1,-0.6) {$\uparrow$};
		\node[below] at (1,-1.2) {$x$};
	\end{tikzpicture} =
	\delta_{i_1,j_1} \delta_{i_2,j_2}.
	\] 
\end{prop}
We will also define a re-normalized version of the weights from Definition \ref{R-matrix defn}. These new weights inherit many of the algebraic properties of the stochastic ones. 
\begin{defn}
	A \emph{re-normalized vertex} is represented pictorially
		\begin{align}
		\bigdot{R}_{y/x}(i,j;k,l) & =
		\begin{tikzpicture}[baseline={([yshift=-.5ex]current bounding box.center)}]
			\draw[lightgray,line width=1.5pt,->] (-1,0) -- (1,0);
			\draw[lightgray,line width=1.5pt,->] (0,-1) -- (0,1);
			\node[below] at (0,-1) {$i$};
			\node[left] at (-1,0) {$x\to j$};
			\node[above] at (0,1) {$k$};
			\node[right] at (1,0) {$l$};
			\node[below] at (0,-1.4) {$\uparrow$};
			\node[below] at (0,-1.9) {$y$};
			\draw[black,fill=black] (0,0) circle (0.12cm);
		\end{tikzpicture};
    \hspace{0.5cm} i,j,k,l\in\{0,1\}.
	\end{align} 
	The values of these weights are given in the following table, where $z = y/x$. Any weight which does not appear in the table is defined to be equal to zero.
	\begin{align}
		\begin{tabular}{ccc}
			\qquad
			\begin{tikzpicture}
				\draw[gray,dashed,line width=1pt,-] (-1,0) -- (1,0);
				\draw[gray,dashed,line width=1pt,-] (0,-1) -- (0,1);
				\draw[black,fill=black] (0,0) circle (0.1cm);
			\end{tikzpicture}
			\qquad
			&
			\qquad
			\begin{tikzpicture}
				\draw[gray,dashed,line width=1pt,-] (-1,0) -- (1,0);
				\draw[red,line width=2pt,->] (0,-1) -- (0,1);
				\draw[black,fill=black] (0,0) circle (0.1cm);
			\end{tikzpicture}
			\qquad
			&
			\qquad
			\begin{tikzpicture}
				\draw[gray,dashed,line width=1pt,-] (0,1) -- (0,0) -- (-1,0);
				\draw[red,line width=2pt,->] (0,-1) -- (0,0) -- (1,0);
				\draw[black,fill=black] (0,0) circle (0.1cm);
			\end{tikzpicture}
			\qquad
			\\
			\qquad
			$\dfrac{1-qz}{1-z}$
			\qquad
			& 
			\qquad
			$q$
			\qquad
			& 
			\qquad
			$\dfrac{1-q}{1-z}$
			\qquad
			\\
			\vspace{0.1cm}
			\\
			\qquad
			\begin{tikzpicture}
				\draw[gray,dashed,line width=1pt,-] (-1,0) -- (1,0);
				\draw[gray,dashed,line width=1pt,-] (0,-1) -- (0,1);
				\draw[red,line width=2pt,->] (-1,0) -- (-0.2,0) -- (0,0.2) -- (0,1);
				\draw[red,line width=2pt,<-] (1,0) -- (0.2,0) -- (0,-0.2) -- (0,-1);
				\draw[black,fill=black] (0,0) circle (0.1cm);
			\end{tikzpicture}
			\qquad
			&
			\qquad
			\begin{tikzpicture}
				\draw[red,line width=2pt,->] (-1,0) -- (1,0);
				\draw[gray,dashed,line width=1pt,-] (0,-1) -- (0,1);
				\draw[black,fill=black] (0,0) circle (0.1cm);
			\end{tikzpicture}
			\qquad
			&
			\qquad
			\begin{tikzpicture}
				\draw[gray,dashed,line width=1pt,-] (0,-1) -- (0,0) -- (1,0);
				\draw[red,line width=2pt,->] (-1,0) -- (0,0) -- (0,1);
				\draw[black,fill=black] (0,0) circle (0.1cm);
			\end{tikzpicture}
			\qquad
			\\
			\qquad
			$\dfrac{1-qz}{1-z}$
			\qquad
			& 
			\qquad
			$1$
			\qquad
			&
			\qquad
			$\dfrac{z(1-q)}{1-z}$
			\qquad
			\\
		\end{tabular} 
	\end{align}
\end{defn}
These weights are related to their stochastic counterparts \eqref{R-matrix defn} through the relation 
\begin{equation}
	\label{Dotted R-matrix eq}
	\bigdot{R}_z(i,j;k,l) = \frac{1-qz}{1-z}R_z(i,j;k,l).
\end{equation}
Since the re-normalized vertices differ from the stochastic ones only by an overall multiplicative factor, we can write versions of the Yang\textendash Baxter equation \eqref{Yang-Baxter eq R-matrices} what incorporate mixtures of both $R$ and re-normalized $\bigdot{R}$-matrices.

\subsection{Boundary vertex weights}
We also define weights of a boundary vertex. Such boundary vertices and their matrices were introduced in \cite{cher_84,sklyanin_boundary_1988} and in a more general non-diagonal case in \cite{de_Vega93}. We consider the non-diagonal case in which the boundary vertex weights depend on two free parameters.
\begin{defn}
	\label{K-matrix boundary weight defn}
	A \emph{stochastic boundary weight} is an assignment of a rational function to a half-vertex
	\begin{equation}
        \label{boundary weight defn eq}
		K_x(i;j) = 
		\begin{tikzpicture}[baseline={([yshift=-.5ex]current bounding box.center)}]
			\draw[lightgray,line width=1.5pt,->] (1,0) -- (0,1) -- (1,2);
			\draw[blue,fill=blue] (0,1) circle (0.1cm);
			\node[right] at (1,0) {$i \leftarrow x$};
			\node[right] at (1,2) {$j \rightarrow x^{-1}$};
		\end{tikzpicture};
    \hspace{0.5cm} i,j\in\{0,1\}.
	\end{equation}
	As previously, a 0/1 at a given edge indicates the absence/presence of a path. With the use of the rational function 
	\be
	\label{eq:hdef}
	h(x) = \frac{ac(1-x^2)}{(a-x)(c-x)},
	\ee
	these weights are tabulated below.
	\begin{align}
		\label{Stochastic boundary K-weight table}
		\begin{tabular}{cccc}
			\qquad
			\begin{tikzpicture}[baseline={([yshift=-.5ex]current bounding box.center)}]
				\draw[gray,dashed,line width=1pt,-] (1,0) -- (0,1) -- (1,2);
				\draw[blue,fill=blue] (0,1) circle (0.1cm);
			\end{tikzpicture}
			\qquad
			&
			\qquad
			\begin{tikzpicture}[baseline={([yshift=-.5ex]current bounding box.center)}]
				\draw[gray,dashed,line width=1pt,-] (1,0) -- (0,1);
				\draw[red,line width=2pt,->] (0,1) -- (1,2);
				\draw[blue,fill=blue] (0,1) circle (0.1cm);
			\end{tikzpicture}
			\qquad
			&
			\qquad
			\begin{tikzpicture}[baseline={([yshift=-.5ex]current bounding box.center)}]
				\draw[gray,dashed,line width=1pt,-] (0,1) -- (1,2);
				\draw[red,line width=2pt,-] (1,0) -- (0,1);
				\draw[red,line width=2pt,->] (1,0) -- (0.5,0.5);
				\draw[blue,fill=blue] (0,1) circle (0.1cm);
			\end{tikzpicture}
			\qquad
			&
			\qquad
			\begin{tikzpicture}[baseline={([yshift=-.5ex]current bounding box.center)}]
				\draw[gray,dashed,line width=1pt,-] (0,1) -- (1,2);
				\draw[red,line width=2pt,->] (1,0) -- (0,1) -- (1,2);
				\draw[red,line width=2pt,->] (1,0) -- (0.5,0.5);
				\draw[blue,fill=blue] (0,1) circle (0.1cm);
			\end{tikzpicture}
			\qquad
			\\
			\vspace{0cm}
			\\
			\qquad
			$1-h(x)$
			\qquad
			&
			\qquad
			$h(x)$
			\qquad
			&
			\qquad
			$\dfrac{-h(x)}{ac}$
			\qquad
			&
			\qquad
			$1+\dfrac{h(x)}{ac}$
			\qquad
		\end{tabular}
	\end{align}
\end{defn}
\begin{prop}[Stochasticity]
    \label{prop: stochasticity boundary}
    The stochastic boundary weights also satisfy a sum to unity property
    \begin{equation}
        \sum_{j\in\{0,1\}} K_x(i;j) = 1,
    \end{equation}
    for any fixed $i\in\{0,1\}$.        
\end{prop}
\begin{defn}
    The \emph{$K$-matrix} of the stochastic boundary weights is $K_1(x) \in\mathrm{End}(V_1)$ for $V_1\cong\mathbb{C}^2$ given by
    \begin{equation}
        K_1(x) = \sum_{i,j\in\{0,1\}} K_x(i;j) E_1^{(i,j)} .
    \end{equation}
\end{defn}
Just as with the $R$-matrix, while acting on spaces $V_1\otimes V_2\otimes\cdots$ we will denote by $K_j$ the $K$-matrix which acts non-trivially on the space $V_j$. The boundary vertices \eqref{boundary weight defn eq}, together with the bulk vertices \eqref{R-matrix defn pic}, combine to give the reflection equation. 
\begin{prop}[Sklyanin reflection equation]
	The boundary vertex weights from Definition \ref{K-matrix boundary weight defn} and the bulk weights from Definition \ref{R-matrix defn} satisfy the \emph{Sklyanin reflection equation} 
	\begin{equation}
		\label{Sklyanin reflection equation R,K matrices}
		R_{21} \left(\frac{x}{y}\right) K_1 (x) R_{12}(xy) K_2(y) = K_2(y) R_{21}(xy) K_1 (x) R_{12} \left(\frac{x}{y}\right).
	\end{equation}
	For fixed $i_1,i_2,j_1,j_2\in\{0,1\}$ this can be represented pictorially as
	\[
	\sum_{k_1,k_2,\ell_1,\ell_2 \in\{0,1\}} 
	\begin{tikzpicture}[baseline={([yshift=-.5ex]current bounding box.center)}]
		\draw[lightgray,line width=1.5pt,->] (4,1) -- (3,1) -- (1,0) -- (0,0.5) -- (3,2) -- (4,2);
		\draw[lightgray,line width=1.5pt,->] (4,0) -- (3,0) -- (0,2.5) -- (1,3) -- (4,3);
		\draw[blue,fill=blue] (0,0.5) circle (0.1cm);
		\draw[blue,fill=blue] (0,2.5) circle (0.1cm);
		\node[right] at (4,0) {$i_2\leftarrow y$};
		\node[right] at (4,1) {$i_1\leftarrow x$};
		\node[right] at (4,2) {$j_1\rightarrow x^{-1}$};
		\node[right] at (4,3) {$j_2\rightarrow y^{-1}$};
		\node[below] at (1,0) {$k_1$};
		\node[right] at (1.85,1.15) {$k_2$};
		\node[above] at (0.7,0.82) {$\ell_1$};
		\node[right] at (0.8,2) {$\ell_2$};
	\end{tikzpicture}
	=
	\sum_{k_1,k_2,\ell_1,\ell_2\in\{0,1\}} 
	\begin{tikzpicture}[baseline={([yshift=-.5ex]current bounding box.center)}]
		\draw[lightgray,line width=1.5pt,->] (4,1) -- (3,1) -- (0,2.5) -- (1,3) -- (3,2) -- (4,2);
		\draw[lightgray,line width=1.5pt,->] (4,0) -- (1,0) -- (0,0.5) -- (3,3) -- (4,3);
		\draw[blue,fill=blue] (0,0.5) circle (0.1cm);
		\draw[blue,fill=blue] (0,2.5) circle (0.1cm);
		\node[right] at (4,0) {$i_2\leftarrow y$};
		\node[right] at (4,1) {$i_1\leftarrow x$};
		\node[right] at (4,2) {$j_1\rightarrow x^{-1}$};
		\node[right] at (4,3) {$j_2\rightarrow y^{-1}$};
		\node[above] at (1,3) {$\ell_1$};
		\node[right] at (1.85,1.9) {$\ell_2$};
		\node[below] at (0.75,2.2) {$k_1$};
		\node[right] at (0.8,1) {$k_2$};
	\end{tikzpicture}.
	\]
\end{prop}
This relation is sometimes referred to as the \emph{boundary Yang\textendash Baxter equation}.
\begin{prop}[$K$-matrix unitarity]
	\label{K-matrix unitrarity prop}
	The $K$-matrix from Definition \ref{K-matrix boundary weight defn} satisfies its own unitarity relation 
	\begin{equation}
		K_1(x) K_1\left(x^{-1}\right) = \mathsf{id},
	\end{equation}
	where $\mathsf{id}$ is the identity within $\mathrm{End}(V_1)$. This can be represented pictorially as 
	\[
	\sum_{k\in\{0,1\}}
	\begin{tikzpicture}[baseline={([yshift=-.5ex]current bounding box.center)}]
		\draw[lightgray,line width=1.5pt,->] (1,0) -- (0,1) -- (1,2) -- (2,1);
		\draw[blue,fill=blue] (0,1) circle (0.1cm);
		\draw[blue,fill=blue] (1,2) circle (0.1cm);
		\node[right] at (1,0) {$i \leftarrow x$};
		\node[right] at (2,1) {$j \rightarrow x$};
		\node[left] at (0.55,1.65) {$k$};
	\end{tikzpicture} = \delta_{i,j},
	\] 
    for all $i,j \in \{0,1\}$.
\end{prop}

\section{Row operators and symmetric functions}
\label{sec:row-ops + sym func}
\subsection{Space of states and row-operators}
\label{sec:row-operators}
In this section we construct double-row operators which serve as our transfer matrices. Our operators will act on the vector space with basis elements indexed by configurations in the set 
\[\mathbb{W}=\left\{S\subset \mathbb{N}:\sum_i S_i\text{ is finite}\right\}.\]
We denote a configuration $\mu\in\mathbb{W}$ with $m\geq 0$ parts by $\mu = (\mu_1,\dots,\mu_m)$ where $\mu_m < \cdots < \mu_2 < \mu_1$. By agreement $\mu=\emptyset$ is defined when $m=0$. We will only consider finite configurations $\mu,\nu\in\mathbb{W}$. That means that there are only finitely-many occupations in these states and that these occur at finite positions. We also define an orthogonal inner product on $\mathbb{W}$ by $\bra{\mu}\ket{\nu} = \delta_{\mu,\nu}$.

\begin{defn}[Occupation notation]
	\label{Occupation notation defn}
	For $\mu \in \mathbb{W}$ we define the occupation at site $i\in \mathbb{N}$ as
	\be
	\label{eq:occdef}
	\eta^\mu_i = \begin{cases}
		1 & \text{ if $i\in\mu$}\\
		0 & \text{ otherwise}
	\end{cases}.
	\ee
\end{defn}

\begin{defn}
	\label{AB row operator defn}
	A \emph{double-row operator} on $\mathbb{W}$ is defined by its action on co-vectors:
	\begin{equation}
		\label{Row-operator A defn}
		\bra{\mu}A(x|\yalph) := \sum_{\nu\in \mathbb{W}} \text{weight} \left(\begin{tikzpicture}[baseline={([yshift=-.5ex]current bounding box.center)},scale=0.8]
			\draw[lightgray,line width=1.5pt,->] (8,0) -- (1,0) -- (0,0.5) -- (1,1) -- (8,1);
			\draw[lightgray,line width=1.5pt,->] (8,0) -- (7.5,0);
			\draw[blue,fill=blue] (0,0.5) circle (0.1cm);
			\foreach \x in {2,...,7}
			{\draw[lightgray,line width=1.5pt,->] (\x,-0.5) -- (\x,1.5);}
			\node[right] at (8,0) {$0 \leftarrow x$};
			\node[right] at (8,1) {$0 \rightarrow x^{-1}$};
			
			\node[below] at (2,-0.5) {$\eta^\mu_1$};
			\node[below] at (3,-0.5) {$\eta^\mu_2$};
			\node[below] at (4,-0.5) {$\eta^\mu_3$};
			\node[below] at (5,-0.5) {$\cdots$};
			\node[below] at (6,-0.5) {$\cdots$};
			\node[below] at (7,-0.5) {$\cdots$};
			
			\node[below] at (2,-1.2) {$\uparrow$}; 
			\node[below] at (3,-1.2) {$\uparrow$};
			\node[below] at (4,-1.2) {$\uparrow$};  
			
			\node[below] at (2,-1.8) {$y_1$}; 
			\node[below] at (3,-1.8) {$y_2$};
			\node[below] at (4,-1.8) {$y_3$}; 
			
			\node[above] at (2,1.5) {$\eta^\nu_1$};
			\node[above] at (3,1.5) {$\eta^\nu_2$};
			\node[above] at (4,1.5) {$\eta^\nu_3$};
			\node[above] at (5,1.5) {$\cdots$};
			\node[above] at (6,1.5) {$\cdots$};
			\node[above] at (7,1.5) {$\cdots$};
		\end{tikzpicture}\right)
		\bra{\nu},
	\end{equation}
	\begin{equation}
			\label{Row-operator B defn}
			\bra{\mu}\bigdot{B}(z|\yalph) := \sum_{\nu\in \mathbb{W}} \text{weight} \left(\begin{tikzpicture}[baseline={([yshift=-.5ex]current bounding box.center)},scale=0.8]
			\draw[lightgray,line width=1.5pt,->] (8,0) -- (1,0) -- (0,0.5) -- (1,1) -- (8,1);
			\draw[lightgray,line width=1.5pt,->] (8,0) -- (7.5,0);
			\draw[blue,fill=blue] (0,0.5) circle (0.1cm);
			\foreach \x in {2,...,7}
			{\draw[lightgray,line width=1.5pt,->] (\x,-0.5) -- (\x,1.5);
				\draw[black,fill=black] (\x,1) circle (0.1cm);}
			\node[right] at (8,0) {$0 \leftarrow z$};
			\node[right] at (8,1) {$1 \rightarrow z^{-1}$};
			
			\node[below] at (2,-0.5) {$\eta^\mu_1$};
			\node[below] at (3,-0.5) {$\eta^\mu_2$};
			\node[below] at (4,-0.5) {$\eta^\mu_3$};
			\node[below] at (5,-0.5) {$\cdots$};
			\node[below] at (6,-0.5) {$\cdots$};
			\node[below] at (7,-0.5) {$\cdots$};
			
			\node[below] at (2,-1.2) {$\uparrow$}; 
			\node[below] at (3,-1.2) {$\uparrow$};
			\node[below] at (4,-1.2) {$\uparrow$};  
			
			\node[below] at (2,-1.8) {$y_1$}; 
			\node[below] at (3,-1.8) {$y_2$};
			\node[below] at (4,-1.8) {$y_3$}; 
			
			\node[above] at (2,1.5) {$\eta^\nu_1$};
			\node[above] at (3,1.5) {$\eta^\nu_2$};
			\node[above] at (4,1.5) {$\eta^\nu_3$};
			\node[above] at (5,1.5) {$\cdots$};
			\node[above] at (6,1.5) {$\cdots$};
			\node[above] at (7,1.5) {$\cdots$};
		\end{tikzpicture}\right)
		\bra{\nu}.
	\end{equation}
	where $x,z\in \mathbb{C}$ are horizontal spectral parameters and $\yalph=(y_1,y_2,\dots)$ is an infinite collection of vertical spectral parameters. For conciseness, we will often omit the family of vertical parameters from our notation by writing $A(x|\yalph)=A(x)$ and $\bigdot{B}(z|\yalph)=\bigdot{B}(z)$. 
\end{defn}

\begin{prop}
	\label{B left-eigenvalue prop}
	The empty state $\emptyset\in\mathbb{W}$ corresponds to a left-eigenvector of the double-row operator \eqref{Row-operator B defn}
	\begin{equation}
		\bra{\emptyset} \bigdot{B}(z) = h(z) \bra{\emptyset},
	\end{equation}
    where $h(z)$ is given by \eqref{eq:hdef}.
\end{prop}
\begin{proof}
	When the bottom state is empty in \eqref{Row-operator B defn} there is only one possible configuration on the double-row. This single state is depicted as
	
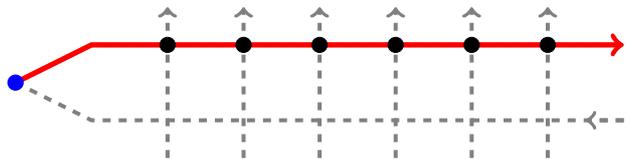
\[
	\begin{tikzpicture}[baseline={([yshift=-.5ex]current bounding box.center)},scale=1]
	\draw[gray,dashed,line width=1.5pt,->] (8,0) -- (1,0) -- (0,0.5) -- (1,1) -- (8,1);
	\draw[gray,dashed,line width=1.5pt,->] (8,0) -- (7.5,0);
	\draw[red,line width=2pt,->] (0,0.5) -- (1,1) -- (8,1);
	\draw[blue,fill=blue] (0,0.5) circle (0.1cm);
	\foreach \x in {2,...,7}
	{\draw[gray,dashed,line width=1.5pt,->] (\x,-0.5) -- (\x,1.5);
		\draw[black,fill=black] (\x,1) circle (0.1cm);
	}
	\end{tikzpicture},\]
	which propagates the empty state from below to above the double-row. The weight of all bulk vertices in this picture are 1, while the boundary vertex has weight $h(z)$. 
\end{proof}
A crucial property of the double-row operators of Definition \ref{AB row operator defn} is their algebra of commutation and exchange relations. In order to prove these relations, we must first define a version of the double-row operators with finitely many columns.
\begin{defn}
    For some fixed $N\in\mathbb{N}$, we define the monodromy matrices $\mathcal{T}^{(N)}(x|\yalph)$ and $\bigdot{\mathcal{T}}^{(N)}(z|\yalph)$. The elements of these matrices are double-row row transfer matrices with $N$ columns indexed by $i,j\in\{0,1\}$. For fixed states $\mu,\nu\in\mathbb{N}$ with $\mu_1,\nu_1\leq N$ these matrices are represented as
    \begin{equation}
        \bra{\mu}\mathcal{T}_{i,j}^{(N)}(x|\yalph)\ket{\nu} := \begin{tikzpicture}[baseline={([yshift=-.5ex]current bounding box.center)},scale=0.8]
			\draw[lightgray,line width=1.5pt,->] (7,0) -- (1,0) -- (0,0.5) -- (1,1) -- (7,1);
			\draw[lightgray,line width=1.5pt,->] (7,0) -- (6.5,0);
			\draw[blue,fill=blue] (0,0.5) circle (0.1cm);
			\foreach \x in {2,...,6}
			{\draw[lightgray,line width=1.5pt,->] (\x,-0.5) -- (\x,1.5);}
			\node[right] at (7,0) {$i \leftarrow x$};
			\node[right] at (7,1) {$j \rightarrow x^{-1}$};
			
			\node[below] at (2,-0.5) {$\eta^\mu_1$};
			\node[below] at (3,-0.5) {$\eta^\mu_2$};
			\node[below] at (4,-0.5) {$\cdots$};
			\node[below] at (5,-0.5) {$\cdots$};
            \node[below] at (6,-0.5) {$\eta^\mu_N$};
			
			\node[below] at (2,-1.2) {$\uparrow$}; 
			\node[below] at (3,-1.2) {$\uparrow$};
			\node[below] at (6,-1.2) {$\uparrow$};  
			
			\node[below] at (2,-1.8) {$y_1$}; 
			\node[below] at (3,-1.8) {$y_2$};
			\node[below] at (6,-1.8) {$y_N$}; 
			
			\node[above] at (2,1.5) {$\eta^\nu_1$};
			\node[above] at (3,1.5) {$\eta^\nu_2$};
			\node[above] at (4,1.5) {$\cdots$};
			\node[above] at (5,1.5) {$\cdots$};
			\node[above] at (6,1.5) {$\eta^\nu_N$};
		\end{tikzpicture},
        \end{equation}
        \begin{equation}
            \bra{\mu}\bigdot{\mathcal{T}}_{i,j}^{(N)}(z|\yalph)\ket{\nu} := \begin{tikzpicture}[baseline={([yshift=-.5ex]current bounding box.center)},scale=0.8]
			\draw[lightgray,line width=1.5pt,->] (7,0) -- (1,0) -- (0,0.5) -- (1,1) -- (7,1);
			\draw[lightgray,line width=1.5pt,->] (7,0) -- (6.5,0);
			\draw[blue,fill=blue] (0,0.5) circle (0.1cm);
			\foreach \x in {2,...,6}
			{\draw[lightgray,line width=1.5pt,->] (\x,-0.5) -- (\x,1.5);
            \draw[black,fill=black] (\x,1) circle (0.1cm);}
			\node[right] at (7,0) {$i \leftarrow z$};
			\node[right] at (7,1) {$j \rightarrow z^{-1}$};
			
			\node[below] at (2,-0.5) {$\eta^\mu_1$};
			\node[below] at (3,-0.5) {$\eta^\mu_2$};
			\node[below] at (4,-0.5) {$\cdots$};
			\node[below] at (5,-0.5) {$\cdots$};
            \node[below] at (6,-0.5) {$\eta^\mu_N$};
			
			\node[below] at (2,-1.2) {$\uparrow$}; 
			\node[below] at (3,-1.2) {$\uparrow$};
			\node[below] at (6,-1.2) {$\uparrow$};  
			
			\node[below] at (2,-1.8) {$y_1$}; 
			\node[below] at (3,-1.8) {$y_2$};
			\node[below] at (6,-1.8) {$y_N$}; 
			
			\node[above] at (2,1.5) {$\eta^\nu_1$};
			\node[above] at (3,1.5) {$\eta^\nu_2$};
			\node[above] at (4,1.5) {$\cdots$};
			\node[above] at (5,1.5) {$\cdots$};
			\node[above] at (6,1.5) {$\eta^\nu_N$};
		\end{tikzpicture},
        \end{equation}
        where $x,z\in\mathbb{C}$ are horizontal spectral parameters and $\yalph = (y_1,\dots,y_N)$ is a collection of vertical spectral parameters. As with the definition of the double-row operators, we will omit the family of vertical spectral parameters from our notation by writing $\mathcal{T}^{(N)}(x|\yalph) = \mathcal{T}^{(N)}(x)$ and $\bigdot{\mathcal{T}}^{(N)}(z|\yalph) = \bigdot{\mathcal{T}}^{(N)}(z)$. The elements of these matrices are represented as
        \begin{equation}
            \mathcal{T}^{(N)}(x) = \left(\begin{array}{cc}
               A^{(N)}(x)  & B^{(N)}(x) \\
               C^{(N)}(x)  & D^{(N)}(x)
            \end{array}\right),
            \hspace{0.5cm}
            \bigdot{\mathcal{T}}^{(N)}(z) = \left(\begin{array}{cc}
               \bigdot{A}^{(N)}(z)  & \bigdot{B}^{(N)}(z) \\
               \bigdot{C}^{(N)}(z)  & \bigdot{D}^{(N)}(z)
            \end{array}\right)
        \end{equation}
\end{defn}
For states $\mu,\nu\in \mathbb{W}$, we can recover the infinite column double row operators by 
\begin{align}
    \lim_{N\to\infty}\bra{\mu}A^{(N)}(x|\yalph) \ket{\nu} & = \bra{\mu}A(x|\yalph)\ket{\nu}, \\
    \lim_{N\to\infty}\bra{\mu}\bigdot{B}^{(N)}(z|\yalph) \ket{\nu} & = \bra{\mu}\bigdot{B}(z|\yalph)\ket{\nu},
\end{align}
where the infinite column row-double row operators depend on the infinite collection of vertical parameters $\yalph = (y_1,y_2,\dots)$ and we regard the finite column operators as having dependence on the first $N$ elements of the collection, i.e. $(y_1,\dots,y_N)$.
\begin{prop}
    \label{prop: AA limit condition}
    Fix $x_1,x_2 \in \mathbb{C}$ and assume that there exists $\rho>0$ such that
	\begin{equation}
		\label{eq: AA limit condition}
		\abs{\frac{1-x_1y_k}{1-qx_1y_k}\frac{q(1-x_2/y_k)}{1-qx_2/y_k}}\leq \rho <1,
	\end{equation}
    for all $k\in\mathbb{N}$. Then the following limit holds for states $\mu,\nu\in\mathbb{W}$
    \begin{equation}
        \label{eq:AA N limit}
        \lim_{N\to\infty} \sum_{p\in\{0,1\}} \bra{\mu}\mathcal{T}_{0,p}^{(N)}(x_1)\mathcal{T}_{p,0}^{(N)}(x_2)\ket{\nu} R_{x_1 x_2}(0,p;p,0) = \bra{\mu} A(x_1)A(x_2) \ket{\nu}.
    \end{equation}
    Which has a graphical interpretation
    \begin{equation}
        \label{eq:AA N limit pic}
        \lim_{N\to\infty}\sum_{p\in\{0,1\}}\begin{tikzpicture}[baseline={([yshift=-.5ex]current bounding box.center)},scale=0.7]
    		\draw[lightgray,line width=1.5pt,->] (7,0) -- (1,0) -- (0,0.5) -- (1,1) -- (6,1) -- (7,2);
    		\draw[lightgray,line width=1.5pt,->] (7,0) -- (6.5,0);
    		\draw[lightgray,line width=1.5pt,->] (7,1) -- (6,2) -- (1,2) -- (0,2.5) -- (1,3) -- (7,3);
    		\draw[lightgray,line width=1.5pt,->] (7,1) -- (6.75,1.25);
    		\draw[blue,fill=blue] (0,0.5) circle (0.1cm);
    		\draw[blue,fill=blue] (0,2.5) circle (0.1cm);
    		\foreach \x in {2,...,5}
    		{\draw[lightgray,line width=1.5pt,->] (\x,-0.5) -- (\x,3.5);}
    		\node[right] at (7,0) {$0 \leftarrow x_1$};
    		\node[right] at (7,1) {$0 \leftarrow x_2$};
    		\node[right] at (7,2) {$0 \rightarrow x_1^{-1}$};
    		\node[right] at (7,3) {$0 \rightarrow x_2^{-1}$};
    		\node[below] at (6,1) {$p$};
    		\node[above] at (6,2) {$p$};
    
            \node[above] at (2,3.5) {$\eta_1^\nu$};
                \node[above] at (3,3.5) {$\cdots$};
                \node[above] at (4,3.5) {$\cdots$};
                \node[above] at (5,3.5) {$\eta_N^\nu$};
    
                \node[below] at (2,-0.5) {$\eta_1^\mu$};
                \node[below] at (3,-0.5) {$\cdots$};
                \node[below] at (4,-0.5) {$\cdots$};
                \node[below] at (5,-0.5) {$\eta_N^\mu$};
    	\end{tikzpicture}
     =
     \begin{tikzpicture}[baseline={([yshift=-.5ex]current bounding box.center)},scale=0.7]
    		\draw[lightgray,line width=1.5pt,->] (7,0) -- (1,0) -- (0,0.5) -- (1,1) -- (7,1);
    		\draw[lightgray,line width=1.5pt,->] (7,0) -- (6.5,0);
    		\draw[lightgray,line width=1.5pt,->] (7,2) -- (1,2) -- (0,2.5) -- (1,3) -- (7,3);
    		\draw[lightgray,line width=1.5pt,->] (7,2) -- (6.5,2);
    		\draw[blue,fill=blue] (0,0.5) circle (0.1cm);
    		\draw[blue,fill=blue] (0,2.5) circle (0.1cm);
    		\foreach \x in {2,...,6}
    		{\draw[lightgray,line width=1.5pt,->] (\x,-0.5) -- (\x,3.5);}
    		\node[right] at (7,0) {$0 \leftarrow x_1$};
    		\node[right] at (7,1) {$0 \rightarrow x_1^{-1}$};
    		\node[right] at (7,2) {$0 \leftarrow x_2$};
    		\node[right] at (7,3) {$0 \rightarrow x_2^{-1}$};
    
            \node[above] at (2,3.5) {$\eta_1^\nu$};
                \node[above] at (3,3.5) {$\eta_2^\nu$};
                \node[above] at (4,3.5) {$\cdots$};
                \node[above] at (5,3.5) {$\cdots$};
                \node[above] at (6,3.5) {$\cdots$};
    
                \node[below] at (2,-0.5) {$\eta_1^\mu$};
                \node[below] at (3,-0.5) {$\eta_2^\mu$};
                \node[below] at (4,-0.5) {$\cdots$};
                \node[below] at (5,-0.5) {$\cdots$};
                \node[below] at (6,-0.5) {$\cdots$};
    	\end{tikzpicture}.
    \end{equation}
\end{prop}
\begin{proof}
    Consider $\mu,\nu\in\mathbb{W}$  to be finite states. We denote their maximum occupation by $\tau=\max\{\mu_1,\nu_1\}$. Let $N$ be an integer satisfying $N>\tau$ which is also independent of $\tau$. By expanding the sum on the left side of \eqref{eq:AA N limit} as
    \begin{equation}
        \label{eq:AA limit proof eq 1}
        \lim_{N\to\infty} \bra{\mu}A^{(N)}(x_1)A^{(N)}(x_2) \ket{\nu} + \frac{x_1x_2(1-q)}{1-qx_1x_2} \lim_{N\to\infty} \bra{\mu}B^{(N)}(x_1)C^{(N)}(x_2)\ket{\nu}.
    \end{equation}
    The first term, which corresponds to $p=0$, has the limit 
    \[\lim_{N\to\infty} \bra{\mu}A^{(N)}(x_1)A^{(N)}(x_2)\ket{\nu} = \bra{\mu} A(x_1)A(x_2) \ket{\nu},\]
    which is our final result. So all that remains to prove is that the second term ($p=1$) in \eqref{eq:AA limit proof eq 1} vanishes under the limit.
    
    Since we are interested in the large $N$ limit, we consider for $N>\tau$ the partition function
    \begin{equation}
        \label{eq:AA limit proof eq 2}
        \bra{\mu}B^{(N)}(x_1)C^{(N)}(x_2) \ket{\nu} = 
        \begin{tikzpicture}[baseline={([yshift=-.5ex]current bounding box.center)},scale=0.7]
    		\draw[lightgray,line width=1.5pt,->] (9,0) -- (1,0) -- (0,0.5) -- (1,1) -- (9,1);
    		\draw[lightgray,line width=1.5pt,->] (9,0) -- (8.5,0);
    		\draw[lightgray,line width=1.5pt,->] (9,2) -- (1,2) -- (0,2.5) -- (1,3) -- (9,3);
    		\draw[lightgray,line width=1.5pt,->] (9,2) -- (8.5,2);
    		\draw[blue,fill=blue] (0,0.5) circle (0.1cm);
    		\draw[blue,fill=blue] (0,2.5) circle (0.1cm);
    		\foreach \x in {2,...,8}
    		{\draw[lightgray,line width=1.5pt,->] (\x,-0.5) -- (\x,3.5);}
    		\node[right] at (9,0) {$0 \leftarrow x_1$};
    		\node[right] at (9,1) {$1 \rightarrow x_1^{-1}$};
    		\node[right] at (9,2) {$1 \leftarrow x_2$};
    		\node[right] at (9,3) {$0 \rightarrow x_2^{-1}$};
    
            \node[above] at (2,3.5) {$\eta_1^\nu$};
                \node[above] at (3,3.5) {$\cdots$};
                \node[above] at (4,3.5) {$\cdots$};
                \node[above] at (5,3.5) {$\eta_\tau^\nu$};
                \node[above] at (6,3.5) {$0$};
                \node[above] at (7,3.5) {$\cdots$};
                \node[above] at (8,3.5){$0$};
    
                \node[below] at (2,-0.5) {$\eta_1^\mu$};
                \node[below] at (3,-0.5) {$\cdots$};
                \node[below] at (4,-0.5) {$\cdots$};
                \node[below] at (5,-0.5) {$\eta_\tau^\mu$};
                \node[below] at (6,-0.5) {$0$};
                \node[below] at (7,-0.5) {$\cdots$};
                \node[below] at (8,-0.5) {$0$};

            \node[below] at (2,-1.2) {$\uparrow$}; 
			\node[below] at (5,-1.2) {$\uparrow$};
			\node[below] at (6,-1.2) {$\uparrow$};  
            \node[below] at (8,-1.2) {$\uparrow$}; 
			
			\node[below] at (2,-1.8) {$y_1$}; 
			\node[below] at (5,-1.8) {$y_\tau$};
			\node[below] at (6,-1.8) {$y_{\tau+1}$}; 
            \node[below] at (8,-1.8) {$y_N$};

            \draw[red,line width=2pt,->] (5.5,1) -- (9,1);
            \draw[red,line width=2pt,<-] (5.5,2) -- (9,2);

            \draw[dotted,line width=1pt,rounded corners=15pt]  (-0.3,-0.2) rectangle ++(5.8,3.4);
    	\end{tikzpicture}.
    \end{equation}
    Since $N>\tau$, it follows that $\eta_i^\mu,\eta_i^\nu=0$ for all $\tau<i\leq N$. This freezes the columns to the right of the rectangle as shown in \eqref{eq:AA limit proof eq 2}. The rectangle itself can be identified as a double row partition with $\tau$ columns. By evaluating the frozen section, \eqref{eq:AA limit proof eq 2} is reduced to 
    \begin{equation}
        \label{eq:AA limit proof eq 3}
        \bra{\mu}B^{(N)}(x_1)C^{(N)}(x_2) \ket{\nu} = \prod_{k=\tau+1}^N \left[\frac{1-x_1y_k}{1-qx_1y_k}\frac{q(1-x_2/y_k)}{1-qx_2/y_k}\right] \bra{\mu}B^{(\tau)}(x_1)C^{(\tau)}(x_2) \ket{\nu}. 
    \end{equation}
    Then with condition \eqref{eq: AA limit condition}, we can bound \eqref{eq:AA limit proof eq 3} as
    \begin{equation}
        \abs{\bra{\mu}B^{(N)}(x_1)C^{(N)}(x_2) \ket{\nu}} \leq \rho^{N-\tau} \abs{\bra{\mu}B^{(\tau)}(x_1)C^{(\tau)}(x_2) \ket{\nu}} ,
    \end{equation}
    so that the $p=1$ term of \eqref{eq:AA limit proof eq 1} vanishes as $N\to\infty$.
\end{proof}
\begin{prop}
    \label{prop: AB limit condition}
    Fix $x,z\in\mathbb{C}$ and assume that there exists $\rho>0$ such that
	\begin{equation}
		\label{AB operator commutation condition 1}
		\abs{\frac{1-xy_k}{1-qxy_k}\frac{q(1-z/y_k)}{1-qz/y_k}}\le\rho<1, \hspace{0.5cm} \abs{\frac{1-xy_k}{1-qxy_k}\frac{1-qzy_k}{1-zy_k}}\leq\rho<1,
	\end{equation}
    for all $k\in\mathbb{N}$. Then the following limit holds
    \begin{multline}
        \label{eq: AB limit condition}
        \lim_{N\to\infty} \sum_{p_1,p_2\in\{0,1\} \atop p_1\geq p_2} \bra{\mu}\mathcal{T}_{0,p_1}^{(N)}(x)\bigdot{\mathcal{T}}_{p_1-p_2,1-p_2}^{(N)}(z) \ket{\nu} R_{xz}(0,p_1;p_1-p_2,p_2) R_{z/x}(p_2,1-p_2;0,1) \\ = \frac{x-z}{x-qz}\bra{\mu} A(x)\bigdot{B}(z) \ket{\nu}.
    \end{multline}
    Which has graphical interpretation
    \begin{equation}
        \lim_{N\to\infty} \sum_{p_1,p_2\in\{0,1\} \atop p_1\geq p_2} \begin{tikzpicture}[baseline={([yshift=-.5ex]current bounding box.center)},scale=0.7]
    		\draw[lightgray,line width=1.5pt,->] (8,0) -- (1,0) -- (0,0.5) -- (1,1) -- (6,1) -- (8,3);
    		\draw[lightgray,line width=1.5pt,->] (8,0) -- (7.5,0);
    		\draw[lightgray,line width=1.5pt,->] (8,1) -- (7,1) -- (6,2) -- (1,2) -- (0,2.5) -- (1,3) -- (7,3) -- (8,2);
    		\draw[lightgray,line width=1.5pt,->] (8,1) -- (7.5,1);
    		\draw[blue,fill=blue] (0,0.5) circle (0.1cm);
    		\draw[blue,fill=blue] (0,2.5) circle (0.1cm);
    		\foreach \x in {2,...,5}
    		{\draw[lightgray,line width=1.5pt,->] (\x,-0.5) -- (\x,3.5);
                \draw[black,fill=black] (\x,3) circle (0.1cm);}
    		\node[right] at (8,0) {$0 \leftarrow x$};
    		\node[right] at (8,1) {$0 \leftarrow z$};
    		\node[right] at (8,2) {$1 \rightarrow z^{-1}$};
    		\node[right] at (8,3) {$0 \rightarrow x^{-1}$};
    		\node[below] at (6,1) {$p_1$};
    		\node[above] at (6,2) {$p_1-p_2$};
    		\node[right] at (6.9,1.8) {$p_2$};
    		\node[above] at (7,3) {$1-p_2$};
    
            \node[above] at (2,3.5) {$\eta_1^\nu$};
                \node[above] at (3,3.5) {$\cdots$};
                \node[above] at (4,3.5) {$\cdots$};
                \node[above] at (5,3.5) {$\eta_N^\nu$};
    
                \node[below] at (2,-0.5) {$\eta_1^\mu$};
                \node[below] at (3,-0.5) {$\cdots$};
                \node[below] at (4,-0.5) {$\cdots$};
                \node[below] at (5,-0.5) {$\eta_N^\mu$};
    	\end{tikzpicture}
        =
        \frac{x-z}{x-qz}\times\begin{tikzpicture}[baseline={([yshift=-.5ex]current bounding box.center)},scale=0.7]
    		\draw[lightgray,line width=1.5pt,->] (7,0) -- (1,0) -- (0,0.5) -- (1,1) -- (7,1);
    		\draw[lightgray,line width=1.5pt,->] (7,0) -- (6.5,0);
    		\draw[lightgray,line width=1.5pt,->] (7,2) -- (1,2) -- (0,2.5) -- (1,3) -- (7,3);
    		\draw[lightgray,line width=1.5pt,->] (7,2) -- (6.5,2);
    		\draw[blue,fill=blue] (0,0.5) circle (0.1cm);
    		\draw[blue,fill=blue] (0,2.5) circle (0.1cm);
    		\foreach \x in {2,...,6}
    		{\draw[lightgray,line width=1.5pt,->] (\x,-0.5) -- (\x,3.5);
             \draw[black,fill=black] (\x,3) circle (0.1cm);}
    		\node[right] at (7,0) {$0 \leftarrow x$};
    		\node[right] at (7,1) {$0 \rightarrow x^{-1}$};
    		\node[right] at (7,2) {$0 \leftarrow z$};
    		\node[right] at (7,3) {$1 \rightarrow z^{-1}$};
    
            \node[above] at (2,3.5) {$\eta_1^\nu$};
                \node[above] at (3,3.5) {$\eta_2^\nu$};
                \node[above] at (4,3.5) {$\cdots$};
                \node[above] at (5,3.5) {$\cdots$};
                \node[above] at (6,3.5) {$\cdots$};
    
                \node[below] at (2,-0.5) {$\eta_1^\mu$};
                \node[below] at (3,-0.5) {$\eta_2^\mu$};
                \node[below] at (4,-0.5) {$\cdots$};
                \node[below] at (5,-0.5) {$\cdots$};
                \node[below] at (6,-0.5) {$\cdots$};
    	\end{tikzpicture}.
    \end{equation}
\end{prop}
\begin{proof}
    In a similar manner to the proof of Proposition \ref{prop: AA limit condition}, we fix $\mu,\nu\in\mathbb{W}$ with maximum occupation $\tau=\max\{\mu_1,\nu_1\}$. We may write the terms in the sum over $p_1,p_2$ explicitly as
    \begin{multline}
        \label{eq:AB limit proof eq 1}
        \lim_{N\to\infty}\bra{\mu}\left[\frac{x-z}{x-qz}A^{(N)}(x)\bigdot{B}^{(N)}(z) + \frac{xz(1-q)}{1-qxz}\frac{x-z}{x-qz}B^{(N)}(x)\bigdot{D}^{(N)}(z) \right. \\
        +\left. \frac{1-xz}{1-qxz}\frac{x(1-q)}{x-qz} B^{(N)}(x)\bigdot{A}^{(N)}(z)\right]\ket{\nu},
    \end{multline}
    where each of the terms corresponds to $(p_1,p_2)=(0,0),(1,0),(1,1)$ respectively. Each of these three terms are depicted in Figure \ref{fig:AB limit proof}. We will analyse each term individually.
    \begin{figure}
        \begin{tabular}{ccc}
        \quad
             \begin{tikzpicture}[scale=0.6]
            \draw[lightgray,dashed,line width=1.5pt,->] (7,0) -- (1,0) -- (0,0.5) -- (1,1) -- (5,1) -- (7,3);
    		\draw[lightgray,dashed,line width=1.5pt,->] (7,0) -- (6.5,0);
    		\draw[lightgray,dashed,line width=1.5pt,->] (7,1) -- (6,1) -- (5,2) -- (1,2) -- (0,2.5) -- (1,3) -- (6,3) -- (7,2);
    		\draw[lightgray,dashed,line width=1.5pt,->] (7,1) -- (6.5,1);
            \draw[red,line width=2pt,->] (4,3) -- (6,3) -- (7,2);
    		\draw[blue,fill=blue] (0,0.5) circle (0.1cm);
    		\draw[blue,fill=blue] (0,2.5) circle (0.1cm);
    		\foreach \x in {2,...,4}
    		{\draw[lightgray,dashed,line width=1.5pt,->] (\x,-0.5) -- (\x,3.5);
                \draw[black,fill=black] (\x,3) circle (0.1cm);}
    
            \node[above] at (2,3.5) {$\eta_1^\nu$};
                \node[above] at (3,3.5) {$\cdots$};
                \node[above] at (4,3.5) {$\eta_N^\nu$};
    
                \node[below] at (2,-0.5) {$\eta_1^\mu$};
                \node[below] at (3,-0.5) {$\cdots$};
                \node[below] at (4,-0.5) {$\eta_N^\mu$};
             \end{tikzpicture}
             \quad
             &
             \quad
             \begin{tikzpicture}[scale=0.6]
            \draw[lightgray,dashed,line width=1.5pt,->] (7,0) -- (1,0) -- (0,0.5) -- (1,1) -- (5,1) -- (7,3);
    		\draw[lightgray,dashed,line width=1.5pt,->] (7,0) -- (6.5,0);
    		\draw[lightgray,dashed,line width=1.5pt,->] (7,1) -- (6,1) -- (5,2) -- (1,2) -- (0,2.5) -- (1,3) -- (6,3) -- (7,2);
    		\draw[lightgray,dashed,line width=1.5pt,->] (7,1) -- (6.5,1);
            \draw[red,line width=2pt,->] (4,3) -- (6,3) -- (7,2);
            \draw[red,line width=2pt,->] (4,1) -- (5,1) -- (5.5,1.5) -- (5,2) -- (4,2);
    		\draw[blue,fill=blue] (0,0.5) circle (0.1cm);
    		\draw[blue,fill=blue] (0,2.5) circle (0.1cm);
    		\foreach \x in {2,...,4}
    		{\draw[lightgray,dashed,line width=1.5pt,->] (\x,-0.5) -- (\x,3.5);
                \draw[black,fill=black] (\x,3) circle (0.1cm);}
    
            \node[above] at (2,3.5) {$\eta_1^\nu$};
                \node[above] at (3,3.5) {$\cdots$};
                \node[above] at (4,3.5) {$\eta_N^\nu$};
    
                \node[below] at (2,-0.5) {$\eta_1^\mu$};
                \node[below] at (3,-0.5) {$\cdots$};
                \node[below] at (4,-0.5) {$\eta_N^\mu$};
             \end{tikzpicture}
             \quad
             &
             \quad
             \begin{tikzpicture}[scale=0.6]
            \draw[lightgray,dashed,line width=1.5pt,->] (7,0) -- (1,0) -- (0,0.5) -- (1,1) -- (5,1) -- (7,3);
    		\draw[lightgray,dashed,line width=1.5pt,->] (7,0) -- (6.5,0);
    		\draw[lightgray,dashed,line width=1.5pt,->] (7,1) -- (6,1) -- (5,2) -- (1,2) -- (0,2.5) -- (1,3) -- (6,3) -- (7,2);
    		\draw[lightgray,dashed,line width=1.5pt,->] (7,1) -- (6.5,1);
            \draw[red,line width=2pt,->] (4,1) -- (5,1) -- (6.5,2.5) -- (7,2);
    		\draw[blue,fill=blue] (0,0.5) circle (0.1cm);
    		\draw[blue,fill=blue] (0,2.5) circle (0.1cm);
    		\foreach \x in {2,...,4}
    		{\draw[lightgray,dashed,line width=1.5pt,->] (\x,-0.5) -- (\x,3.5);
                \draw[black,fill=black] (\x,3) circle (0.1cm);}
    
            \node[above] at (2,3.5) {$\eta_1^\nu$};
                \node[above] at (3,3.5) {$\cdots$};
                \node[above] at (4,3.5) {$\eta_N^\nu$};
    
                \node[below] at (2,-0.5) {$\eta_1^\mu$};
                \node[below] at (3,-0.5) {$\cdots$};
                \node[below] at (4,-0.5) {$\eta_N^\mu$};
             \end{tikzpicture}
            \quad
        \end{tabular}
        \caption{The terms of the sum \eqref{eq: AB limit condition} corresponding to $(p_1,p_2)=(0,0),(1,0),(1,1)$ respectively.}
        \label{fig:AB limit proof}
    \end{figure}
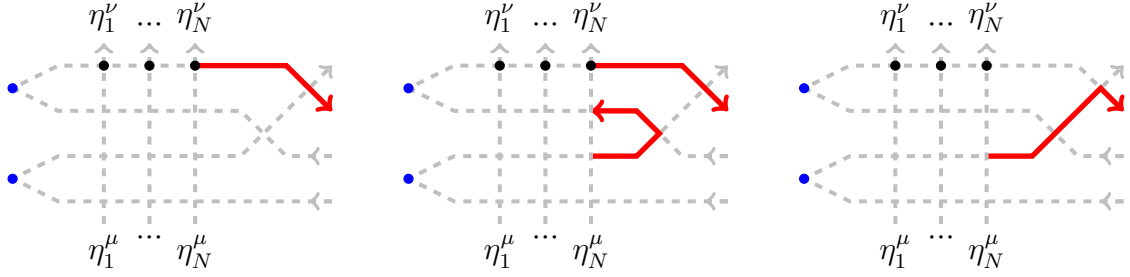

    Firstly, the $p_1=p_2=0$ term has the limit 
    \[\lim_{N\to\infty}\bra{\mu}\frac{x-z}{x-qz}A^{(N)}(x)\bigdot{B}^{(N)}(z)\ket{\nu} = \frac{x-z}{x-qz} \bra{\mu}A(x)\bigdot{B}(z)\ket{\nu},\]
    which is the desired result. So it remains to show that the other terms vanish in the limit. 
    
    Consider now the term associated to $p_1=1,p_2=0$. By virtue of the configurations being finite, this term can be decomposed into two possible configuration types. For some integer $N>\tau$ which is independent of $\tau$, these configurations are summed over
    \begin{multline}
    \label{eq:AB limit proof eq 2}
        \bra{\mu}B^{(N)}(x)\bigdot{D}^{(N)}(z)\ket{\nu} = \begin{tikzpicture}[baseline={([yshift=-.5ex]current bounding box.center)},scale=0.7]
    		\draw[lightgray,line width=1.5pt,->] (11,0) -- (1,0) -- (0,0.5) -- (1,1) -- (11,1);
    		\draw[lightgray,line width=1.5pt,->] (11,0) -- (10.5,0);
    		\draw[lightgray,line width=1.5pt,->] (11,2) -- (1,2) -- (0,2.5) -- (1,3) -- (11,3);
    		\draw[lightgray,line width=1.5pt,->] (11,2) -- (10.5,2);
    		\draw[blue,fill=blue] (0,0.5) circle (0.1cm);
    		\draw[blue,fill=blue] (0,2.5) circle (0.1cm);
    		\foreach \x in {2,...,10}
    		{\draw[lightgray,line width=1.5pt,->] (\x,-0.5) -- (\x,3.5);
            \draw[black,fill=black] (\x,3) circle (0.1cm);}
    		\node[right] at (11,0) {$0 \leftarrow x$};
    		\node[right] at (11,1) {$1 \rightarrow x^{-1}$};
    		\node[right] at (11,2) {$1 \leftarrow z$};
    		\node[right] at (11,3) {$1 \rightarrow z^{-1}$};
            \node[above] at (2,3.5) {$\eta_1^\nu$};
                \node[above] at (3,3.5) {$\cdots$};
                \node[above] at (4,3.5) {$\cdots$};
                \node[above] at (5,3.5) {$\eta_\tau^\nu$};
                \node[above] at (6,3.5) {$0$};
                \node[above] at (7,3.5) {$\cdots$};
                \node[above] at (8,3.5){$\cdots$};
                \node[above] at (9,3.5){$\cdots$};
                \node[above] at (10,3.5){$0$};  
                \node[below] at (2,-0.5) {$\eta_1^\mu$};
                \node[below] at (3,-0.5) {$\cdots$};
                \node[below] at (4,-0.5) {$\cdots$};
                \node[below] at (5,-0.5) {$\eta_\tau^\mu$};
                \node[below] at (6,-0.5) {$0$};
                \node[below] at (7,-0.5) {$\cdots$};
                \node[below] at (8,-0.5) {$\cdots$};
                \node[below] at (9,-0.5) {$\cdots$};
                \node[below] at (10,-0.5) {$0$};
            \node[below] at (2,-1.2) {$\uparrow$}; 
			\node[below] at (5,-1.2) {$\uparrow$};
			\node[below] at (6,-1.2) {$\uparrow$};  
            \node[below] at (10,-1.2) {$\uparrow$}; 		
			\node[below] at (2,-1.8) {$y_1$}; 
			\node[below] at (5,-1.8) {$y_\tau$};
			\node[below] at (6,-1.8) {$y_{\tau+1}$}; 
            \node[below] at (10,-1.8) {$y_N$};
            \draw[red,line width=2pt,->] (5.5,1) -- (11,1);
            \draw[red,line width=2pt,<-] (5.5,2) -- (11,2);
            \draw[red,line width=2pt,->] (5.5,3) -- (11,3);
            \draw[dotted,line width=1pt,rounded corners=15pt]  (-0.3,-0.2) rectangle ++(5.8,3.4);
            \foreach \x in {2,...,10}{\draw[black,fill=black] (\x,3) circle (0.1cm);}
    	\end{tikzpicture}\\
        +\sum_{\ell=\tau+1}^N  \begin{tikzpicture}[baseline={([yshift=-.5ex]current bounding box.center)},scale=0.7]
    		\draw[lightgray,line width=1.5pt,->] (11,0) -- (1,0) -- (0,0.5) -- (1,1) -- (11,1);
    		\draw[lightgray,line width=1.5pt,->] (11,0) -- (10.5,0);
    		\draw[lightgray,line width=1.5pt,->] (11,2) -- (1,2) -- (0,2.5) -- (1,3) -- (11,3);
    		\draw[lightgray,line width=1.5pt,->] (11,2) -- (10.5,2);
    		\draw[blue,fill=blue] (0,0.5) circle (0.1cm);
    		\draw[blue,fill=blue] (0,2.5) circle (0.1cm);
    		\foreach \x in {2,...,10}
    		{\draw[lightgray,line width=1.5pt,->] (\x,-0.5) -- (\x,3.5);}
    		\node[right] at (11,0) {$0 \leftarrow x$};
    		\node[right] at (11,1) {$1 \rightarrow x^{-1}$};
    		\node[right] at (11,2) {$1 \leftarrow z$};
    		\node[right] at (11,3) {$1 \rightarrow z^{-1}$};
            \node[above] at (2,3.5) {$\eta_1^\nu$};
                \node[above] at (3,3.5) {$\cdots$};
                \node[above] at (4,3.5) {$\cdots$};
                \node[above] at (5,3.5) {$\eta_\tau^\nu$};
                \node[above] at (6,3.5) {$0$};
                \node[above] at (7,3.5) {$\cdots$};
                \node[above] at (8,3.5) {$0$};
                \node[above] at (9,3.5) {$\cdots$};
                \node[above] at (10,3.5){$0$};  
                \node[below] at (2,-0.5) {$\eta_1^\mu$};
                \node[below] at (3,-0.5) {$\cdots$};
                \node[below] at (4,-0.5) {$\cdots$};
                \node[below] at (5,-0.5) {$\eta_\tau^\mu$};
                \node[below] at (6,-0.5) {$0$};
                \node[below] at (7,-0.5) {$\cdots$};
                \node[below] at (8,-0.5) {$0$};
                \node[below] at (9,-0.5) {$\cdots$};
                \node[below] at (10,-0.5) {$0$};
            \node[below] at (2,-1.2) {$\uparrow$}; 
			\node[below] at (5,-1.2) {$\uparrow$};
			\node[below] at (6,-1.2) {$\uparrow$};
            \node[below] at (8,-1.2) {$\uparrow$}; 
            \node[below] at (10,-1.2) {$\uparrow$}; 		
			\node[below] at (2,-1.8) {$y_1$}; 
			\node[below] at (5,-1.8) {$y_\tau$};
			\node[below] at (6,-1.8) {$y_{\tau+1}$};
            \node[below] at (8,-1.8) {$y_\ell$};
            \node[below] at (10,-1.8) {$y_N$};
            \draw[red,line width=2pt,->] (5.5,1) -- (11,1);
            \draw[red,line width=2pt,->] (11,2) -- (8,2) -- (8,3) -- (11,3);
            \draw[dotted,line width=1pt,rounded corners=15pt]  (-0.3,-0.2) rectangle ++(5.8,3.4);
            \foreach \x in {2,...,10}{\draw[black,fill=black] (\x,3) circle (0.1cm);}
    	\end{tikzpicture},
    \end{multline}
    where the dotted rectangles can are identified as stacked double row partition functions with $\tau$ columns. The columns attached to the right of these rectangles in \eqref{eq:AB limit proof eq 2} can be explicitly evaluated as 
    \begin{multline}
        \label{eq:AB limit proof eq 3}
        \bra{\mu}B^{(N)}(x)\bigdot{D}^{(N)}(z)\ket{\nu} =  \prod_{k=\tau+1}^N \left[\frac{1-xy_k}{1-qxy_k}\frac{q(1-z/y_k)}{1-qz/y_k}\right] \bra{\mu}B^{(\tau)}(x)\bigdot{D}^{(\tau)}(z)\ket{\nu}\\
         +\prod_{k=\tau+1}^N\frac{1-xy_k}{1-qxy_k}\sum_{\ell=\tau+1}^N\frac{1-q}{1-qz/y_\ell}\frac{1-q}{1-zy_\ell} \prod_{k=\tau+1}^{\ell-1} \frac{1-qzy_k}{1-zy_k}\prod_{k=\ell+1}^N \frac{q(1-z/y_k)}{1-qz/y_k} \bra{\mu}B^{(\tau)}(x)\bigdot{A}^{(\tau)}(z)\ket{\nu}.
    \end{multline}
    Using condition \eqref{eq: AB limit condition}, we can then bound this term effectively as 
    \begin{multline}
        \label{eq:AB limit proof eq 4}
        \abs{\bra{\mu}B^{(N)}(x)\bigdot{D}^{(N)}(z)\ket{\nu}} \leq \rho^{N-\tau} \abs{\bra{\mu}B^{(\tau)}(x)\bigdot{D}^{(\tau)}(z)\ket{\nu}} \\ + \rho^{N-\tau-1} (N-\tau)\max_{\ell\in\{\tau+1,\dots,N\}}\left\{\abs{\frac{1-xy_\ell}{1-qxy_\ell}}\right\}\max_{\ell\in\{\tau+1,\dots,N\}}\left\{\abs{\frac{1-q}{1-qz/y_\ell}\frac{1-q}{1-zy_\ell}}\right\} \abs{\bra{\mu}B^{(\tau)}(x)\bigdot{A}^{(\tau)}(z)\ket{\nu}}.
    \end{multline}
    We note here that the conditions \eqref{AB operator commutation condition 1} imply that, for all $N>\tau$, the points $q x y_\ell,qz/y_\ell,zy_\ell$ are all bounded uniformly away 1 for all $\ell \in \{\tau+1,\dots, N\}$, and hence the maxima in the second term in \eqref{eq:AB limit proof eq 4} remain finite as $N\to\infty$, and therefore both terms in \eqref{eq:AB limit proof eq 4} vanish as $N\to\infty$.

    It remains to show that the third term in \eqref{eq:AB limit proof eq 1} vanishes. This follows similarly as
    \begin{equation}
        \bra{\mu}B^{(N)}(x)\bigdot{A}^{(N)}(z)\ket{\nu} = \prod_{k=\tau+1}^N \left[\frac{1-xy_k}{1-qxy_k}\frac{1-qzy_k}{1-zy_k}\right] \bra{\mu}B^{(\tau)}(x)\bigdot{A}^{(\tau)}(z)\ket{\nu},
    \end{equation}
    which can be bounded using condition \eqref{eq: AB limit condition} as
    \begin{equation}
        \abs{\bra{\mu}B^{(N)}(x)\bigdot{A}^{(N)}(z)\ket{\nu}} \leq \rho^{N-\tau} \abs{\bra{\mu}B^{(\tau)}(x)\bigdot{A}^{(\tau)}(z)\ket{\nu}}.
    \end{equation}
    The limit of this term vanishes also, so we may conclude that only the $p_1=p_2=0$ term remains in the limit $N\to \infty$ which implies the result.
\end{proof}

\begin{prop}
	\label{A commutation prop}
	Fix $x_1,x_2 \in \mathbb{C}$ and assume there exists $\rho>0$ such that
	\begin{equation}
		\label{A operator commutation condition}
		\abs{\frac{1-x_iy_k}{1-qx_iy_k}\frac{q(1-x_j/y_k)}{1-qx_j/y_k}}\leq \rho <1,
	\end{equation}
    for all $i\neq j$ and $k\in\mathbb{N}$. Then the double-row operators from \eqref{Row-operator A defn} commute:
	\begin{equation}
		\label{eq:Acommute}
		A(x_1)A(x_2) = A(x_2) A(x_1).
	\end{equation}
\end{prop}
\begin{proof}
    Let $\mu,\nu\in\mathbb{W}$ with maximum occupation $\tau=\max\{\mu_1,\nu_1\}$. Also let $N$ be an integer satisfying $N\geq\tau$ and consider the double row partition functions with $N$ columns 
    \begin{equation}
        \label{eq: AA inf commutation proof 1}
        \sum_{p\in\{0,1\}} \bra{\mu}\mathcal{T}_{0,p}^{(N)}(x_1)\mathcal{T}_{p,0}^{(N)}(x_2)\ket{\nu} R_{x_1 x_2}(0,p;p,0) = 
        \begin{tikzpicture}[baseline={([yshift=-.5ex]current bounding box.center)},scale=0.7]
    		\draw[lightgray,line width=1.5pt,->] (7,0) -- (1,0) -- (0,0.5) -- (1,1) -- (6,1) -- (7,2);
    		\draw[lightgray,line width=1.5pt,->] (7,0) -- (6.5,0);
    		\draw[lightgray,line width=1.5pt,->] (7,1) -- (6,2) -- (1,2) -- (0,2.5) -- (1,3) -- (7,3);
    		\draw[lightgray,line width=1.5pt,->] (7,1) -- (6.75,1.25);
    		\draw[blue,fill=blue] (0,0.5) circle (0.1cm);
    		\draw[blue,fill=blue] (0,2.5) circle (0.1cm);
    		\foreach \x in {2,...,5}
    		{\draw[lightgray,line width=1.5pt,->] (\x,-0.5) -- (\x,3.5);}
    		\node[right] at (7,0) {$0 \leftarrow x_1$};
    		\node[right] at (7,1) {$0 \leftarrow x_2$};
    		\node[right] at (7,2) {$0 \rightarrow x_1^{-1}$};
    		\node[right] at (7,3) {$0 \rightarrow x_2^{-1}$};
    
            \node[above] at (2,3.5) {$\eta_1^\nu$};
                \node[above] at (3,3.5) {$\cdots$};
                \node[above] at (4,3.5) {$\cdots$};
                \node[above] at (5,3.5) {$\eta_N^\nu$};
    
                \node[below] at (2,-0.5) {$\eta_1^\mu$};
                \node[below] at (3,-0.5) {$\cdots$};
                \node[below] at (4,-0.5) {$\cdots$};
                \node[below] at (5,-0.5) {$\eta_N^\mu$};
    	\end{tikzpicture},
    \end{equation}
    which, due to Proposition \ref{prop: AA limit condition}, has limit $\bra{\mu}A(x_1)A(x_2)\ket{\nu}$ as $N\to\infty$.

    We may then append an additional intertwining vertex to the lattice after the last column. The boundary conditions of this vertex mean that there is only one allowed vertex configuration on the intertwiner so that it can be added at no overall cost to the partition function. We have the \eqref{eq: AA inf commutation proof 1} is equal to 
    \begin{equation}
         \label{eq: AA inf commutation proof 2}
         \begin{tikzpicture}[baseline={([yshift=-.5ex]current bounding box.center)},scale=1]
    		\draw[lightgray,line width=1.5pt,->] (10,0) -- (1,0) -- (0,0.5) -- (1,1) -- (8,1) -- (10,3);
    		\draw[lightgray,line width=1.5pt,->] (10,0) -- (9.5,0);
    		\draw[lightgray,line width=1.5pt,->] (10,1) -- (9,1) -- (8,2) -- (1,2) -- (0,2.5) -- (1,3) -- (9,3) -- (10,2);
    		\draw[lightgray,line width=1.5pt,->] (10,1) -- (9.5,1);
    		\draw[blue,fill=blue] (0,0.5) circle (0.1cm);
    		\draw[blue,fill=blue] (0,2.5) circle (0.1cm);
    		\foreach \x in {2,...,7}
    		{\draw[lightgray,line width=1.5pt,->] (\x,-0.5) -- (\x,3.5);}
    		\node[right] at (10,0) {$0 \leftarrow x_1$};
    		\node[right] at (10,1) {$0 \leftarrow x_2$};
    		\node[right] at (10,2) {$0 \rightarrow x_2^{-1}$};
    		\node[right] at (10,3) {$0 \rightarrow x_1^{-1}$};
    
            \node[above] at (2,3.5) {$\eta_1^\nu$};
                \node[above] at (3,3.5) {$\eta_2^\nu$};
                \node[above] at (4,3.5) {$\eta_3^\nu$};
                \node[above] at (5,3.5) {$\cdots$};
                \node[above] at (6,3.5) {$\cdots$};
                \node[above] at (7,3.5) {$\eta_N^\nu$};
    
                \node[below] at (2,-0.5) {$\eta_1^\mu$};
                \node[below] at (3,-0.5) {$\eta_2^\mu$};
                \node[below] at (4,-0.5) {$\eta_3^\mu$};
                \node[below] at (5,-0.5) {$\cdots$};
                \node[below] at (6,-0.5) {$\cdots$};
                \node[below] at (7,-0.5) {$\eta_N^\mu$};
    	\end{tikzpicture}
    \end{equation}
    We may then repeatedly apply the Yang\textendash Baxter equation \eqref{Yang-Baxter eq R-matrices} to manipulate the diagram. This equation can be applied successively to each column of partition function leading to the relation
	\[ 
	\begin{tikzpicture}[baseline={([yshift=-.5ex]current bounding box.center)},scale=0.7]
		\draw[lightgray,line width=1.5pt,->] (7,0) -- (1,0) -- (0,0.5) -- (1.5,2) -- (6,2) -- (7,3);
		\draw[lightgray,line width=1.5pt,->] (7,0) -- (6.5,0);
		\draw[lightgray,line width=1.5pt,->] (7,1) -- (1.5,1) -- (0,2.5) -- (1,3) -- (6,3) -- (7,2);
		\draw[lightgray,line width=1.5pt,->] (7,1) -- (6.5,1);
		\draw[blue,fill=blue] (0,0.5) circle (0.1cm);
		\draw[blue,fill=blue] (0,2.5) circle (0.1cm);
		\foreach \x in {2,...,5}
		{\draw[lightgray,line width=1.5pt,->] (\x,-0.5) -- (\x,3.5);}
		\node[right] at (7,0) {$0 \leftarrow x_1$};
		\node[right] at (7,1) {$0 \leftarrow x_2$};
		\node[right] at (7,2) {$0 \rightarrow x_2^{-1}$};
		\node[right] at (7,3) {$0 \rightarrow x_1^{-1}$};
        \node[above] at (2,3.5) {$\eta_1^\nu$};
                \node[above] at (3,3.5) {$\cdots$};
                \node[above] at (4,3.5) {$\cdots$};
                \node[above] at (5,3.5) {$\eta_N^\nu$};
                \node[below] at (2,-0.5) {$\eta_1^\mu$};
                \node[below] at (3,-0.5) {$\cdots$};
                \node[below] at (4,-0.5) {$\cdots$};
                \node[below] at (5,-0.5) {$\eta_N^\mu$};
	\end{tikzpicture}
	= \begin{tikzpicture}[baseline={([yshift=-.5ex]current bounding box.center)},scale=0.7]
		\draw[lightgray,line width=1.5pt,->] (7,0) -- (1,0) -- (0,0.5) -- (2,3) -- (7,3);
		\draw[lightgray,line width=1.5pt,->] (7,0) -- (6.5,0);
		\draw[lightgray,line width=1.5pt,->] (7,1) -- (1.5,1) -- (0,2.5) -- (1,3) -- (2,2) -- (7,2);
		\draw[lightgray,line width=1.5pt,->] (7,1) -- (6.5,1);
		\draw[blue,fill=blue] (0,0.5) circle (0.1cm);
		\draw[blue,fill=blue] (0,2.5) circle (0.1cm);
		\foreach \x in {3,...,6}
		{\draw[lightgray,line width=1.5pt,->] (\x,-0.5) -- (\x,3.5);}
		\node[right] at (7,0) {$0 \leftarrow x_1$};
		\node[right] at (7,1) {$0 \leftarrow x_2$};
		\node[right] at (7,2) {$0 \rightarrow x_2^{-1}$};
		\node[right] at (7,3) {$0 \rightarrow x_1^{-1}$};
        \node[above] at (2,3.5) {$\eta_1^\nu$};
                \node[above] at (3,3.5) {$\cdots$};
                \node[above] at (4,3.5) {$\cdots$};
                \node[above] at (5,3.5) {$\eta_N^\nu$};
                \node[below] at (2,-0.5) {$\eta_1^\mu$};
                \node[below] at (3,-0.5) {$\cdots$};
                \node[below] at (4,-0.5) {$\cdots$};
                \node[below] at (5,-0.5) {$\eta_N^\mu$};
	\end{tikzpicture}.
	\]
	At this point the reflection equation \eqref{Sklyanin reflection equation R,K matrices} can be applied, followed by the Yang\textendash Baxter equation to push the intertwining vertices back to the right edge of the partition function. This yields 
	\[
	\begin{tikzpicture}[baseline={([yshift=-.5ex]current bounding box.center)},scale=0.7]
		\draw[lightgray,line width=1.5pt,->] (7,0) -- (2,0) -- (0,2.5) -- (1,3) -- (7,3);
		\draw[lightgray,line width=1.5pt,->] (7,0) -- (6.5,0);
		\draw[lightgray,line width=1.5pt,->] (7,1) -- (2,1) -- (1,0) -- (0,0.5) -- (1.5,2) -- (7,2);
		\draw[lightgray,line width=1.5pt,->] (7,1) -- (6.5,1);
		\draw[blue,fill=blue] (0,0.5) circle (0.1cm);
		\draw[blue,fill=blue] (0,2.5) circle (0.1cm);
		\foreach \x in {3,...,6}
		{\draw[lightgray,line width=1.5pt,->] (\x,-0.5) -- (\x,3.5);}
		\node[right] at (7,0) {$0 \leftarrow x_1$};
		\node[right] at (7,1) {$0 \leftarrow x_2$};
		\node[right] at (7,2) {$0 \rightarrow x_2^{-1}$};
		\node[right] at (7,3) {$0 \rightarrow x_1^{-1}$};
        \node[above] at (2,3.5) {$\eta_1^\nu$};
                \node[above] at (3,3.5) {$\cdots$};
                \node[above] at (4,3.5) {$\cdots$};
                \node[above] at (5,3.5) {$\eta_N^\nu$};
                \node[below] at (2,-0.5) {$\eta_1^\mu$};
                \node[below] at (3,-0.5) {$\cdots$};
                \node[below] at (4,-0.5) {$\cdots$};
                \node[below] at (5,-0.5) {$\eta_N^\mu$};
	\end{tikzpicture}
	= \begin{tikzpicture}[baseline={([yshift=-.5ex]current bounding box.center)},scale=0.7]
		\draw[lightgray,line width=1.5pt,->] (7,0) -- (6.5,0) -- (4.5,2) -- (0.5,2) -- (0,2.5) -- (0.5,3) -- (7,3);
		\draw[lightgray,line width=1.5pt,->] (7,0) -- (6.65,0);
		\draw[lightgray,line width=1.5pt,->] (7,1) -- (6.5,1) -- (5.5,0) -- (0.5,0) -- (0,0.5) -- (0.5,1) -- (4.5,1) -- (5.5,2) -- (7,2);
		\draw[lightgray,line width=1.5pt,->] (7,1) -- (6.65,1);
		\draw[blue,fill=blue] (0,0.5) circle (0.1cm);
		\draw[blue,fill=blue] (0,2.5) circle (0.1cm);
		\foreach \x in {1,...,4}
		{\draw[lightgray,line width=1.5pt,->] (\x,-0.5) -- (\x,3.5);}
		\node[right] at (7,0) {$0 \leftarrow x_1$};
		\node[right] at (7,1) {$0 \leftarrow x_2$};
		\node[right] at (7,2) {$0 \rightarrow x_2^{-1}$};
		\node[right] at (7,3) {$0 \rightarrow x_1^{-1}$};
        \node[above] at (2,3.5) {$\eta_1^\nu$};
                \node[above] at (3,3.5) {$\cdots$};
                \node[above] at (4,3.5) {$\cdots$};
                \node[above] at (5,3.5) {$\eta_N^\nu$};
                \node[below] at (2,-0.5) {$\eta_1^\mu$};
                \node[below] at (3,-0.5) {$\cdots$};
                \node[below] at (4,-0.5) {$\cdots$};
                \node[below] at (5,-0.5) {$\eta_N^\mu$};
	\end{tikzpicture}.
	\]
    At this point, the intertwiner at the bottom-right of the diagram can be removed at no cost the partition function due to the boundary conditions. This yields
    \begin{equation}
    \label{eq: AA inf commutation proof 3}
        \begin{tikzpicture}[baseline={([yshift=-.5ex]current bounding box.center)},scale=1]
    		\draw[lightgray,line width=1.5pt,->] (7,0) -- (1,0) -- (0,0.5) -- (1,1) -- (6,1) -- (7,2);
    		\draw[lightgray,line width=1.5pt,->] (7,0) -- (6.5,0);
    		\draw[lightgray,line width=1.5pt,->] (7,1) -- (6,2) -- (1,2) -- (0,2.5) -- (1,3) -- (7,3);
    		\draw[lightgray,line width=1.5pt,->] (7,1) -- (6.75,1.25);
    		\draw[blue,fill=blue] (0,0.5) circle (0.1cm);
    		\draw[blue,fill=blue] (0,2.5) circle (0.1cm);
    		\foreach \x in {2,...,5}
    		{\draw[lightgray,line width=1.5pt,->] (\x,-0.5) -- (\x,3.5);}
    		\node[right] at (7,0) {$0 \leftarrow x_2$};
    		\node[right] at (7,1) {$0 \leftarrow x_1$};
    		\node[right] at (7,2) {$0 \rightarrow x_2^{-1}$};
    		\node[right] at (7,3) {$0 \rightarrow x_1^{-1}$};
    
            \node[above] at (2,3.5) {$\eta_1^\nu$};
                \node[above] at (3,3.5) {$\cdots$};
                \node[above] at (4,3.5) {$\cdots$};
                \node[above] at (5,3.5) {$\eta_N^\nu$};
    
                \node[below] at (2,-0.5) {$\eta_1^\mu$};
                \node[below] at (3,-0.5) {$\cdots$};
                \node[below] at (4,-0.5) {$\cdots$};
                \node[below] at (5,-0.5) {$\eta_N^\mu$};
    	\end{tikzpicture},
    \end{equation}
    which we recognize as \eqref{eq: AA inf commutation proof 1} with $x_1$ and $x_2$ interchanged. The limit of \eqref{eq: AA inf commutation proof 3} can then be evaluated as $\bra{\mu}A(x_2)A(x_1)\ket{\nu}$ as $N\to\infty$ due to condition \eqref{A operator commutation condition}. Since \eqref{eq: AA inf commutation proof 1} and \eqref{eq: AA inf commutation proof 3} are equal for all $N\geq\tau$, we can conclude that their limits must be equal. This is the result \eqref{eq:Acommute}.
\end{proof}

\begin{prop}
	\label{B commutation prop}
	Given configurations $\mu,\nu\in\mathbb{W}$ let $N$ be an integer $N\geq\max\{\mu_1,\nu_1\}$. For $z_1,z_2$ the double-row operators  with $N$ columns commute
        \begin{equation}
            \label{eq:Bcommute finite}
            \bra{\mu}\bigdot{B}^{(N)}(z_1)\bigdot{B}^{(N)}(z_2)\ket{\nu} = \bra{\mu}\bigdot{B}^{(N)}(z_2)\bigdot{B}^{(N)}(z_1)\ket{\nu}.
        \end{equation}
        This can be extended to the case of infinite columns to obtain the commutation relation of the double-row operators \eqref{Row-operator B defn}
	\begin{equation}
		\label{eq:Bcommute}
		\bigdot{B}(z_1)\bigdot{B}(z_2) = \bigdot{B}(z_2)\bigdot{B}(z_1).
	\end{equation}
\end{prop}
\begin{proof}
    Let $\mu,\nu\in\mathbb{W}$ with maximum occupation $\tau=\max\{\mu_1,\nu_1\}$. Also let $N\geq\tau$ and consider the following double row partition functions with $N$ columns
    \begin{equation}
        \label{eq:BB commutation proof 1}
         f^{(N)} (z_1,z_2) :=
         \begin{tikzpicture}[baseline={([yshift=-.5ex]current bounding box.center)},scale=0.8]
			\draw[lightgray,line width=1.5pt,->] (8,0) -- (1,0) -- (0,0.5) -- (1,1) -- (6,1) -- (8,3);
			\draw[lightgray,line width=1.5pt,->] (8,0) -- (7.5,0);
			\draw[lightgray,line width=1.5pt,->] (8,1) -- (7,1) -- (6,2) -- (1,2) -- (0,2.5) -- (1,3) -- (7,3) -- (8,2);
			\draw[lightgray,line width=1.5pt,->] (8,1) -- (7.5,1);
			\draw[blue,fill=blue] (0,0.5) circle (0.1cm);
			\draw[blue,fill=blue] (0,2.5) circle (0.1cm);
			\foreach \x in {2,...,5}
			{\draw[lightgray,line width=1.5pt,->] (\x,-0.5) -- (\x,3.5);
			\draw[black,fill=black] (\x,3) circle (0.1cm);
            \draw[black,fill=black] (\x,1) circle (0.1cm);}
			\node[right] at (8,0) {$0 \leftarrow z_1$};
			\node[right] at (8,1) {$0 \leftarrow z_2$};
			\node[right] at (8,2) {$1 \rightarrow z_2^{-1}$};
			\node[right] at (8,3) {$1 \rightarrow z_1^{-1}$};
            \node[above] at (2,3.5) {$\eta_1^\nu$};
            \node[above] at (3,3.5) {$\cdots$};
            \node[above] at (4,3.5) {$\cdots$};
            \node[above] at (5,3.5) {$\eta_N^\nu$};
            \node[below] at (2,-0.5) {$\eta_1^\mu$};
            \node[below] at (3,-0.5) {$\cdots$};
            \node[below] at (4,-0.5) {$\cdots$};
            \node[below] at (5,-0.5) {$\eta_N^\mu$};
		\end{tikzpicture},
     \end{equation}
    with two intertwining vertices appended to the right of the diagram. We note that these intertwiners are frozen in their own right and can simply be evaluated. For any $N\geq\tau$ this yields
    \begin{equation}
        \label{eq:BB commutation proof 2}
        f^{(N)}(z_1,z_2) = \frac{1-z_1z_2}{1-qz_1z_2} \bra{\mu} \bigdot{B}^{(N)}(z_1)\bigdot{B}^{(N)}(z_2)\ket{\nu}.
    \end{equation}
     It is important to note here that this holds for all $N\geq \tau$ here rather than in just under the large $N$ limit as with the proof of Proposition \ref{A commutation prop}.

     Then following the same procedure as the proof of Proposition \ref{A commutation prop}, we may apply the Yang\textendash Baxter equation \eqref{Yang-Baxter eq R-matrices} and reflection equation \eqref{Sklyanin reflection equation R,K matrices} to manipulate the diagram \eqref{eq:BB commutation proof 1} to obtain
     \begin{equation}
        \label{eq:BB commutation proof 3}
         f^{(N)}(z_1,z_2) = \begin{tikzpicture}[baseline={([yshift=-.5ex]current bounding box.center)},scale=0.8]
    		\draw[lightgray,line width=1.5pt,->] (7,0) -- (6.5,0) -- (4.5,2) -- (0.5,2) -- (0,2.5) -- (0.5,3) -- (7,3);
    		\draw[lightgray,line width=1.5pt,->] (7,0) -- (6.65,0);
    		\draw[lightgray,line width=1.5pt,->] (7,1) -- (6.5,1) -- (5.5,0) -- (0.5,0) -- (0,0.5) -- (0.5,1) -- (4.5,1) -- (5.5,2) -- (7,2);
    		\draw[lightgray,line width=1.5pt,->] (7,1) -- (6.65,1);
    		\draw[blue,fill=blue] (0,0.5) circle (0.1cm);
    		\draw[blue,fill=blue] (0,2.5) circle (0.1cm);
    		\foreach \x in {1,...,4}
    		{\draw[lightgray,line width=1.5pt,->] (\x,-0.5) -- (\x,3.5);
            \draw[black,fill=black] (\x,3) circle (0.1cm);
            \draw[black,fill=black] (\x,1) circle (0.1cm);}
    		\node[right] at (7,0) {$0 \leftarrow z_1$};
    		\node[right] at (7,1) {$0 \leftarrow z_2$};
    		\node[right] at (7,2) {$1 \rightarrow z_2^{-1}$};
    		\node[right] at (7,3) {$1 \rightarrow z_1^{-1}$};
            \node[above] at (1,3.5) {$\eta_1^\nu$};
                \node[above] at (2,3.5) {$\cdots$};
                \node[above] at (3,3.5) {$\cdots$};
                \node[above] at (4,3.5) {$\eta_N^\nu$};
                \node[below] at (1,-0.5) {$\eta_1^\mu$};
                \node[below] at (2,-0.5) {$\cdots$};
                \node[below] at (3,-0.5) {$\cdots$};
                \node[below] at (4,-0.5) {$\eta_N^\mu$};
    	\end{tikzpicture}.
    \end{equation}
    The intertwiners on the right side of \eqref{eq:BB commutation proof 3} are also frozen in their own right and can be evaluated as
    \begin{equation}
        \label{eq:BB commutation proof 4}
        f^{(N)}(z_1,z_2) = \frac{1-z_1z_2}{1-qz_1z_2} \bra{\mu} \bigdot{B}^{(N)}(z_2)\bigdot{B}^{(N)}(z_1)\ket{\nu}.
    \end{equation}
    The result \eqref{eq:Bcommute finite} is obtained by comparing \eqref{eq:BB commutation proof 2} and \eqref{eq:BB commutation proof 4}. Taking the large $N$ limit yields the result \eqref{eq:Bcommute}.
\end{proof}

\begin{prop}
	\label{AB commutation relation prop}
	Fix $x,z\in\mathbb{C}$ and assume there exists $\rho>0$ such that 
    \begin{equation}
		\label{AB operator commutation condition}
		\abs{\frac{1-xy_k}{1-qxy_k}\frac{q(1-z/y_k)}{1-qz/y_k}}\le\rho<1, \hspace{0.5cm} \abs{\frac{1-xy_k}{1-qxy_k}\frac{1-qzy_k}{1-zy_k}}\leq\rho<1,
	\end{equation}
    for all $k\in \mathbb{N}$. Then the double-row operators from Definition \ref{AB row operator defn} obey the exchange relation
	\begin{equation}
        \label{eq: AB exchangr rel}
		A(x) \bigdot{B}(z) = \frac{x-qz}{x-z}\frac{1-xz}{1-qxz} \bigdot{B}(z) A(x).
	\end{equation}	
\end{prop}
\begin{proof}
     Let $\mu,\nu\in\mathbb{W}$ with maximum occupation $\tau=\max\{\mu_1,\nu_1\}$. Also let $N$ be an integer satisfying $N\geq\tau$ and consider the double row partition functions with $N$ columns
     \begin{multline}
        \label{eq:AB commutation proof 1}
         \sum_{p_1,p_2\in\{0,1\} \atop p_1\geq p_2} \bra{\mu}\mathcal{T}_{0,p_1}^{(N)}(x)\bigdot{\mathcal{T}}_{p_1-p_2,1-p_2}^{(N)}(z) \ket{\nu} R_{xz}(0,p_1;p_1-p_2,p_2) R_{z/x}(p_2,1-p_2;0,1)\\
         =\begin{tikzpicture}[baseline={([yshift=-.5ex]current bounding box.center)},scale=0.8]
			\draw[lightgray,line width=1.5pt,->] (8,0) -- (1,0) -- (0,0.5) -- (1,1) -- (6,1) -- (8,3);
			\draw[lightgray,line width=1.5pt,->] (8,0) -- (7.5,0);
			\draw[lightgray,line width=1.5pt,->] (8,1) -- (7,1) -- (6,2) -- (1,2) -- (0,2.5) -- (1,3) -- (7,3) -- (8,2);
			\draw[lightgray,line width=1.5pt,->] (8,1) -- (7.5,1);
			\draw[blue,fill=blue] (0,0.5) circle (0.1cm);
			\draw[blue,fill=blue] (0,2.5) circle (0.1cm);
			\foreach \x in {2,...,5}
			{\draw[lightgray,line width=1.5pt,->] (\x,-0.5) -- (\x,3.5);
			\draw[black,fill=black] (\x,3) circle (0.1cm);}
			\node[right] at (8,0) {$0 \leftarrow x$};
			\node[right] at (8,1) {$0 \leftarrow z$};
			\node[right] at (8,2) {$1 \rightarrow z^{-1}$};
			\node[right] at (8,3) {$0 \rightarrow x^{-1}$};
            \node[above] at (2,3.5) {$\eta_1^\nu$};
            \node[above] at (3,3.5) {$\cdots$};
            \node[above] at (4,3.5) {$\cdots$};
            \node[above] at (5,3.5) {$\eta_N^\nu$};
            \node[below] at (2,-0.5) {$\eta_1^\mu$};
            \node[below] at (3,-0.5) {$\cdots$};
            \node[below] at (4,-0.5) {$\cdots$};
            \node[below] at (5,-0.5) {$\eta_N^\mu$};
		\end{tikzpicture} =: f^{(N)} (x,z).
     \end{multline}
    Due to condition \eqref{AB operator commutation condition} and Proposition \ref{prop: AB limit condition}, the limit of \eqref{eq:AB commutation proof 1} is 
    \begin{equation}
        \label{eq:AB commutation proof 2}
        \lim_{N\to\infty} f^{(N)}(x,z) = \frac{x-z}{x-qz}\bra{\mu} A(x)\bigdot{B}(z) \ket{\nu}
    \end{equation}
    We may manipulate the diagram of \eqref{eq:AB commutation proof 1} to obtain an exchange relation in a similar way to the proof of Proposition \ref{A commutation prop}.

    By following the same steps as in the proof of Proposition \ref{A commutation prop}, we may arrive at the following diagram which is equal to \eqref{eq:AB commutation proof 1} as a partition function
    \begin{equation}
        \label{eq:AB commutation proof 3}
        f^{(N)}(x,z) = \begin{tikzpicture}[baseline={([yshift=-.5ex]current bounding box.center)},scale=0.8]
		\draw[lightgray,line width=1.5pt,->] (7,0) -- (6.5,0) -- (4.5,2) -- (0.5,2) -- (0,2.5) -- (0.5,3) -- (7,3);
		\draw[lightgray,line width=1.5pt,->] (7,0) -- (6.65,0);
		\draw[lightgray,line width=1.5pt,->] (7,1) -- (6.5,1) -- (5.5,0) -- (0.5,0) -- (0,0.5) -- (0.5,1) -- (4.5,1) -- (5.5,2) -- (7,2);
		\draw[lightgray,line width=1.5pt,->] (7,1) -- (6.65,1);
		\draw[blue,fill=blue] (0,0.5) circle (0.1cm);
		\draw[blue,fill=blue] (0,2.5) circle (0.1cm);
		\foreach \x in {1,...,4}
		{\draw[lightgray,line width=1.5pt,->] (\x,-0.5) -- (\x,3.5);
        \draw[black,fill=black] (\x,1) circle (0.1cm);}
		\node[right] at (7,0) {$0 \leftarrow x$};
		\node[right] at (7,1) {$0 \leftarrow z$};
		\node[right] at (7,2) {$1 \rightarrow z^{-1}$};
		\node[right] at (7,3) {$0 \rightarrow x^{-1}$};
        \node[above] at (1,3.5) {$\eta_1^\nu$};
            \node[above] at (2,3.5) {$\cdots$};
            \node[above] at (3,3.5) {$\cdots$};
            \node[above] at (4,3.5) {$\eta_N^\nu$};
            \node[below] at (1,-0.5) {$\eta_1^\mu$};
            \node[below] at (2,-0.5) {$\cdots$};
            \node[below] at (3,-0.5) {$\cdots$};
            \node[below] at (4,-0.5) {$\eta_N^\mu$};
	\end{tikzpicture}
    = \frac{1-xz}{1-qxz} \begin{tikzpicture}[baseline={([yshift=-.5ex]current bounding box.center)},scale=0.8]
		\draw[lightgray,line width=1.5pt,->] (5,2) -- (0.5,2) -- (0,2.5) -- (0.5,3) -- (5,3);
		\draw[lightgray,line width=1.5pt,->] (5,2) -- (4.5,2);
		\draw[lightgray,line width=1.5pt,->] (5,0) -- (4.5,0) -- (0.5,0) -- (0,0.5) -- (0.5,1) -- (5,1);
		\draw[lightgray,line width=1.5pt,->] (5,0) -- (4.5,0);
		\draw[blue,fill=blue] (0,0.5) circle (0.1cm);
		\draw[blue,fill=blue] (0,2.5) circle (0.1cm);
		\foreach \x in {1,...,4}
		{\draw[lightgray,line width=1.5pt,->] (\x,-0.5) -- (\x,3.5);
        \draw[black,fill=black] (\x,1) circle (0.1cm);}
		\node[right] at (5,0) {$0 \leftarrow z$};
		\node[right] at (5,1) {$1 \rightarrow z^{-1}$};
		\node[right] at (5,2) {$0 \leftarrow x$};
		\node[right] at (5,3) {$0 \rightarrow x^{-1}$};
        \node[above] at (1,3.5) {$\eta_1^\nu$};
            \node[above] at (2,3.5) {$\cdots$};
            \node[above] at (3,3.5) {$\cdots$};
            \node[above] at (4,3.5) {$\eta_N^\nu$};
            \node[below] at (1,-0.5) {$\eta_1^\mu$};
            \node[below] at (2,-0.5) {$\cdots$};
            \node[below] at (3,-0.5) {$\cdots$};
            \node[below] at (4,-0.5) {$\eta_N^\mu$};
	\end{tikzpicture}.
    \end{equation}
    The last equality follows from noticing that the intertwining vertices are both frozen by the boundary conditions, so they may be evaluated as a factor and removed from the diagram. The limit of the right diagram yields
    \[\lim_{N\to\infty} f^{(N)}(x,z) = \frac{1-xz}{1-qxz} \bra{\mu} \bigdot{B}(z) A(x) \ket{\nu},\]
    which can be combined with \eqref{eq:AB commutation proof 2} to give the result.
\end{proof}

\subsection{Multi-parameter symmetric functions}
We will now define a partition function which will be central to much the remainder of this work.
\begin{defn}
	\label{G partition function defn}
	Fix two alphabets $(x_1,\dots,x_L)$ and $(z_1,\dots,z_M)$, and configurations $\mu,\nu\in\mathbb{W}$. We define 
	\begin{align}
		\label{G partition function defn A eq}
		G_{\nu/\mu} (x_1,\dots,x_L|\yalph) & = \bra{\mu}A(x_1|\yalph) \cdots A(x_L|\yalph) \ket{\nu},\\
        \label{F partition function defn B eq}
		F_{\mu/\nu} (z_1,\dots,z_M|\yalph) & = \bra{\mu}\bigdot{B}(z_1|\yalph) \cdots \bigdot{B}(z_M|\yalph) \ket{\nu}.
	\end{align}
	The functions \eqref{G partition function defn A eq} and \eqref{F partition function defn B eq} can be represented diagrammatically by stacking double-row operators \eqref{Row-operator A defn} and \eqref{Row-operator B defn} appropriately. We find
	\begin{align}
		\label{G partition function defn picture}
		G_{\nu/\mu} (x_1,\dots,x_L|\yalph) & = 
		\begin{tikzpicture}[baseline={([yshift=-.5ex]current bounding box.center)},scale=0.8]
			\draw[lightgray,line width=1.5pt,->] (8,0) -- (1,0) -- (0,0.5) -- (1,1) -- (8,1);
			\draw[lightgray,line width=1.5pt,->] (8,0) -- (7.5,0);
			\draw[lightgray,line width=1.5pt,->] (8,2) -- (1,2) -- (0,2.5) -- (1,3) -- (8,3);
			\draw[lightgray,line width=1.5pt,->] (8,2) -- (7.5,2);
			\draw[lightgray,line width=1.5pt,->] (8,4) -- (1,4) -- (0,4.5) -- (1,5) -- (8,5);
			\draw[lightgray,line width=1.5pt,->] (8,4) -- (7.5,4);
			\draw[lightgray,line width=1.5pt,->] (8,6) -- (1,6) -- (0,6.5) -- (1,7) -- (8,7);
			\draw[lightgray,line width=1.5pt,->] (8,6) -- (7.5,6);
			\draw[blue,fill=blue] (0,0.5) circle (0.1cm);
			\draw[blue,fill=blue] (0,2.5) circle (0.1cm);
			\draw[blue,fill=blue] (0,4.5) circle (0.1cm);
			\draw[blue,fill=blue] (0,6.5) circle (0.1cm);
			\foreach \x in {2,...,7}
			{\draw[lightgray,line width=1.5pt,->] (\x,-0.5) -- (\x,7.5);}
			\node[right] at (8,0) {$0 \leftarrow x_1$};
			\node[right] at (8,1) {$0 \rightarrow x_1^{-1}$};
			\node[right] at (8,2.5) {$\vdots$};
			\node[right] at (8,4.5) {$\vdots$};
			\node[right] at (8,6) {$0 \leftarrow x_L$};
			\node[right] at (8,7) {$0 \rightarrow x_L^{-1}$};
			\node[above] at (2,7.5) {$\eta^\nu_1$};
			\node[above] at (3,7.5) {$\eta^\nu_2$};
			\node[above] at (4,7.5) {$\eta^\nu_3$};
			\node[above] at (5,7.5) {$\cdots$};
			\node[above] at (6,7.5) {$\cdots$};
			\node[below] at (2,-0.5) {$\eta^\mu_1$};
			\node[below] at (3,-0.5) {$\eta^\mu_2$};
			\node[below] at (4,-0.5) {$\eta^\mu_3$};
			\node[below] at (5,-0.5) {$\cdots$};
			\node[below] at (6,-0.5) {$\cdots$};
			\node[below] at (7,-0.5) {$\cdots$};
			\node[below] at (2,-1.2) {$\uparrow$}; 
			\node[below] at (3,-1.2) {$\uparrow$};
			\node[below] at (4,-1.2) {$\uparrow$};  
			\node[below] at (2,-1.8) {$y_1$}; 
			\node[below] at (3,-1.8) {$y_2$};
			\node[below] at (4,-1.8) {$y_3$}; 
		\end{tikzpicture},
	\end{align}
	and
    \begin{align}
		\label{F partition function defn picture}
		F_{\mu/\nu} (z_1,\dots,z_M|\yalph) & = 
		\begin{tikzpicture}[baseline={([yshift=-.5ex]current bounding box.center)},scale=0.8]
			\draw[lightgray,line width=1.5pt,->] (8,0) -- (1,0) -- (0,0.5) -- (1,1) -- (8,1);
			\draw[lightgray,line width=1.5pt,->] (8,0) -- (7.5,0);
			\draw[lightgray,line width=1.5pt,->] (8,2) -- (1,2) -- (0,2.5) -- (1,3) -- (8,3);
			\draw[lightgray,line width=1.5pt,->] (8,2) -- (7.5,2);
			\draw[lightgray,line width=1.5pt,->] (8,4) -- (1,4) -- (0,4.5) -- (1,5) -- (8,5);
			\draw[lightgray,line width=1.5pt,->] (8,4) -- (7.5,4);
			\draw[lightgray,line width=1.5pt,->] (8,6) -- (1,6) -- (0,6.5) -- (1,7) -- (8,7);
			\draw[lightgray,line width=1.5pt,->] (8,6) -- (7.5,6);
			\draw[blue,fill=blue] (0,0.5) circle (0.1cm);
			\draw[blue,fill=blue] (0,2.5) circle (0.1cm);
			\draw[blue,fill=blue] (0,4.5) circle (0.1cm);
			\draw[blue,fill=blue] (0,6.5) circle (0.1cm);
			\foreach \x in {2,...,7}
			{\draw[lightgray,line width=1.5pt,->] (\x,-0.5) -- (\x,7.5);
            \draw[black,fill=black] (\x,1) circle (0.1cm);
            \draw[black,fill=black] (\x,3) circle (0.1cm);
            \draw[black,fill=black] (\x,5) circle (0.1cm);
            \draw[black,fill=black] (\x,7) circle (0.1cm);}
			\node[right] at (8,0) {$0 \leftarrow z_1$};
			\node[right] at (8,1) {$1 \rightarrow z_1^{-1}$};
			\node[right] at (8,2.5) {$\vdots$};
			\node[right] at (8,4.5) {$\vdots$};
			\node[right] at (8,6) {$0 \leftarrow z_M$};
			\node[right] at (8,7) {$1 \rightarrow z_M^{-1}$};
			\node[above] at (2,7.5) {$\eta^\nu_1$};
			\node[above] at (3,7.5) {$\eta^\nu_2$};
			\node[above] at (4,7.5) {$\eta^\nu_3$};
			\node[above] at (5,7.5) {$\cdots$};
			\node[above] at (6,7.5) {$\cdots$};
			\node[below] at (2,-0.5) {$\eta^\mu_1$};
			\node[below] at (3,-0.5) {$\eta^\mu_2$};
			\node[below] at (4,-0.5) {$\eta^\mu_3$};
			\node[below] at (5,-0.5) {$\cdots$};
			\node[below] at (6,-0.5) {$\cdots$};
			\node[below] at (7,-0.5) {$\cdots$};
			\node[below] at (2,-1.2) {$\uparrow$}; 
			\node[below] at (3,-1.2) {$\uparrow$};
			\node[below] at (4,-1.2) {$\uparrow$};  
			\node[below] at (2,-1.8) {$y_1$}; 
			\node[below] at (3,-1.8) {$y_2$};
			\node[below] at (4,-1.8) {$y_3$}; 
		\end{tikzpicture}.
	\end{align}
 
 We will also, where convenient, omit the family of parameters $\yalph$ from our notation. 
\end{defn}
The primary focus for the remainder of this work will be the partition function depicted in \eqref{G partition function defn picture}. This will ultimately be shown to reduce to describing the behaviour of the ASEP on the half-line with generic open boundary conditions.

In many cases we will be interested in the partition function \eqref{G partition function defn picture} whose bottom state is empty, so that $\mu = \emptyset$. While for the partition function \eqref{F partition function defn picture} we will often be interested in cases when the top state is empty, so that $\nu=\emptyset$. In such cases we will write 
\begin{align*}
G_{\nu/\emptyset}(x_1,\dots,x_L|\yalph) & =: G_\nu(x_1,\dots,x_L) \\
F_{\mu/\emptyset}(z_1,\dots,z_M|\yalph) & =: F_\mu(z_1,\dots,z_M) \\
\end{align*}
\begin{cor}[of Propositions \ref{A commutation prop} and \ref{B commutation prop}]
	\label{G symmetric cor}
	Fix $\mu,\nu\in\mathbb{W}$. Given parameters $x_1,\dots,x_L\in\mathbb{C}$ and $\yalph\in\mathbb{C}^\mathbb{N}$ all satisfying the conditions \eqref{A operator commutation condition}, the partition function $G_{\nu/\mu}(x_1,\dots,x_L|\yalph)$ is symmetric under permuting its $x$-alphabet. Given $z_1,\dots,z_M\in\mathbb{C}$, the partition function $F_{\mu/\nu}(z_1,\dots,z_M|\yalph)$ is symmetric in permuting its $z$-alphabet. 
\end{cor}
\begin{prop}[Branching relations]
    \label{branching relations prop}
	The partition functions from Definition \ref{G partition function defn} obey the branching relations
	\begin{align}
        \label{G branching rel}
		G_{\nu/\mu}(x_1,\dots,x_{L+M}) & = \sum_{\kappa} G_{\nu/\kappa}(x_{M+1},\dots,x_{L+M}) G_{\kappa/\mu}(x_{1},\dots,x_{M}), \\
		F_{\mu/\nu}(z_1,\dots,z_{L+M}) & = \sum_{\lambda} F_{\lambda/\nu}(z_{M+1},\dots,z_{L+M}) F_{\mu/\lambda}(z_{1},\dots,z_{M}).
	\end{align}
\end{prop}
\begin{proof}
	This can be seen by inserting a sum over a complete set of states between the double-row operators in \eqref{G partition function defn A eq} and \eqref{F partition function defn B eq}.
\end{proof}
\begin{prop}
    \label{prop: G stocahsticity}
    Fix a configuration $\mu\in\mathbb{W}$. The partition function \eqref{G partition function defn A eq} obeys the sum-to-unity property
    \begin{equation}
        \sum_{\nu} G_{\nu/\mu} (x_1,\dots,x_L) = 1.
    \end{equation}
\end{prop}
\begin{proof}
    This follows from the stochasticity of the bulk and boundary vertices, Propositions \ref{prop: stochasticity bulk} and \ref{prop: stochasticity boundary}.
\end{proof}

\subsection{Cauchy summation identity}
In this section, we use the exchange relation \eqref{eq: AB exchangr rel} to prove an infinite summation identity of Cauchy type between the functions \eqref{G partition function defn A eq} and \eqref{F partition function defn B eq}. This identity is the hint of a deeper orthogonality theory behind these functions that we plan to explore in a future text.
\begin{thm}
	\label{skew Cauchy identity thm}
	Fix the alphabets $(x_1\dots,x_L),(z_1,\dots,z_M)$ and assume there exists $\rho>0$ such that 
	\begin{equation}
		\label{Cauchy identity condition}
		\abs{\frac{1-x_iy_k}{1-qx_iy_k}\frac{q(1-z_j/y_k)}{1-qz_j/y_k}}\leq\rho<1, \hspace{0.5cm} \abs{\frac{1-x_iy_k}{1-qx_iy_k}\frac{1-qz_jy_k}{1-z_jy_k}}\leq\rho<1,
	\end{equation}
	for all $1\leq i\leq L,1\leq j\leq M$ and $k\in\mathbb{N}$. Then the partition functions from Definition \ref{G partition function defn} satisfy the skew Cauchy identity
	\begin{multline}
		\label{skew Cauchy identity eq}
		\sum_{\kappa} G_{\kappa/\mu}(x_1,\dots,x_L) F_{\kappa/\nu}(z_1,\dots,z_M) \\  = \prod_{i=1}^M\prod_{j=1}^L\left[\frac{x_j-qz_i}{x_j-z_i}\frac{1-z_ix_j}{1-qz_ix_j}\right]\sum_{\lambda}F_{\mu/\lambda}(z_1,\dots,z_M)G_{\nu/\lambda}(x_1,\dots,x_L).
	\end{multline}
    where the sum on the left is an infinite sum over $\kappa\in\mathbb{W}$ while the sum on the right is a finite sum is over $\lambda\in\mathbb{W}$ contained within $\mu$. That is, $\lambda_i<\mu_i$ for all $i$ less than the lengths of both $\lambda$ and $\mu$.
\end{thm}
\begin{proof}
	We begin by writing the left-hand side of \eqref{skew Cauchy identity eq} in double-row operator notation as
	\[\sum_{\kappa} G_{\kappa/\mu}(x_1,\dots,x_L) F_{\kappa/\nu}(z_1,\dots,z_M) = \bra{\mu}A(x_1)\cdots A(x_L) \bigdot{B}(z_1)\cdots \bigdot{B}(z_M)\ket{\nu}.\]
	From here we may commute the $A$ operators through the $\bigdot{B}$ operators using Proposition \ref{AB commutation relation prop}. Each commutation will generate a multiplicative rational factor; collecting all of these, we have 
	\begin{equation}
        \label{eq: skew Cauchy proof 1}
        \bra{\mu}A(x_1)\cdots A(x_L) \bigdot{B}(z_1)\cdots \bigdot{B}(z_M)\ket{\nu} = \prod_{i=1}^M\prod_{j=1}^L\left[\frac{x_j-qz_i}{x_j-z_i}\frac{1-z_ix_j}{1-qz_ix_j}\right] \bra{\mu} \bigdot{B}(z_1)\cdots \bigdot{B}(z_M)A(x_1)\cdots A(x_L) \ket{\nu}.
    \end{equation}
	The right side of this may be recognized as the right-hand side of \eqref{skew Cauchy identity eq}
\end{proof}
\begin{cor}
	\label{Cauchy identity cor}
	With the same set of assumptions as in Theorem \ref{skew Cauchy identity thm}, one has the following summation identity 
	\begin{equation}
		\label{Cauchy identity eq}
		\sum_{\kappa} G_\kappa(x_1,\dots,x_L) F_{\kappa/\nu}(z_1,\dots,z_M) = \prod_{i=1}^M h(z_i) \prod_{i=1}^M\prod_{j=1}^L\left[\frac{x_j-qz_i}{x_j-z_i}\frac{1-z_ix_j}{1-qz_ix_j}\right]G_\nu(x_1,\dots,x_L).
	\end{equation}
\end{cor}
\begin{proof}
    This is the $\mu=\emptyset$ case of Theorem \ref{skew Cauchy identity thm}. Indeed, following the same steps as in the previous proof we use the fact, due to Proposition \ref{B left-eigenvalue prop}, that 
    \[\bra{\emptyset}\bigdot{B}(z_1) \cdots \bigdot{B}(z_M) = \prod_{i=1}^M h(z_i) \bra{\emptyset}\]
    in \eqref{eq: skew Cauchy proof 1}. This recovers precisely \eqref{Cauchy identity eq}.
\end{proof}
\begin{remark}
    A further specification of Corollary \ref{Cauchy identity cor} appears later in the text in Section \ref{sec: Cauchy identity Pfaffian} where $\nu=\emptyset$.
\end{remark}

\begin{remark}
    Subject to the positivity, Propositions \ref{branching relations prop} and \ref{prop: G stocahsticity} motivate the understanding of $G_{\nu/\mu}$ as the propagator of a discrete-time Markov process from initial state $\mu$ to state $\nu$. In order to define an appropriate probability measure, the spectral parameters $(x_1,\dots,x_L)$ and $Y=(y_1,y_2,\dots)$ need to be fixed so that the weights from table \eqref{Stochastic 6VM weights table} and \eqref{Stochastic boundary K-weight table} are all real, non-negative and less than or equal to 1.
    
    We may also regard the symmetric function $F$ as an observable of this Markov process; indeed, the left-side of \eqref{Cauchy identity eq} can be interpreted as the formal definition of the expectation value of the observable $F_{\kappa/\nu}$ with respect to the discrete measure $G_\kappa$. Provided that the right-hand side of \eqref{Cauchy identity eq} can be evaluated explicitly, Corollary \ref{Cauchy identity cor} then provides a systematic method for evaluating the expectation value of the observable $F_{\kappa/\nu}$.
\end{remark}

\subsection{Recursion relations}
In this section we demonstrate a series of recursion relations for the symmetric function \eqref{G partition function defn A eq} which will prove important in deriving a formula for the function. These relations follow from the unitary of the $R$ and $K$-matrices (Propositions \ref{R-matrix unitrarity prop} and \ref{K-matrix unitrarity prop}) and the $R$-matrix factorization (Proposition \ref{lm:Rfactor}). 
\begin{prop}
	\label{Row operator identities}
	We have the following relations for the row-operator \eqref{Row-operator A defn}:
	\begin{align}
		A(0) & = 0, \label{A(0) row-operator identity}\\
		A(\pm 1) & = \mathsf{id}, \label{A(1) row-operator identity}
	\end{align}
	where $\mathsf{id}$ is the identity within $\mathrm{End}(\Span\mathbb{W})$. Further, for fixed $x\in \mathbb{C}$, if there exists $\rho>0$ such that
	\begin{equation}
		\label{A operator recur condition}
		\abs{\frac{1-xy_k}{1-qxy_k}\frac{q(xy_k-1)}{xy_k-q}}\leq\rho<1
	\end{equation}
	for all $j\in\mathbb{N}$, then it holds that
	\begin{equation}
		A(x)A\left(x^{-1}\right) = \mathsf{id}. \label{A(x)A(1/x) row-operator identity}
	\end{equation}
\end{prop}

\begin{proof} 
        We will proceed with the proof of each identity separately. In each proof, we will consider arbitrary finite configurations $\mu,\nu\in\mathbb{W}$. 
	\begin{enumerate}[label=\textbf{(\roman*)}]
		\item \textbf{Proof of \eqref{A(0) row-operator identity}.} It is sufficient here to consider the partition function 
		\begin{equation}
			\label{A one row identity proof f}
			f(x) = \bra{\mu}A(x)\ket{\nu} = \begin{tikzpicture}[baseline={([yshift=-.5ex]current bounding box.center)},scale=0.8]
				\draw[lightgray,line width=1.5pt,->] (8,0) -- (1,0) -- (0,0.5) -- (1,1) -- (8,1);
				\draw[lightgray,line width=1.5pt,->] (8,0) -- (7.5,0);
				\draw[blue,fill=blue] (0,0.5) circle (0.1cm);
				\foreach \x in {2,...,7}
				{\draw[lightgray,line width=1.5pt,->] (\x,-0.5) -- (\x,1.5);}
				\node[right] at (8,0) {$0 \leftarrow x$};
				\node[right] at (8,1) {$0 \rightarrow x^{-1}$};
                \node[above] at (2,1.5) {$\eta^\nu_1$};
			\node[above] at (3,1.5) {$\eta^\nu_2$};
			\node[above] at (4,1.5) {$\eta^\nu_3$};
			\node[above] at (5,1.5) {$\cdots$};
			\node[above] at (6,1.5) {$\cdots$};
			\node[below] at (2,-0.5) {$\eta^\mu_1$};
			\node[below] at (3,-0.5) {$\eta^\mu_2$};
			\node[below] at (4,-0.5) {$\eta^\mu_3$};
			\node[below] at (5,-0.5) {$\cdots$};
			\node[below] at (6,-0.5) {$\cdots$};
			\node[below] at (7,-0.5) {$\cdots$};
			\end{tikzpicture},
		\end{equation}
		where we will be interested in the case $x=0$. We note here that all contributions from the weights which make up \eqref{A one row identity proof f} are non-singular at $x=0$. In this case, the vertex configuration with weight $z(1-q)/(1-qz)=0$ from \eqref{Stochastic 6VM weights table} cannot appear. This means that the lower row of \eqref{A one row identity proof f} cannot carry any occupations on any horizontal edges. 
		
		This will mean that the only possible boundary vertex configuration will be that of weight $h(x) = 1$. Since the weight $z(1-q)/(1-qz)=0$ will not appear on the upper row, each horizontal edge on the upper row must be occupied.
		
		However, this is not permitted since the boundary conditions on the right edge must be vacant. This means that there is no possible path configuration in \eqref{A one row identity proof f}. It follows that $f(0)=0$, which gives the result.
		\item \textbf{Proof of \eqref{A(1) row-operator identity}.} We will again consider \eqref{A one row identity proof f}, this time with $x=\pm1$. In this case the function $h(\pm1)=0$ so that the boundary vertex must have both entry and exits either both occupied or both vacant. Both of these configurations carry weight 1. 
		
		This is equivalent to the fact that that $K(\pm1)$ is the identity matrix. We are then left with 
		\[
		f(\pm1) =
		\begin{tikzpicture}[baseline={([yshift=-.5ex]current bounding box.center)},scale=0.8]
			\draw[lightgray,line width=1.5pt,-] (8,0) -- (1.5,0);
			\draw[lightgray,line width=1.5pt,-] (1.5,0) arc (270:90:0.5);
			\draw[lightgray,line width=1.5pt,->]	(1.5,1) -- (8,1);
			\draw[lightgray,line width=1.5pt,->] (8,0) -- (7.5,0);
			\foreach \x in {2,...,7}
			{\draw[lightgray,line width=1.5pt,->] (\x,-0.5) -- (\x,1.5);}
			\node[right] at (8,0) {$0$};
			\node[right] at (8,1) {$0$};
                \node[above] at (2,1.5) {$\eta^\nu_1$};
			\node[above] at (3,1.5) {$\eta^\nu_2$};
			\node[above] at (4,1.5) {$\eta^\nu_3$};
			\node[above] at (5,1.5) {$\cdots$};
			\node[above] at (6,1.5) {$\cdots$};
            \node[above] at (7,1.5) {$\cdots$};
			\node[below] at (2,-0.5) {$\eta^\mu_1$};
			\node[below] at (3,-0.5) {$\eta^\mu_2$};
			\node[below] at (4,-0.5) {$\eta^\mu_3$};
			\node[below] at (5,-0.5) {$\cdots$};
			\node[below] at (6,-0.5) {$\cdots$};
			\node[below] at (7,-0.5) {$\cdots$};
		\end{tikzpicture}
        =
        \begin{tikzpicture}[baseline={([yshift=-.5ex]current bounding box.center)},scale=0.8]
			\draw[lightgray,line width=1.5pt,->] (8,0) arc (270:90:0.5);
			\foreach \x in {2,...,7}
			{\draw[lightgray,line width=1.5pt,->] (\x,-0.5) -- (\x,1.5);}
			\node[right] at (8,0) {$0$};
			\node[right] at (8,1) {$0$};
                \node[above] at (2,1.5) {$\eta^\nu_1$};
			\node[above] at (3,1.5) {$\eta^\nu_2$};
			\node[above] at (4,1.5) {$\eta^\nu_3$};
			\node[above] at (5,1.5) {$\cdots$};
			\node[above] at (6,1.5) {$\cdots$};
            \node[above] at (7,1.5) {$\cdots$};
			\node[below] at (2,-0.5) {$\eta^\mu_1$};
			\node[below] at (3,-0.5) {$\eta^\mu_2$};
			\node[below] at (4,-0.5) {$\eta^\mu_3$};
			\node[below] at (5,-0.5) {$\cdots$};
			\node[below] at (6,-0.5) {$\cdots$};
			\node[below] at (7,-0.5) {$\cdots$};
		\end{tikzpicture},
		\]
        where we have made repeated use of the unitarity condition of $R$-matrices (Proposition \ref{R-matrix unitrarity prop}), to produce the final equality. From here we may conclude that $f(\pm1)=1$ which gives the result.
		
		\item \textbf{Proof of \eqref{A(x)A(1/x) row-operator identity}.} Let $\tau=\max\{\mu_1,\nu_1\}$ and let $N$ be an integer satisfying $N\geq\tau$. Then for $x\neq 0$, consider 
        \begin{equation}
            \label{eq: G-recur proof 1}
            \sum_{p\in\{0,1\}} \bra{\mu}\mathcal{T}_{0,p}^{(N)}(x)\mathcal{T}_{p,0}^{(N)}\left(x^{-1}\right)\ket{\nu} R_{x/x}(0,p,p,0) = 
        \begin{tikzpicture}[baseline={([yshift=-.5ex]current bounding box.center)},scale=0.7]
    		\draw[lightgray,line width=1.5pt,->] (7,0) -- (1,0) -- (0,0.5) -- (1,1) -- (6,1) -- (7,2);
    		\draw[lightgray,line width=1.5pt,->] (7,0) -- (6.5,0);
    		\draw[lightgray,line width=1.5pt,->] (7,1) -- (6,2) -- (1,2) -- (0,2.5) -- (1,3) -- (7,3);
    		\draw[lightgray,line width=1.5pt,->] (7,1) -- (6.75,1.25);
    		\draw[blue,fill=blue] (0,0.5) circle (0.1cm);
    		\draw[blue,fill=blue] (0,2.5) circle (0.1cm);
    		\foreach \x in {2,...,5}
    		{\draw[lightgray,line width=1.5pt,->] (\x,-0.5) -- (\x,3.5);}
    		\node[right] at (7,0) {$0 \leftarrow x$};
    		\node[right] at (7,1) {$0 \leftarrow x^{-1}$};
    		\node[right] at (7,2) {$0 \rightarrow x^{-1}$};
    		\node[right] at (7,3) {$0 \rightarrow x$};
    
            \node[above] at (2,3.5) {$\eta_1^\nu$};
                \node[above] at (3,3.5) {$\cdots$};
                \node[above] at (4,3.5) {$\cdots$};
                \node[above] at (5,3.5) {$\eta_N^\nu$};
    
                \node[below] at (2,-0.5) {$\eta_1^\mu$};
                \node[below] at (3,-0.5) {$\cdots$};
                \node[below] at (4,-0.5) {$\cdots$};
                \node[below] at (5,-0.5) {$\eta_N^\mu$};
    	\end{tikzpicture},
        \end{equation}
        which is equal to $\bra{\mu}A(x)A\left(x^{-1}\right)\ket{\nu}$ in the limit $N\to\infty$ from Proposition \ref{prop: AA limit condition}. Due to the factorization property of the $R$-matrix (Proposition \ref{lm:Rfactor}), we may recognize that the intertwining vertex in \eqref{eq: G-recur proof 1} is the identity. We may simplify \eqref{eq: G-recur proof 1} as 
        \begin{equation}
            \label{eq: G-recur proof 2}
            \begin{tikzpicture}[baseline={([yshift=-.5ex]current bounding box.center)},scale=0.7]
    		\draw[lightgray,line width=1.5pt] (6,0) -- (1,0) -- (0,0.5) -- (1,1) -- (5.5,1);
    		\draw[lightgray,line width=1.5pt,->] (6,0) -- (5.5,0);
            \draw[lightgray,line width=1.5pt] (5.5,1) arc (-90:90:0.5);
    		\draw[lightgray,line width=1.5pt,->] (5.5,2) -- (1,2) -- (0,2.5) -- (1,3) -- (6,3);
    		\draw[blue,fill=blue] (0,0.5) circle (0.1cm);
    		\draw[blue,fill=blue] (0,2.5) circle (0.1cm);
    		\foreach \x in {2,...,5}
    		{\draw[lightgray,line width=1.5pt,->] (\x,-0.5) -- (\x,3.5);}
    		\node[right] at (6,0) {$0 \leftarrow x$};
    		\node[right] at (6,3) {$0 \rightarrow x$};

            \node[above] at (2,3.5) {$\eta_1^\nu$};
                \node[above] at (3,3.5) {$\cdots$};
                \node[above] at (4,3.5) {$\cdots$};
                \node[above] at (5,3.5) {$\eta_N^\nu$};
    
                \node[below] at (2,-0.5) {$\eta_1^\mu$};
                \node[below] at (3,-0.5) {$\cdots$};
                \node[below] at (4,-0.5) {$\cdots$};
                \node[below] at (5,-0.5) {$\eta_N^\mu$};
    	\end{tikzpicture}
        = \hspace{0.2cm}\begin{tikzpicture}[baseline={([yshift=-.5ex]current bounding box.center)},scale=0.7]
			\draw[lightgray,line width=1.5pt,->] (6,0) -- (1,0) -- (0,0.5) -- (0,2.5) -- (1,3) -- (6,3);
			\draw[blue,fill=blue] (0,0.5) circle (0.1cm);
			\draw[blue,fill=blue] (0,2.5) circle (0.1cm);
			\foreach \x in {2,...,5}
			{\draw[lightgray,line width=1.5pt,->] (\x,-0.5) -- (\x,3.5);}
			\node[right] at (6,0) {$0 \leftarrow x$};
			\node[right] at (6,3) {$0 \rightarrow x$};

            \node[above] at (2,3.5) {$\eta_1^\nu$};
                \node[above] at (3,3.5) {$\cdots$};
                \node[above] at (4,3.5) {$\cdots$};
                \node[above] at (5,3.5) {$\eta_N^\nu$};
    
                \node[below] at (2,-0.5) {$\eta_1^\mu$};
                \node[below] at (3,-0.5) {$\cdots$};
                \node[below] at (4,-0.5) {$\cdots$};
                \node[below] at (5,-0.5) {$\eta_N^\mu$};
		\end{tikzpicture},
        \end{equation}
        where we have use the unitarity property of the $R$-matrix (Proposition \ref{R-matrix unitrarity prop}) over the $N$ columns to obtain the second diagram. Within this diagram, we may use the unitarity property of the $K$-matrix (Proposition \ref{K-matrix unitrarity prop}) to simplify \eqref{eq: G-recur proof 2} as
        \begin{equation}
            \begin{tikzpicture}[baseline={([yshift=-.5ex]current bounding box.center)},scale=0.8]
			\draw[lightgray,line width=1.5pt,-] (6,0) -- (1.5,0);
			\draw[lightgray,line width=1.5pt,-] (1.5,0) arc (270:90:0.5);
			\draw[lightgray,line width=1.5pt,->]	(1.5,1) -- (6,1);
			\draw[lightgray,line width=1.5pt,->] (6,0) -- (5.5,0);
			\foreach \x in {2,...,5}
			{\draw[lightgray,line width=1.5pt,->] (\x,-0.5) -- (\x,1.5);}
			\node[right] at (6,0) {$0\leftarrow x$};
			\node[right] at (6,1) {$0\rightarrow x$};
                \node[above] at (2,1.5) {$\eta^\nu_1$};
			\node[above] at (3,1.5) {$\cdots$};
			\node[above] at (4,1.5) {$\cdots$};
			\node[above] at (5,1.5) {$\eta^\nu_N$};
			\node[below] at (2,-0.5) {$\eta^\mu_1$};
			\node[below] at (3,-0.5) {$\cdots$};
			\node[below] at (4,-0.5) {$\cdots$};
			\node[below] at (5,-0.5) {$\eta^\mu_N$};

		\end{tikzpicture},
        \end{equation}
        from which we may again use $R$-matrix unitarity to remove all vertices. This shows that \eqref{eq: G nu empty picture proof 1} is equal to $\bra{\mu}\ket{\nu}$, which yields the result under the limit $N\to\infty$.
        
	\end{enumerate}
\end{proof}
The relations in Proposition~\ref{Row operator identities} for the row-operators lead to recursion relations for the partition function \eqref{G partition function defn A eq}.
\begin{cor}
	\label{cor:recursionsG}
	The partition function $G_{\nu/\mu}(x_1,\dots,x_L|\yalph)$ from \eqref{G partition function defn A eq} satisfies the following recursion relations
    \begin{align}
        G_{\nu/\mu} (x_1,\dots,x_L)\Big|_{x_i=0} & = 0,\\
        G_{\nu/\mu} (x_1,\dots,x_L)\Big|_{x_i=\pm 1}  & = G_{\nu/\mu} (x_1,\dots,\hat{x}_i,\dots,x_L),\\
        \label{eq: G x=1/x recur}
        G_{\nu/\mu} (x_1,\dots,x_L)\Big|_{x_j=1/x_k} & = G_{\nu/\mu} (x_1,\dots,\hat{x}_j,\dots,\hat{x}_k,\dots,x_L),
    \end{align}
    for all $1\leq i\leq L$ and $1\leq j<k\leq L$. Here, $\hat{x}_i$ means that $x_i$ is omitted from the alphabet $(x_1,\dots,x_L)$.
\end{cor}
\begin{proof}
	These follow from the row-operator identities,\eqref{A(0) row-operator identity},\eqref{A(1) row-operator identity},\eqref{A(x)A(1/x) row-operator identity}, in Proposition \ref{Row operator identities} as well as the fact that $G_{\nu/\mu}$ is symmetric in its alphabet.
\end{proof}

\begin{cor}
    Let $\mu,\nu\in\mathbb{W}$ and $x_1,\dots,x_L\in\mathbb{C}\backslash\{0\}$. The partition function \eqref{G partition function defn A eq} satisfies the following unitarity property:
    \begin{equation}
        \sum_{\kappa} G_{\nu/\kappa} \left(x_1^{-1},\dots,x_L^{-1}\right) G_{\kappa/\mu}(x_1,\dots,x_L) = \delta_{\mu,\nu}.
    \end{equation}
\end{cor}
\begin{proof}
    The proof follows from the branching relation \eqref{G branching rel}, the recursion relation \eqref{eq: G x=1/x recur}, and the fact that $G_{\nu/\mu}$ evaluated on an empty alphabet is equal to $\delta_{\mu,\nu}$.
\end{proof}

\subsection{Empty initial conditions}
\begin{lm}
	\label{G nu picture lemma}
	Let $(x_1,\dots,x_L)$ be an alphabet satisfying conditions \eqref{A operator commutation condition} for all $i\neq j$. The evaluation of the symmetric function indexed by a single configuration $G_\nu=G_{\nu/\emptyset}$ from \eqref{G partition function defn A eq} reduces to the following partition function
	\begin{equation}
		\label{G nu picture lemma eq}
		G_\nu(x_1,\dots,x_L|\yalph) = 
		\begin{tikzpicture}[baseline={([yshift=-.5ex]current bounding box.center)},scale=0.8]
			\draw[lightgray,line width=1.5pt,->] (-4,0) -- (-4,1) -- (7,1);
			\draw[lightgray,line width=1.5pt,->] (-3,0) -- (-3,2) -- (7,2);
			\draw[lightgray,line width=1.5pt,->] (-2,0) -- (-2,3) -- (7,3);
			\draw[lightgray,line width=1.5pt,->] (-1,0) -- (-1,4) -- (7,4);
			\draw[lightgray,line width=1.5pt,->] (0,0) -- (0,5) -- (7,5);
			
			\draw[lightgray,line width=1.5pt,->] (-4,0) -- (-4,0.5);
			\draw[lightgray,line width=1.5pt,->] (-3,0) -- (-3,0.5);
			\draw[lightgray,line width=1.5pt,->] (-2,0) -- (-2,0.5);
			\draw[lightgray,line width=1.5pt,->] (-1,0) -- (-1,0.5);
			\draw[lightgray,line width=1.5pt,->] (0,0) -- (0,0.5);
			
			\foreach \x in {1,...,6}
			{\draw[lightgray,line width=2pt,->] (\x,0) -- (\x,6);
				\node[below] at (\x,0) {$0$};}
			
			\draw[blue,fill=blue] (-4,1) circle (0.1cm);
			\draw[blue,fill=blue] (-3,2) circle (0.1cm);
			\draw[blue,fill=blue] (-2,3) circle (0.1cm);
			\draw[blue,fill=blue] (-1,4) circle (0.1cm);
			\draw[blue,fill=blue] (0,5) circle (0.1cm);
			
			\node[below] at (-4,0) {$0$};
			\node[below] at (-3,0) {$0$};
			\node[below] at (-2,0) {$0$};
			\node[below] at (-1,0) {$0$};
			\node[below] at (0,0) {$0$};
			
			\node[below] at (-4,-0.5) {$\uparrow$};
			\node[below] at (-3,-0.5) {$\uparrow$};
			\node[below] at (0,-0.5) {$\uparrow$};
			
			\node[below] at (-4,-1.1) {$x_1$};
			\node[below] at (-3,-1.1) {$x_2$};
			\node[below] at (-2,-1.1) {$\cdots$};
			\node[below] at (-1,-1.1) {$\cdots$};
			\node[below] at (0,-1.1) {$x_L$};
			
			\node[below] at (1,-0.5) {$\uparrow$};
			\node[below] at (2,-0.5) {$\uparrow$};
			\node[below] at (3,-0.5) {$\uparrow$};
			
			\node[below] at (1,-1.1) {$y_1$};
			\node[below] at (2,-1.1) {$y_2$};
			\node[below] at (3,-1.1) {$y_3$};
			
			\node[below] at (4,-1.1) {$\cdots$};
			\node[below] at (5,-1.1) {$\cdots$};
			\node[below] at (6,-1.1) {$\cdots$};
			
			\node[right] at (7,1) {$0 \rightarrow x_1^{-1}$};
			\node[right] at (7,2) {$0 \rightarrow x_2^{-1}$};
			\node[right] at (7,3) {$0$};
			\node[right] at (7,4) {$0$};
			\node[right] at (7,5) {$0 \rightarrow x_L^{-1}$};
			
			\node[right] at (7.9,3) {$\vdots$};
			\node[right] at (7.9,4) {$\vdots$};
			
			\node[above] at (1,6) {$\eta^\nu_1$};
			\node[above] at (2,6) {$\eta^\nu_2$};
			\node[above] at (3,6) {$\eta^\nu_3$};
			\node[above] at (4,6) {$\cdots$};
			\node[above] at (5,6) {$\cdots$};   
			\node[above] at (6,6) {$\cdots$};   
		\end{tikzpicture}
		.\end{equation}
\end{lm}

\begin{proof}
    Given an integer $N\geq\nu_1$, consider the $N$ column version of the double-row picture \eqref{G partition function defn picture}. We may append a triangular arrangement of $R$-matrices onto the right-hand side of the lattice to obtain the following object
    \begin{equation}
        \label{eq: G nu empty picture proof 1}
        g_\nu^{(N)}(x_1,\dots,x_L) :=
		\sum_{p_1,p_2,\dots\in\{0,1\}}
		\begin{tikzpicture}[baseline={([yshift=-.5ex]current bounding box.center)},scale=0.8]
			\draw[lightgray,line width=1.5pt,->] (10,0) -- (6,0) -- (1,0) -- (0,0.5) -- (1,1) -- (6,1) -- (9,4) -- (10,4);
			\draw[lightgray,line width=1.5pt,->] (10,0) -- (9.5,0);
			\draw[lightgray,line width=1.5pt,->] (10,1) -- (7,1) -- (6,2) -- (1,2) -- (0,2.5) -- (1,3) -- (6,3) -- (8,5) -- (10,5);
			\draw[lightgray,line width=1.5pt,->] (10,1) -- (9.5,1);
			\draw[lightgray,line width=1.5pt,->] (10,2) -- (8,2) -- (6,4) -- (1,4) -- (0,4.5) -- (1,5) -- (6,5) -- (7,6) -- (10,6);
			\draw[lightgray,line width=1.5pt,->] (10,2) -- (9.5,2);
			\draw[lightgray,line width=1.5pt,->] (10,3) -- (9,3) -- (6,6) -- (1,6) -- (0,6.5) -- (1,7) -- (6,7) -- (10,7);
			\draw[lightgray,line width=1.5pt,->] (10,3) -- (9.5,3);
			\draw[blue,fill=blue] (0,0.5) circle (0.1cm);
			\draw[blue,fill=blue] (0,2.5) circle (0.1cm);
			\draw[blue,fill=blue] (0,4.5) circle (0.1cm);
			\draw[blue,fill=blue] (0,6.5) circle (0.1cm);
			\foreach \x in {2,...,5}
			{\draw[lightgray,line width=1.5pt,->] (\x,-0.5) -- (\x,7.5);}
			\node[right] at (10,0) {$0 \leftarrow x_1$};
			\node[right] at (10,4) {$0 \rightarrow x_1^{-1}$};
			\node[right] at (10,3) {$0 \leftarrow x_L$};
			\node[right] at (10,7) {$0 \rightarrow x_L^{-1}$};
			\node[right] at (10,1) {$0$};
			\node[right] at (10,2) {$0$};
			\node[right] at (10,5) {$0$};
			\node[right] at (10,6) {$0$};
                \node[right] at (10.9,1) {$\vdots$};
			\node[right] at (10.9,2) {$\vdots$};
                \node[right] at (10.9,5) {$\vdots$};
			\node[right] at (10.9,6) {$\vdots$};
   
            \node[below] at (6,1) {$p_1$};
            \node[above] at (6,2) {$p_2$};
            \node[above] at (6,6) {$p_{2L-2}$};
            
            \node[above] at (2,7.5) {$\eta^\nu_1$};
            \node[above] at (3,7.5) {$\cdots$};
			\node[above] at (4,7.5) {$\cdots$};
			\node[above] at (5,7.5) {$\eta^\nu_N$};
            \node[below] at (2,-0.5) {$0$};
            \node[below] at (3,-0.5) {$0$};
			\node[below] at (4,-0.5) {$0$};
			\node[below] at (5,-0.5) {$0$};   
		\end{tikzpicture},
    \end{equation}
    where the sum is over all occupations of the edges of the appended triangle. Absorbing the sum into the notation of the partition function, we can apply the Yang\textendash Baxter equation to move the intertwiners to be adjacent to the boundary. This is depicted as
    \begin{equation}
        \label{eq: G nu empty picture proof 2}
        g_\nu^{(N)}(x_1,\dots,x_L)
        \begin{tikzpicture}[baseline={([yshift=-.5ex]current bounding box.center)},scale=0.8]
			\draw[lightgray,line width=1.5pt,->] (9,0) -- (1,0) -- (0,0.5) -- (4,4) -- (9,4);
			\draw[lightgray,line width=1.5pt,->] (9,0) -- (8.5,0);
			\draw[lightgray,line width=1.5pt,->] (9,1) -- (2,1) -- (0,2.5) -- (3,5) -- (9,5);
			\draw[lightgray,line width=1.5pt,->] (9,1) -- (8.5,1);
			\draw[lightgray,line width=1.5pt,->] (9,2) -- (3,2) -- (0,4.5) -- (2,6) -- (9,6);
			\draw[lightgray,line width=1.5pt,->] (9,2) -- (8.5,2);
			\draw[lightgray,line width=1.5pt,->] (9,3) -- (4,3) -- (0,6.5) -- (1,7) -- (9,7);
			\draw[lightgray,line width=1.5pt,->] (9,3) -- (8.5,3);
			\draw[blue,fill=blue] (0,0.5) circle (0.1cm);
			\draw[blue,fill=blue] (0,2.5) circle (0.1cm);
			\draw[blue,fill=blue] (0,4.5) circle (0.1cm);
			\draw[blue,fill=blue] (0,6.5) circle (0.1cm);
			\foreach \x in {5,...,8}
			{\draw[lightgray,line width=1.5pt,->] (\x,-0.5) -- (\x,7.5);}
			\node[right] at (9,0) {$0 \leftarrow x_1$};
			\node[right] at (9,4) {$0 \rightarrow x_1^{-1}$};
			\node[right] at (9,3) {$0 \leftarrow x_L$};
			\node[right] at (9,7) {$0 \rightarrow x_L^{-1}$};
			\node[right] at (9,1) {$\vdots$};
			\node[right] at (9,2) {$\vdots$};
			\node[right] at (9,5) {$\vdots$};
			\node[right] at (9,6) {$\vdots$};

            \node[above] at (5,7.5) {$\eta^\nu_1$};
            \node[above] at (6,7.5) {$\cdots$};
			\node[above] at (7,7.5) {$\cdots$};
			\node[above] at (8,7.5) {$\eta^\nu_N$};
            \node[below] at (5,-0.5) {$0$};
            \node[below] at (6,-0.5) {$\cdots$};
			\node[below] at (7,-0.5) {$\cdots$};
			\node[below] at (8,-0.5) {$0$}; 
		\end{tikzpicture}.
    \end{equation}
    We note that the left-moving sector of the vertical columns is frozen with no occupations. As a result, these vertices may be evaluated to 1 and removed from the diagram. In the limit $N\to\infty$ this results in the desired diagram \eqref{G nu picture lemma eq}.
    
    Now directly consider the large $N$ limit of \eqref{eq: G nu empty picture proof 1}. Due to the conditions \eqref{A operator commutation condition} the only term which survives this limit is the one with $p_1=\cdots=p_{2L-2}=0$. This forces all $p_i=0$ so that the sum of \eqref{eq: G nu empty picture proof 1} collapses into a single term where all intertwining vertices are equal to 1. This yields
    \begin{equation}
        \lim_{N\to\infty} g_\nu^{(N)}(x_1,\dots,x_L) = \lim_{N\to\infty} \bra{\emptyset}A^{(N)}(x_1) \cdots A^{(N)}(x_L) \ket{\nu} = G_\nu(x_1,\dots,x_L).
    \end{equation}
\end{proof}

The diagram in \eqref{G nu picture lemma eq} leads us to the following definition and theorem.
\begin{defn} 
	We refer to the following partition function, with generic boundary parameters, as the \emph{triangular partition function}:
	\begin{equation}
		\label{Z triangular partition function defn eq}
		Z_m(x_1\dots,x_m) = 
		\begin{tikzpicture}[baseline={([yshift=-.5ex]current bounding box.center)},scale=1]
			\draw[lightgray,line width=1.5pt,->] (1,0) -- (1,1) -- (6,1);
			\draw[lightgray,line width=1.5pt,->] (2,0) -- (2,2) -- (6,2);
			\draw[lightgray,line width=1.5pt,->] (3,0) -- (3,3) -- (6,3);
			\draw[lightgray,line width=1.5pt,->] (4,0) -- (4,4) -- (6,4);
			\draw[lightgray,line width=1.5pt,->] (5,0) -- (5,5) -- (6,5);
			
			\draw[lightgray,line width=1.5pt,->] (1,0) -- (1,0.5);
			\draw[lightgray,line width=1.5pt,->] (2,0) -- (2,0.5);
			\draw[lightgray,line width=1.5pt,->] (3,0) -- (3,0.5);
			\draw[lightgray,line width=1.5pt,->] (4,0) -- (4,0.5);
			\draw[lightgray,line width=1.5pt,->] (5,0) -- (5,0.5);
			
			\draw[blue,fill=blue] (5,5) circle (0.1cm);
			\draw[blue,fill=blue] (4,4) circle (0.1cm);
			\draw[blue,fill=blue] (3,3) circle (0.1cm);
			\draw[blue,fill=blue] (2,2) circle (0.1cm);
			\draw[blue,fill=blue] (1,1) circle (0.1cm);
			
			\node[right] at (6,1) {$0 \rightarrow x_1^{-1}$};
			\node[right] at (6,2) {$0 \rightarrow x_2^{-1}$};
			\node[right] at (6,3) {$0$};
			\node[right] at (6,4) {$0$};
			\node[right] at (6,5) {$0 \rightarrow x_m^{-1}$};
			
			\node[right] at (6.9,3) {$\vdots$};
			\node[right] at (6.9,4) {$\vdots$};
			
			\node[below] at (1,-0.4) {$\uparrow$};
			\node[below] at (2,-0.4) {$\uparrow$};
			\node[below] at (5,-0.4) {$\uparrow$};
			
			\node[below] at (1,-0.9) {$x_1$};
			\node[below] at (2,-0.9) {$x_2$};
			\node[below] at (3,-0.9) {$\cdots$};
			\node[below] at (4,-0.9) {$\cdots$};
			\node[below] at (5,-0.9) {$x_m$};
			
			\node[below] at (1,0) {$0$};
			\node[below] at (2,0) {$0$};
			\node[below] at (3,0) {$0$};
			\node[below] at (4,0) {$0$};
			\node[below] at (5,0) {$0$};
		\end{tikzpicture}.
	\end{equation}
\end{defn}

\begin{thm}
	\label{G=Z empty set thm}
	When both the initial and final configurations are empty, i.e. $\mu=\nu=\emptyset$, the symmetric function \eqref{G partition function defn A eq} reduces to the triangular partition function \eqref{Z triangular partition function defn eq}
	\begin{equation}
		G_{\emptyset} (x_1,\dots,x_L) = Z_L(x_1,\dots,x_L).
	\end{equation}
    In particular, $G_\emptyset$ does not depend on the collection of vertical spectral parameters $Y$.
\end{thm}
\begin{proof}
	This result is immediate from diagram \eqref{G nu picture lemma eq} in Lemma \ref{G nu picture lemma} when we set $\nu$ to be empty. This causes there to be no occupations on any of columns or rows past the point of $y_i$-dependence. This means that the entire bulk on the right evaluates to 1 and may be removed without effect on the partition function evaluation.
\end{proof}

\section{Triangular partition function}
In this section we turn our attention to the triangular partition function \eqref{Z triangular partition function defn eq}. An important feature of the six-vertex model with generic open boundary conditions is that the partition function $G_\varnothing(x_1,\ldots,x_L)$ is non-trivial. Theorem~\ref{G=Z empty set thm} tells us that this partition function is equal to the triangular partition function. This function plays an analogous role to the domain wall partition function \cite{izergin_partition_1987} of the six-vertex model on a square geometry. In fact, this function directly generalizes a partition function related to diagonally symmetric alternating-sign matrices introduced in \cite{kuperberg_symmetry_2002}.

We first note that for any $m \in \mathbb{N}$, the partition function $Z_m(x_1,\ldots,x_m)$ is symmetric and satisfies the recursion relations of Corollary~\ref{cor:recursionsG}. It is instructive to derive these properties directly from \eqref{Z triangular partition function defn eq} and we will do so below. Moreover, the triangular partition function is a rational function in its alphabet where the degree of numerator and denominator can be easily established. All these properties together completely determine $Z_m(x_1,\ldots,x_m)$ and we shall find several closed formulas for it.

\subsection{Properties and recursion relations}
\begin{prop}
    \label{Z symmetric function prop}
    The triangular partition function $Z_m(x_1,\dots,x_m)$ from \eqref{Z triangular partition function defn eq} is symmetric in $(x_1,\dots,x_m)$.
\end{prop}
\begin{proof}
    Considering $Z_m(x_1,\dots,x_m)$, the symmetry can be seen by inserting an intertwining vertex at the bottom of the diagram \eqref{Z triangular partition function defn eq} to cross the adjacent lines with spectral parameters $x_i$ and $x_{i+1}$. This can be done at no overall cost to the partition function since the boundary conditions enforce that the only possible vertex configuration is the all-empty one, which carries weight 1.

    Using repeated applications of the Yang\textendash Baxter equation \eqref{Yang-Baxter eq R-matrices} and reflection equation \eqref{Sklyanin reflection equation R,K matrices}, the intertwining $R$-matrix can be pulled through to the right hand side of the lattice between the $i$-th and $(i+1)$-th lines of the lattice. In the process of shifting this intertwiner, the $x_i$ and $x_{i+1}$ spectral parameters swap positions. The intertwiner can then be removed from the right hand side of the lattice at no overall cost, again enforced by the boundary conditions, leaving us with \[Z_m(x_1,\dots,x_i,x_{i+1},\dots,x_m)=Z_m(x_1,\dots,x_{i+1},x_i,\dots,x_m),\]
    which generates symmetry over the whole alphabet.
\end{proof}
\begin{prop}
	\label{Z polynomial prop}
	The function defined by
	\begin{equation}
        \label{Z polynomial eq}
		\widetilde{Z}_m(x_1,\dots,x_m) := \prod_{i=1}^m \left[(a-x_i)(c-x_i)\right] \prod_{1\leq i<j\leq m} (1-q x_i x_j) \cdot Z_m(x_1,\dots,x_m)
	\end{equation}
	is a symmetric polynomial of degree at most $m+1$ in each of the variables $x_1,\dots,x_m$.
\end{prop}
\begin{proof}
    The pre-factors remove all possible denominators of bulk and boundary vertex weights of any lattice configuration, so that we can conclude that $\tilde{Z}_m$ is a polynomial in all $x_1,\dots,x_m$.
    Once the denominators are removed, the vertex on the boundary with argument $x_i$ contributes a factor of power 2. Each of the $m-1$ bulk vertices also contribute a factor of power of at most 1, giving the leading order power of $m+1$.
\end{proof}
\begin{prop}
	\label{Z trianglular recursion prop Z}
	The triangular partition function \eqref{Z triangular partition function defn eq} both satisfies, and is completely determined by, the following recursion relations
	\begin{align}
		Z_m (x_1,\ldots,x_m) \Big|_{x_i=0} &= 0,\label{Z recur x=0} \\
		Z_m (x_1,\ldots,x_m) \Big|_{x_i=\pm 1} &=  Z_{m-1}  (x_1,\ldots,\hat{x}_i,\ldots,x_m),\label{Z recur x=1}\\
		Z_m (x_1,\ldots,x_m) \Big|_{x_i=1/x_j} &= Z_{m-2} (x_1,\ldots,\hat{x}_i,\ldots,\hat{x}_j,\ldots,x_m).\label{Z recur x1=1/x2}
	\end{align}
\end{prop}

\begin{proof}
    Because of Proposition~\ref{Z polynomial prop}, the triangular partition function $Z_m$ is defined by a polynomial of degree $m+1$ in each variable. The $m+2$ recursions for each variable in the statement of Proposition~\ref{Z trianglular recursion prop Z} therefore completely determine the rational function $Z_m$.

    We now prove the recursion relations individually.
    \begin{enumerate}[label=\textbf{(\roman*)}]
        \item \textbf{Proof of \eqref{Z recur x=0}.} We make use of the symmetry of the partition function and consider
        \[Z_m\left(x_i,x_1,\dots,\hat{x}_i,\dots,x_m\right)\Big|_{x_i=0},\]
        so that the boundary vertex corresponding to the parameter $x_i$ is at the bottom-left-most position. Setting $x_i=0$ forces the bottom-left boundary vertex to generate a path with weight $h(0)=1$, where we recall the definition of $h(x)$ in \eqref{eq:hdef}. Since $x_i=0$, the bottom-right weight in \eqref{Stochastic 6VM weights table} vanishes and hence this path cannot turn to any vertical edge on any vertex on the bottom line in the diagram \eqref{Z triangular partition function defn eq}. And so this path must proceed to the right hand side of the lattice. However, due to the imposed boundary conditions there cannot be any occupations on the external boundary edges on the right hand side and so we conclude that there are no allowed configurations when $x_i=0$ and hence that the partition function is equal to zero. 
        \item \textbf{Proof of \eqref{Z recur x=1}.} Again making use of symmetry we consider \[Z_m\left(x_i,x_1,\dots,\hat{x}_i,\dots,x_m\right)\Big|_{x_i=\pm1}.\]
        This forces the boundary vertex at the bottom-left of the diagram in \eqref{Z triangular partition function defn eq} to generate no paths and be of weight $1-h(\pm1)=1$. This causes there to be no occupations along any edges of the bottom line, and hence this line may be removed at no cost leaving us with $Z_{m-1}\left(x_1,\dots,\hat{x}_i,\dots,x_m\right)$.
        \item \textbf{Proof of \eqref{Z recur x1=1/x2}.} Again using symmetry we consider
        \[Z_m\left(x_i,x_j,x_1,\dots,\hat{x}_i,\dots,\hat{x}_j,\dots,x_m\right)\Big|_{x_i=1/x_j}.\]
        Using the $R$-matrix factorization, Lemma \ref{lm:Rfactor}, the partition function becomes 
         \[\begin{tikzpicture}[baseline={([yshift=-.5ex]current bounding box.center)},scale=0.8]
         \begin{scope}[shift={(0,6)},rotate around={-90:(0,0)}]
			\draw[lightgray,line width=1.5pt,-] (6,1) -- (5,1) -- (5,1.3);
            \draw[lightgray,line width=1.5pt,-] (5,1.3) arc (0:90:0.7);
            \draw[lightgray,line width=1.5pt,->] (5,1.3) arc (0:45:0.7);
            \draw[lightgray,line width=1.5pt,->] (4.3,2) -- (4,2) -- (4,6);
            \draw[lightgray,line width=1.5pt,-] (6,2) -- (5.7,2);
			\draw[lightgray,line width=1.5pt,-] (5.7,2) arc (270:180:0.7);
            \draw[lightgray,line width=1.5pt,->] (5,2.7) -- (5,6);
			\draw[lightgray,line width=1.5pt,->] (6,3) -- (3,3) -- (3,6);
			\draw[lightgray,line width=1.5pt,->] (6,4) -- (2,4) -- (2,6);
			\draw[lightgray,line width=1.5pt,->] (6,5) -- (1,5) -- (1,6);
			
			\draw[lightgray,line width=1.5pt,->] (6,1) -- (5.5,1);
			\draw[lightgray,line width=1.5pt,->] (5.7,2) arc (270:235:0.7);
			\draw[lightgray,line width=1.5pt,->] (6,3) -- (5.5,3);
			\draw[lightgray,line width=1.5pt,->] (6,4) -- (5.5,4);
			\draw[lightgray,line width=1.5pt,->] (6,5) -- (5.5,5);
			
			\draw[blue,fill=blue] (5,1) circle (0.1cm);
			\draw[blue,fill=blue] (4,2) circle (0.1cm);
			\draw[blue,fill=blue] (3,3) circle (0.1cm);
			\draw[blue,fill=blue] (2,4) circle (0.1cm);
			\draw[blue,fill=blue] (1,5) circle (0.1cm);
        \end{scope}

			\node[right] at (6,1) {$0$};
			\node[right] at (6,2) {$0$};
			\node[right] at (6,3) {$0$};
			\node[right] at (6,4) {$0$};
			\node[right] at (6,5) {$0$};

            \node[below] at (1,-0.5) {$\uparrow$};
            \node[below] at (2,-0.5) {$\uparrow$};
            \node[below] at (3,-0.5) {$\uparrow$};
            \node[below] at (4,-0.5) {$\uparrow$};
            \node[below] at (5,-0.5) {$\uparrow$};
            
            \node[below] at (1,-1.1) {$x_j^{-1}$};
            \node[below] at (2,-1.3) {$x_j$};
            \node[below] at (3,-1.3) {$x_1$};
            \node[below] at (4,-1.3) {$\cdots$};
            \node[below] at (5,-1.3) {$x_m$};
   
			\node[below] at (1,0) {$0$};
			\node[below] at (2,0) {$0$};
			\node[below] at (3,0) {$0$};
			\node[below] at (4,0) {$0$};
			\node[below] at (5,0) {$0$};
		\end{tikzpicture}\quad
        =\quad
        \begin{tikzpicture}[baseline={([yshift=-.5ex]current bounding box.center)},scale=0.8]
         \begin{scope}[shift={(0,6)},rotate around={-90:(0,0)}]
            \draw[lightgray,line width=1.5pt,->]  (4,2) -- (4,6);
            \draw[lightgray,line width=1.5pt,->] (5,2) -- (5,6);
			\draw[lightgray,line width=1.5pt,->] (6,3) -- (3,3) -- (3,6);
			\draw[lightgray,line width=1.5pt,->] (6,4) -- (2,4) -- (2,6);
			\draw[lightgray,line width=1.5pt,->] (6,5) -- (1,5) -- (1,6);
			
			\draw[lightgray,line width=1.5pt,->] (6,3) -- (5.5,3);
			\draw[lightgray,line width=1.5pt,->] (6,4) -- (5.5,4);
			\draw[lightgray,line width=1.5pt,->] (6,5) -- (5.5,5);
		  
			\draw[blue,fill=blue] (3,3) circle (0.1cm);
			\draw[blue,fill=blue] (2,4) circle (0.1cm);
			\draw[blue,fill=blue] (1,5) circle (0.1cm);
        \end{scope}
			
			\node[left] at (2,1) {$0$};
			\node[left] at (2,2) {$0$};

            \node[below] at (3,-0.5) {$\uparrow$};
            \node[below] at (4,-0.5) {$\uparrow$};
            \node[below] at (5,-0.5) {$\uparrow$};
            
            \node[below] at (3,-1.1) {$x_1$};
            \node[below] at (4,-1.1) {$\cdots$};
            \node[below] at (5,-1.1) {$x_m$};
   
			\node[below] at (3,0) {$0$};
			\node[below] at (4,0) {$0$};
			\node[below] at (5,0) {$0$};
			
			\node[right] at (6,1) {$0$};
			\node[right] at (6,2) {$0$};
			\node[right] at (6,3) {$0$};
			\node[right] at (6,4) {$0$};
			\node[right] at (6,5) {$0$};
		\end{tikzpicture}.\]

        Here, we have used the unitarity of the $K$-matrix (Proposition \ref{K-matrix unitrarity prop}) to obtain the second diagram. We note here that the two bottom rows in the second diagram are completely frozen with no occupations and total weight 1, so that they can be removed at no cost to the partition function. This yields the result.
    \end{enumerate}
\end{proof}

\begin{remark}
Another set of recursions can be derived for the numerator \eqref{Z polynomial eq} of $Z_m$,
\begin{align}
\widetilde{Z}_m (x_1,\ldots,x_m) |_{x_i=1/qx_j} &= -a c (1 - q)q^{m-3} (1-x_j^2)(1-1/q^2x_j^2)\prod_{k=1 \above 0pt k\neq i,j}^m  (1-x_k/qx_j) (1-x_k x_j)  \nonumber \\
&\times \widetilde{Z}_{m-2} (x_1,\ldots,\hat{x}_i,\ldots,\hat{x}_j,\ldots,x_m).
\label{recur3}
\end{align}
These relations follow from the observation that the bottom-right vertex in diagram \eqref{Z triangular partition function defn eq} completely freezes when $qx_1x_m=1$, and as a consequence so does the bottom row and right-most column, leaving a partition function of size $m-2$ multiplied by the pre-factors in \eqref{recur3} that arise from the weights of the frozen vertices. By symmetry a similar result follows for $qx_ix_j=1$ for any $i$ and $j$.
\end{remark}

\subsection{Solution to recursion relations} 

This section provides solutions to the recursion relations of Proposition \ref{Z trianglular recursion prop Z} which in turn provide closed form solutions to the triangular partition function.

Let $Z_m^K$ be the partition function corresponding to the off-diagonal boundary conditions, which can be realized by setting $a=-c=1$, set
     \begin{align}
         \label{eq:Zkup}
         Z^K_m(x_1,\dots,x_m)
         :=Z_m(x_1,\dots,x_m)|_{a=-c=1}.
     \end{align}
     This partition function admits a Pfaffian formula due to Kuperberg \cite{kuperberg_symmetry_2002}:
     \begin{align}
     \label{eq:KPf}
     Z^K_m(x_1,\dots,x_m) &=
     \prod_{i=1}^m x_i
     \prod_{1\leq i<j\leq m} \frac{1-x_i x_j}{x_i-x_j}
     \Pf\left(M(x_i,x_j) \right)_{1\leq i,j\leq m},
     \\
     \label{eq:MKup}
         M(x_i,x_j) &= \frac{(1-q)(x_i-x_j)}{(1-x_i x_j)(1-q x_i x_j)}.
     \end{align}
We note that when $m$ is odd the partition function $Z_m^K$ vanishes.
\begin{thm}
\label{thm:ZsubsetK}
The triangular partition function \eqref{Z triangular partition function defn eq} with the general boundary weights can be expressed as,
    \begin{align}
        Z_m(x_1,\dots,x_m) = H_m(x)  
        \sum_{r=0}^{ \lfloor m/2 \rfloor} \left(\frac{1}{ -ac}\right)^{r}  \sum_{S \subseteq [1,m] \atop |S|=2r} \prod_{i \in S} \frac{h(x_i)}{1-h(x_i)}  \prod_{i \in S \atop j \in \comp{S}} \frac{1-x_ix_j}{x_i-x_j} 
        Z^K_{2r}(x_S),
        \label{eq:Zsubset}
     \end{align}
     where 
\be
\label{eq:Hdef}
H_m(x) := \prod_{i=1}^m (1-h(x_i)),\qquad h(x) = \frac{ac(1-x^2)}{(a-x)(c-x)}.
\ee
\end{thm}
The proof of Theorem~\ref{thm:ZsubsetK} is presented in the next section using shuffle algebra techniques.

\begin{cor}
    \label{Z sum Kuperberg intergral}
    Expression \eqref{eq:Zsubset} for $Z_m$ as a sum over subsets can routinely be converted to a contour integration over a family of contours $\mathcal{L}_i$ which all enclose each pole at $x_1,\ldots,x_m$ but omit all other singularities of the integrand
    \begin{multline}
        Z_m(x_1,\ldots,x_m)= H_m(x) \sum_{r=0}^{m} \frac{1}{r!} \left(\frac{1}{ -ac}\right)^{r/2} \oint_{\mathcal{L}_1} \frac{\dd v_1}{2\pi\ii} \cdots \oint_{\mathcal{L}_r} \frac{\dd v_{r}}{2\pi\ii} \prod_{1\le i< j\le r} \frac{v_j-v_i}{1-v_iv_j} \\
        \times\prod_{i=1}^{r} \left(\frac{v_i h(v_i)}{(1-v_i^2)(1-h(v_i))} \prod_{j=1}^m \frac{1-v_ix_j}{v_i-x_j} \right)\Pf \left( M (v_k,v_\ell) \right)_{1\le k,\ell \le r}.
        \label{eq:Zintegral}
    \end{multline}
\end{cor}
\begin{cor}
\label{cor:triangle-freeze}
For all $m \geq 1$, one has that
\begin{align}
\label{c to inf triangle}
\lim_{c \rightarrow \infty}
Z_m(x_1,\dots,x_m)
=
\lim_{c \rightarrow \infty}
\prod_{i=1}^m
(1-h(x_i))
=
\prod_{i=1}^m
\frac{x_i(1-ax_i)}{x_i-a}.
\end{align}
\end{cor}

\begin{proof}
Examining the sum-over-subsets formula \eqref{eq:Zsubset}, it is easily verified that the limit $c\to\infty$ eliminates all terms in the sum over $r$ except that corresponding to $r=0$. The claim \eqref{c to inf triangle} is then immediate.

Alternatively, one may prove \eqref{c to inf triangle} directly from the definition of the partition function \eqref{Z triangular partition function defn eq}, by noting that the $c \to \infty$ limit causes the third vertex in the table \eqref{Stochastic boundary K-weight table} to vanish. Since the boundary vertices then only have the option to inject (but never eject) paths, and no paths exit the partition function \eqref{Z triangular partition function defn eq} via its right-outgoing edges, it follows that the whole partition function is frozen as a product of empty vertices. The factorization \eqref{c to inf triangle} follows trivially.
\end{proof}

The sum over subsets \eqref{eq:Zsubset} can be compactly written in terms of Pfaffians in various ways. A particularly elegant expression is the following single Pfaffian expression for $Z_m$. \footnote{After completion of this work we became aware that an analogous formula was very recently also presented in \cite{behrend2023diagonally}.}
\begin{thm}
\label{thm:ZPfaff}
    When $m$ is even, the triangular partition function \eqref{Z triangular partition function defn eq} with the generic open boundary weights can be expressed in terms of a Pfaffian,
    \label{Z triangular partition function single Pfaffian thm}
    \begin{equation}
        \label{Z single Pfaffian them eq}
        Z_m(x_1,\dots,x_m) = \prod_{1\leq i<j\leq m} \frac{1-x_ix_j}{x_i-x_j} \cdot \pf \left(\frac{x_i-x_j}{1-x_ix_j} Q(x_i,x_j) \right)_{1\leq i,j\leq m},
    \end{equation}
    where $Q$ is a symmetric function in two variables given by
    \begin{equation}
        \label{eq: Q Pfaffian kernel}
        Q(x_i,x_j) = (1-h(x_i))(1-h(x_j)) - \frac{h(x_i)h(x_j)}{ac} \frac{(1-q)x_ix_j}{1-qx_ix_j}.  
    \end{equation}
\end{thm}
\begin{proof}
Theorem~\ref{thm:ZPfaff} follows from Theorem~\ref{thm:ZsubsetK} and the Pfaffian summation identity \eqref{eq:PFsum} in Lemma~\ref{lm:PFsum}.
\end{proof}
\begin{remark}
    In order to obtain an odd-sized solution from \eqref{Z single Pfaffian them eq} we would write, for $m=2\ell$,
    \[Z_{2\ell-1}(x_1,\dots,x_{2\ell-1}):=Z_{2\ell}(x_1,\dots,x_{2l-1},1),\]
    which makes use of the recursion relation \eqref{Z recur x=1}.
\end{remark}
We note that the Pfaffian kernel \eqref{eq: Q Pfaffian kernel} bears some resemblance to the one appearing in a refined Littlewood identity for spin Hall–Littlewood symmetric rational functions \cite{Gavrilova2023}, though is quite different due to the boundary factors.
\subsubsection{Cauchy summation identity revisited}
\label{sec: Cauchy identity Pfaffian}
\begin{cor}
\label{cor:cauchy-pfaff}
    Fix alphabets $(x_1,\dots,x_L),(z_1,\dots,z_M)$ and assume that there exists $\rho>0$ such that conditions \eqref{Cauchy identity condition} hold. Then the following Cauchy summation identity holds
    \begin{multline}
    \label{cauchy-cor}
        \sum_{\kappa} G_\kappa(x_1,\dots,x_L) F_{\kappa}(z_1,\dots,z_M) = \prod_{i=1}^M h(z_i) \prod_{i=1}^M\prod_{j=1}^L\left[\frac{x_j-qz_i}{x_j-z_i}\frac{1-z_ix_j}{1-qz_ix_j}\right]\\
       \times \prod_{1\leq i<j\leq L} \frac{1-x_ix_j}{x_i-x_j} \cdot \pf \left(\frac{x_i-x_j}{1-x_ix_j} Q(x_i,x_j) \right)_{1\leq i,j\leq L},
    \end{multline} 
    where $Q$ is given by \eqref{eq: Q Pfaffian kernel}.
\end{cor}
\begin{proof}
    The proof follows by using Theorem \ref{thm:ZPfaff} in the Cauchy summation identity of Corollary \ref{Cauchy identity cor}.
\end{proof}
We note here the parallel of \eqref{cauchy-cor} to the refined Cauchy identity of Macdonald polynomials from \cite{KirillovNoumi}, which is expressed as the product of the Macdonald Cauchy kernel and the Izergin\textendash Korepin determinant in \cite{warnaar2008}.

\subsection{Shuffle-exponential generating function}
\label{ssec:shuffle}
The partition function $Z_m$ and its generating function
\begin{align}
    \label{eq:Zgf}
    Z(v):= \sum_{m=0}^\infty v^m Z_m,
\end{align}
can both be conveniently written in terms of a \emph{shuffle product}. 
\begin{defn}\label{def:sh}
Let $f(x_1\dots x_k)$ and  $g(x_1\dots x_\ell)$ be two symmetric rational functions. We define the \emph{shuffle product} $f*g$ to be the symmetric rational function given by
\begin{align}
    \label{eq:sp}
    f * g = \sum_{\substack{S\subseteq [1,k+\ell]\\ |S|=k} } f(x_S) g(x_{S^c}) 
    \prod_{\substack{i\in S\\ j\in \comp{S}}}\frac{1-x_i x_j}{x_i-x_j}.
\end{align}
The identity with respect to the shuffle product is the rational symmetric function $1$ in zero number of arguments and $f*1=1*f=f$. Further, for any rational function $f(x_1\dots x_k)$ we define the \emph{shuffle power} and the \emph{shuffle exponential}, $\exp_*$, by
\begin{align}
    \label{def:sexp}
    f^{*j} := \underbrace{f*f*\cdots *f}_{\text{$j$ times}},
    \qquad
    \exp_*(f) := 1 + f + \tfrac{1}{2!} f^{*2} + \tfrac{1}{3!} f^{*3} +\cdots.
\end{align}
\end{defn}
From Definition \ref{def:sh} it follows that the shuffle product of $f(x_1\dots x_k)$ and  $g(x_1\dots x_l)$ is commutative unless both $k$ and $l$ are odd
\begin{align}
    \label{eq:sh_com}
    f*g = (-1)^{k l} g*f.
\end{align}
It can also be easily shown that this shuffle product is associative $(f*g)*h=f*(g*h)$. This shuffle product can be used to construct an algebra of functions and constitutes a convenient notation.
\begin{prop}\label{prop:Z-shuffle}
    Consider the first three triangular partition functions $Z_0,Z_1$ and $Z_2$. These can be explicitly calculated from the the diagram \eqref{Z triangular partition function defn eq} as
\begin{align}\label{eq:Z012}
    Z_0=1,\qquad 
    Z_1 = 1- h(x_1),
    \qquad
    Z_2 = \left(1- h(x_1)\right)\left(1- h(x_2)\right)
    -\frac{h(x_1) h(x_2)}{a c} \frac{(1-q)x_1 x_2 }{1-q x_1 x_2},
\end{align}    
    which are rational symmetric functions in $0,1$ and $2$ arguments $x$ respectively. The generating function $Z(v)$ takes form
    \begin{align}
        \label{eq:Zexp}
        Z(v) 
        =\exp_*(v^2 Z_2 + v Z_1). 
    \end{align}
This exponential formula is equivalent to
\begin{align}
    \label{eq:ZL_sh}
    Z_{2m} = \frac{1}{m!} Z_2^{*m},
    \qquad
    Z_{2m+1} = \frac{1}{m!} Z_1*Z_2^{*m}.
\end{align}
\end{prop}
\begin{proof}
First we note that the two terms in the exponent in \eqref{eq:Zexp} commute with each other and $Z_1^{*k}=0$ for $k>1$ due to \eqref{eq:sh_com}. Applying definitions \eqref{def:sexp} to the generating function \eqref{eq:Zexp} leads to \eqref{eq:ZL_sh}. Let us examine the expression for $Z_{2m}$ given by \eqref{eq:ZL_sh}. Computing $Z_2^{*m}$ produces a rational function with the (minimal) denominator
\begin{align}\label{eq:denZ}
    \prod_{i=1}^{2m}(a-x_i)(c-x_i) \prod_{1\leq i<j\leq 2m} (1-q x_i x_j).
\end{align}
This denominator is a polynomial of degree $2m+1$ in the individual $x_l$, $l=1,\dots,2m$. Let us fix $l$ and show that the limit $x_l\rightarrow \infty$ of $Z_{2}^{*m}$ exists. By writing $Z_2^{*m}$ using \eqref{def:sexp} and \eqref{eq:sp} we can see that $Z_2^{*m}$ is of the form
\begin{align}\label{eq:ZX}
    Z_2^{*m} = \sum \cdots Z_2(x_{i_1},x_{i_2})  \cdots  Z_2(x_{j_1},x_{j_2}) \cdots 
    \times \prod_{a,b=1,2}\frac{1-x_{i_a} x_{j_b}}{x_{i_a} - x_{j_b}} 
    \times \cdots ,
\end{align}
therefore in each summand the dependence on $x_\ell$ is of the form
\begin{align*}
    Z_2(x_\ell,x_{a}) \prod_{b} \frac{1-x_\ell x_b}{x_\ell-x_b},
\end{align*}
where $a$ and $b$ are some indices not equal to $l$. Computing $x_\ell\rightarrow \infty$ in both of these factors shows that this limit exists. Therefore $Z_2^{*m}$ is given by a ratio of a polynomial of degree at most $2m+1$ in $x_\ell$ and the polynomial in \eqref{eq:denZ}. This implies that in order to prove \eqref{eq:ZL_sh} for $Z_{2m}$ we need to show that $ Z_{2}^{*m}/m!$ satisfies the recursion relations of Proposition~\ref{Z trianglular recursion prop Z}. 

The specializations given in \eqref{Z recur x=0} and \eqref{Z recur x=1} follow from
\begin{align}
    \label{eq:x0}
    &Z_1(0)=0,\qquad\quad Z_2(x,0)=Z_2(0,x)=0,\\
    \label{eq:x1}
    &Z_1(1)=Z_0=1,
    \quad Z_2(x,1)=Z_2(1,x)=Z_1(x).
\end{align}
We have \eqref{eq:ZL_sh} satisfies \eqref{Z recur x=0} due to \eqref{eq:x0} and it satisfies \eqref{Z recur x=1} due to \eqref{eq:x1}. Consider next \eqref{Z recur x1=1/x2} and set $x_k = 1/x_\ell$ in $Z_2^{*m}$ as written in \eqref{eq:ZX} for any distinct $k, \ell =1\dots 2m$. In each term of the sum in \eqref{eq:ZX} the arguments $x_1\dots x_{2m}$ are distributed between various factors $Z_2$. Considering a generic summand we encounter two cases: either $k\in\{i_1,i_2\}$ and $\ell \in \{j_1,j_2\}$ or $k,\ell\in\{i_1,i_2\}$. In the first case the contribution is zero because of the factor which is explicitly written in \eqref{eq:ZX} and in the second case we compute 
\begin{multline*}
    \frac{1}{m!}Z_{2}^{*m}|_{x_k=1/x_\ell} = 
    \frac{1}{m!}
    \sum \cdots Z_2\left(x_{\ell},\frac{1}{x_{\ell}}\right)  \cdots  Z_2(x_{j_1},x_{j_2}) \cdots 
    \times \prod_{b=1,2}
    \frac{1-x_\ell x_{j_b}}{x_\ell - x_{j_b}} 
    \frac{1-\frac{1}{x_\ell} x_{j_b}}{\frac{1}{x_\ell} - x_{j_b}} 
    \times \cdots \\
    =\frac{1}{(m-1)!} \left(Z_2^{*(m-1)}\right)(\dots,\hat{x}_k,\dots,\hat{x}_\ell,\dots),
\end{multline*}
where we noted that $Z_2(x_l,1/x_l)=1$ and the explicitly written rational function is also equal to 1. There are in total $m$ different summands for which $k,\ell\in\{i_1,i_2\}$. All these summands are equal to each other and to the symmetric function $Z_2^{*(m-1)}$ which depends on $x_1,\dots, x_{2m}$ with $x_k,x_\ell$ omitted. These computations show that $Z_{2m}$ given by \eqref{eq:ZL_sh} satisfies the conditions of Proposition~\ref{Z trianglular recursion prop Z}. The case of $Z_{2m+1}$ can be proven analogously. 
\end{proof}
\begin{cor}
The summation formula for $Z_m$ given in Theorem \ref{thm:ZsubsetK} holds as a consequence of \eqref{eq:ZL_sh}. 
\end{cor}
\begin{proof}
We note that $Z_2$ in \eqref{eq:Z012} is given by a sum of two terms and therefore \eqref{eq:ZL_sh} can be expanded using the binomial theorem
\begin{align*}
    Z_{2m}&=\frac{1}{m!}
    \left((1-h(x_1))(1 -h(x_2)) +
\frac{-1}{a c}h(x_1)h(x_2)Z_2^K(x_1,x_2)\right)^{*m}\\
&=
\sum_{r=0}^m 
\left(\frac{-1}{a c}\right)^r
\frac{1}{(m-r)! r!}
\left((1-h(x_1))(1 -h(x_2))\right)^{*(m-r)}
*
\left(h(x_1)h(x_2)Z_2^K(x_1,x_2)\right)^{*r}.
\end{align*}
The two terms given by the shuffle powers $*(m-r)$ and $*r$ can be computed. For example, the second term is computed by observing that
$(Z_{2}^K)^{*r} = r! Z_{2r}^K$ as a consequence of \eqref{eq:ZL_sh} and \eqref{eq:Zkup}. After this
we can write the shuffle product of these two terms using \eqref{def:sh} and match the outcome with \eqref{eq:Zsubset}.

\end{proof}

\subsection{Alternative form of solution to recursion relations}
Theorem~\ref{Z triangular prod sum} below contains an alternative explicit expression for the triangular partition function $Z_m(x_1,\ldots,x_m)$ in terms of subset-sums over factorized expressions and valid for both $m$ even and odd. We first define $S$ by
\be
\label{eq:Sdef}
S(x_i,x_j) = \frac{x_i-x_j}{1-x_i x_j},
\ee
and let
\be
\begin{split}
	Q^{\mathrm{e}} (x_i,x_{j}) &= S(x_i,x_j) + \frac{u^2 q^{1/2}}{ac}\, x_i x_j\, h(x_i)h(x_j)\, S\left(q^{1/2}x_i,q^{1/2}x_j \right),\\
    Q^{\mathrm{o}}(x_i,x_{j}) &= x_ix_j S(x_i,x_j) + \frac{u^2}{q^{1/2}ac} h(x_i)h(x_j) S\left(q^{1/2}x_i,q^{1/2}x_j\right),
\end{split}
\ee
where $u$ is a generating parameter. Furthermore, we define the following functions in terms of Pfaffians
\begin{align}
	Z^\mathrm{e}_{2m} (u;x_1,\dots,x_{2m}) & = \prod_{1\leq i<j\leq 2m} \frac{1-x_i x_j}{x_i-x_j} \cdot \pf \left( Q^{\mathrm{e}}_{2m}(x_i,x_{j}) \right)_{1\le i,j \le 2m},
 \label{eq:Z even Pfaffian def}
 \\[2mm]
	Z^\mathrm{o}_{2m-1} (u;x_1,\dots,x_{2m-1}) & = \prod_{1\leq i<j\leq 2m-1} \frac{1-x_i x_j}{x_i-x_j} \cdot \pf \begin{pmatrix}
    \left(Q^{\mathrm{o}}_{2m-1}(x_i,x_{j})\right)_{1\le i<j\le 2m-1} & (-u h(x_i))_{1\le i \le 2m-1} \\
    (u h(x_j))_{1\le j \le 2m-1} & 0
    \end{pmatrix}.
\end{align}
We also set  
\be
    \begin{split}
	Z^\mathrm{e}_{2m-1} (u;x_1,\dots,x_{2m-1}) & =  Z^\mathrm{e}_{2m} (u;x_1,\dots,x_{2m-1},1),\\
	Z^\mathrm{o}_{2m} (u;x_1,\dots,x_{2m}) & = Z^\mathrm{o}_{2m+1} (u;x_1,\dots,x_{2m},1).
	\end{split}
 \label{eq:Z cross 1}
 \ee
Using these definitions we can now state the following theorem and corollary.
\begin{thm}
	\label{Z triangular prod sum}
	The triangular partition function \eqref{Z triangular partition function defn eq} is recovered by $Z_m(x_1,\dots,x_m) = Z_m(u=1;x_1,\dots,x_m)$ for $m\in\mathbb{N}$ with 
 \be
 \label{eq:z=Ze+Zo}
Z_m(u;x_1,\dots,x_m) = Z^\mathrm{e}_{m} (u;x_1,\dots,x_{m}) + Z^\mathrm{o}_{m} (u;x_1,\dots,x_{m}). 
 \ee
Furthermore, the partition function with generating parameter, $Z_m(u;x_1,\dots,x_m)$, can be written as
	\begin{equation}
		\label{Z triangular partition function S sum}
		Z_m(u;x_1,\dots,x_m) = \sum_{S\subseteq [1,m]} (-u)^\abs{S} g_S (x) \prod_{i\in S} h(x_i) \prod_{i\in S}\prod_{j\in\comp{S}} \frac{x_ix_j-1}{x_i-x_j} \prod_{1\leq i<j\leq m \atop i,j\in S} \frac{1-x_ix_j}{1-q x_i x_j},
    	\end{equation}
	where $r_S=\left\lfloor |S|/2\right\rfloor$ and
    \[g_S (x) = \frac{q^{r_S^2}}{(ac)^{r_S}} \begin{cases}
	\displaystyle	\prod_{i\in S} x_i, & \abs{S}\text{ is even} \\[4mm]
		\displaystyle	 \prod_{i\in\comp{S}} x_i, & \abs{S}\text{ is odd}
	\end{cases}.\]
 \end{thm}

\begin{remark}
The recursions \eqref{recur3} appear in \eqref{Z triangular partition function S sum} as residues of the simple poles at $x_k=1/qx_\ell$.
\end{remark}

\begin{proof}[Proof of Theorem~\ref{Z triangular prod sum}]
The equivalence of \eqref{eq:z=Ze+Zo} and \eqref{Z triangular partition function S sum} follows in a straightforward manner from the Pfaffian definitions of $Z_m^{\mathrm{e}}$ and  $Z_m^{\mathrm{o}}$, the Pfaffian identity \eqref{lm:PFsum} and from the fact that the Pfaffian of $S$ factorizes \cite{stembridge_nonintersecting_1990}, 
\be
\pf \left(S(x_i,x_j)\right)_{1\le i,j\le 2m} = \prod_{1\le i < j\le 2m} \frac{x_i-x_j}{1-x_i x_j}.
\label{eq:PfaffStembridge}
\ee
Next we need to show that $Z_m(x_1,\dots,x_m) = Z_m(u=1;x_1,\dots,x_m)$. We do that by computing the generating function
$$
Z(u;v) := \sum_{m=0}^{\infty}v^m Z_m(u;x_1,\dots,x_m)
$$
at $u=1$ with $Z_m(u=1;x_1,\dots,x_m)$ given by \eqref{Z triangular partition function S sum}. We will show that this generating function is equal to the generating function of $Z_m(x_1,\dots,x_m)$ \eqref{eq:Zexp} from Proposition \ref{prop:Z-shuffle}. Using the definition of the shuffle product \eqref{def:sh} we rewrite \eqref{Z triangular partition function S sum} with $u=1$ as
\begin{align}
\begin{split}
    \label{eq:susbet_shuffle}
    &Z_{2m}(u=1;x_1,\dots,x_{2m})=\sum_{j=0}^{2m} V_{2m-j}*W^{\mathrm{e}}_{j},
\\
    &Z_{2m+1}(u=1;x_1,\dots,x_{2m+1})=\sum_{j=0}^{2m+1} V_{2m+1-j}*W^{\mathrm{o}}_{j},
\end{split}
\end{align}
where we introduced symmetric functions $V_{m}=V_{m}(x_1,\dots,x_m)$
\begin{align}
\begin{split}
    \label{eq:VV}
    &V_{2m}= \frac{q^{m^2}}{(a c)^{m}} \prod_{i=1}^{2m} x_i h(x_i) \prod_{1\leq i<j\leq 2m} \frac{1- x_i x_j}{1-q x_i x_j}, \\
    &V_{2m+1}=-\frac{q^{m^2}}{(a c)^{m}} \prod_{i=1}^{2m+1} h(x_i) \prod_{1\leq i<j\leq 2m+1} \frac{1- x_i x_j}{1-q x_i x_j} ,
    \end{split}
\end{align}
and $W^{\mathrm{e}/\mathrm{o}}_{m}=W^{\mathrm{e}/\mathrm{o}}_{m}(x_1,\dots,x_m)$ are defined by
\begin{align}
    \label{eq:WW}
    W^\mathrm{e}_{2m}=1,
    \qquad 
    W^\mathrm{e}_{2m+1}=- x_1\cdots \,x_{2m+1},
    \qquad
    W^\mathrm{o}_{2m}= x_1\cdots \, x_{2m},
    \qquad
    W^\mathrm{o}_{2m+1}=1.
\end{align}
All of these functions also factorize with respect to the shuffle product
\begin{align}\label{eq:X}
    X_{2m} = \frac{1}{m!} X_2^{*m},
    \qquad
    X_{2m+1} = \frac{1}{m!} X_1*X_2^{*m},
    \qquad 
    \text{for}\quad X=V,W^\mathrm{e},W^\mathrm{o}.
\end{align}
From these formulas it follows that the generating functions of $V,W^\mathrm{e},W^\mathrm{o}$ can be expressed in terms of the shuffle exponential \eqref{def:sexp}. We compute the generating function $Z(u=1;v)$
\begin{align}\label{eq:Z_ev_odd}
    Z(u=1;v) = \sum_{m=0}^\infty v^{2m}Z_{2m}(u=1;x_1,\dots,x_{2m})+\sum_{m=0}^\infty v^{2m+1}Z_{2m+1}(u=1;x_1,\dots,x_{2m+1}),
\end{align}
using \eqref{eq:susbet_shuffle} and by representing each function $V_k,W_j^{\mathrm{e}/\mathrm{o}}$ in the form \eqref{eq:X}. The first summand in \eqref{eq:Z_ev_odd}
is computed as follows 
\begin{align}
     \sum_{m=0}^\infty v^{2m}Z_{2m}(u=1;x_1,\dots,x_{2m}) &=  \left(1+v^2 V_1*W_1^\mathrm{e}\right) *\text{exp}_* \left( v^2\left( V_2+W_2^\mathrm{e} \right)\right)\nonumber\\
    &= \text{exp}_* \left( v^2\left( V_2+W_2^\mathrm{e}+V_1*W_1^\mathrm{e} \right)\right) = 
    \text{exp}_* \left( v^2 Z_2\right),
    \label{eq:Zev_VW}
\end{align}
where $Z_2=Z_2(x_1,x_2)$ in the last expression is the triangular partition function for two sites. In \eqref{eq:Zev_VW} the second equality is due to the nilpotency of the shuffle product and the third equality is a consequence of the identity
\begin{equation}
    \label{eq:suffle VW id}
    V_2+W_2^\mathrm{e}+V_1*W_1^\mathrm{e} = Z_2.
\end{equation}
Let us remark that the numerator on the right hand side of \eqref{eq:suffle VW id} is a polynomial of degree $3$ in each $x_i$ while on the left hand side some terms have numerators which are polynomials of degree $4$ in individual $x_i$. In the above equation it is easy to check that the degree $4$ terms cancel on the left hand side. This phenomenon manifests itself if one tries to evaluate the degrees produced by the formula  \eqref{Z triangular partition function S sum}. The apparent degree is higher than expected and, in order to show the connection with $Z_m$, it is required to argue that \eqref{Z triangular partition function S sum} actually produces the correct degree. 

In the next step we calculate the generating function of the second term in \eqref{eq:Z_ev_odd}
\begin{align}
    \sum_{m=0}^\infty v^{2m+1}Z_{2m+1}(u=1;x_1,\dots,x_{2m+1})
    &=v \left(V_1+W_1^\mathrm{o}\right)*\text{exp}_* \left( v^2\left(V_2+W_2^\mathrm{o} \right)\right)
    \nonumber\\
    &=v Z_1*\text{exp}_* \left( v^2\left(Z_2 + Z_1*x_1 \right)\right)
    =v Z_1*\text{exp}_* \left( v^2 Z_2\right) 
    \label{eq:Zodd_VW}
\end{align}
where $Z_1=Z_1(x_1)$ and $Z_2=Z_2(x_1,x_2)$ in the second line are the triangular partition functions for one and two sites. In the second equality in \eqref{eq:Zodd_VW} we used 
\begin{align*}
    V_1+W_1^\mathrm{o} = Z_1 
    \qquad
    V_2+W_2^\mathrm{o}  = Z_2+Z_1*x_1
\end{align*}
and the last equality of \eqref{eq:Zodd_VW} is due to the nilpotency $Z_1^{*n}=0$, $n>1$. By combining \eqref{eq:Zodd_VW} with \eqref{eq:Zev_VW} in \eqref{eq:Z_ev_odd} we obtain the full generating function
\begin{align*}
        Z(u=1;v) 
        =\exp_*(v^2 Z_2 + v Z_1). 
\end{align*}
which coincides with \eqref{eq:Zexp} and therefore proves the statement of the Theorem. 
\end{proof}

\section{Integral formula for initially empty symmetric function}
\label{Vertex propagator section}
The central objects of this work are the two symmetric functions of Definition \ref{G partition function defn}. Theorem \ref{G=Z empty set thm} shows that the function $G_{\nu/\mu}$ reduces to the triangular partition function when both bottom and top configurations are empty. The previous section demonstrates how even when both conditions are empty this symmetric function is highly non-trivial. In this section we provide more insight into this behaviour by providing two equivalent evaluations of $G_{\nu/\mu}$ for arbitrary $\nu$ from an empty $\mu=\emptyset$. The form of this function leads to a striking conjecture on the orthogonality of the dual family $F_\kappa$.

\subsection{Subset formula}
\begin{thm}
    \label{G nu thm}
	Let the state on the bottom be empty while the arbitrary state on top $\nu = (\nu_1,\dots,\nu_n)$ consist of $n$ occupations at positions finitely far from the origin. Assume that that $L\geq n$ and there exists $\rho>0$ such that 
	\[\abs{\frac{1-x_iy_k}{1-qx_iy_k}\frac{q(1-x_j/y_k)}{1-qx_j/y_k}}\leq \rho <1,\]
    for all $1\leq i\neq j\leq L,k\in\mathbb{N}$. Then the partition function \eqref{G partition function defn A eq} is calculated explicitly as
	\begin{multline}
		\label{G nu sum expression thm eq}
		G_\nu (x_1,\dots,x_L|\yalph) = \sum_{K \subseteq [L] \atop |K|=n} Z_{L-n} (x_{\bar{k}_1},\dots,x_{\bar{k}_{L-n}}) \prod_{i\in K} h(x_i) \\
		\times \prod_{i\in K} \prod_{j\in \comp{K}} \left[\frac{x_j-qx_i}{x_j-x_i}\frac{1-x_i x_j}{1-q x_i x_j}\right] \prod_{\substack{1\leq i<j\leq L \\ i,j\in K}} \frac{1-x_{i} x_{j}}{1- q x_{i} x_{j}} \\
		\times \sum_{\sigma \in S_n} \prod_{1\leq i<j\leq n} \frac{x_{k_{\sigma(j)}}-q x_{k_{\sigma(i)}}}{x_{k_{\sigma(j)}}-x_{k_{\sigma(i)}}} \prod_{i=1}^n \left[  \frac{(1-q)x_{k_{\sigma(i)}}y_{\nu_i}}{1-q x_{k_{\sigma(i)}}y_{\nu_i}} \prod_{j=1}^{\nu_i-1}\frac{1- x_{k_{\sigma(i)}}y_j}{1- q x_{k_{\sigma(i)}}y_j}  \right].
	\end{multline}
	The outer sum is over subsets $K = \{k_1,\dots,k_n\}$ of $[L]=\{1,\dots,L\}$ with $n$ elements, whose complement is denoted $\comp{K}=\{\bar{k}_1,\dots,\bar{k}_{L-n}\}$.
\end{thm}
We will now prepare for the proof of this important result. Instead of proving that the partition function with empty initial condition \eqref{G nu picture lemma eq} is equal the rational function \eqref{G nu sum expression thm eq}, it is convenient to invert the family of vertical spectral parameters, with $\yalph^{-1} = \left(y_1^{-1},y_2^{-1},\dots\right)$, by considering the diagram from \eqref{G nu picture lemma eq}
\begin{equation}
        \label{G nu picture 1/y eq}
        G_\nu\left(x_1,\dots,x_L|\yalph^{-1}\right) = 
        \begin{tikzpicture}[baseline={([yshift=-.5ex]current bounding box.center)},scale=0.8]
			\draw[lightgray,line width=1.5pt,->] (-4,0) -- (-4,1) -- (7,1);
			\draw[lightgray,line width=1.5pt,->] (-3,0) -- (-3,2) -- (7,2);
			\draw[lightgray,line width=1.5pt,->] (-2,0) -- (-2,3) -- (7,3);
			\draw[lightgray,line width=1.5pt,->] (-1,0) -- (-1,4) -- (7,4);
			\draw[lightgray,line width=1.5pt,->] (0,0) -- (0,5) -- (7,5);

			\draw[lightgray,line width=1.5pt,->] (-4,0) -- (-4,0.5);
			\draw[lightgray,line width=1.5pt,->] (-3,0) -- (-3,0.5);
			\draw[lightgray,line width=1.5pt,->] (-2,0) -- (-2,0.5);
			\draw[lightgray,line width=1.5pt,->] (-1,0) -- (-1,0.5);
			\draw[lightgray,line width=1.5pt,->] (0,0) -- (0,0.5);

            \foreach \x in {1,...,6}
			{\draw[lightgray,line width=2pt,->] (\x,0) -- (\x,6);
            \node[below] at (\x,0) {$0$};}

			\draw[blue,fill=blue] (-4,1) circle (0.1cm);
			\draw[blue,fill=blue] (-3,2) circle (0.1cm);
			\draw[blue,fill=blue] (-2,3) circle (0.1cm);
			\draw[blue,fill=blue] (-1,4) circle (0.1cm);
			\draw[blue,fill=blue] (0,5) circle (0.1cm);

            \node[below] at (-4,0) {$0$};
            \node[below] at (-3,0) {$0$};
            \node[below] at (-2,0) {$0$};
            \node[below] at (-1,0) {$0$};
            \node[below] at (0,0) {$0$};

            \node[below] at (-4,-0.5) {$\uparrow$};
            \node[below] at (-3,-0.5) {$\uparrow$};
            \node[below] at (0,-0.5) {$\uparrow$};
            
            \node[below] at (-4,-1.1) {$x_1$};
            \node[below] at (-3,-1.1) {$x_2$};
            \node[below] at (-2,-1.1) {$\cdots$};
            \node[below] at (-1,-1.1) {$\cdots$};
            \node[below] at (0,-1.1) {$x_L$};

            \node[below] at (1,-0.5) {$\uparrow$};
            \node[below] at (2,-0.5) {$\uparrow$};
            \node[below] at (3,-0.5) {$\uparrow$};

            \node[below] at (1,-1.1) {$y_1^{-1}$};
            \node[below] at (2,-1.1) {$y_2^{-1}$};
            \node[below] at (3,-1.1) {$y_3^{-1}$};

            \node[below] at (4,-1.1) {$\cdots$};
            \node[below] at (5,-1.1) {$\cdots$};
            \node[below] at (6,-1.1) {$\cdots$};

            \node[right] at (7,1) {$0 \rightarrow x_1^{-1}$};
            \node[right] at (7,2) {$0 \rightarrow x_2^{-1}$};
            \node[right] at (7,3) {$0$};
            \node[right] at (7,4) {$0$};
            \node[right] at (7,5) {$0 \rightarrow x_L^{-1}$};

            \node[right] at (7.9,3) {$\vdots$};
            \node[right] at (7.9,4) {$\vdots$};

            \node[above] at (1,6) {$\eta^\nu_1$};
			\node[above] at (2,6) {$\eta^\nu_2$};
			\node[above] at (3,6) {$\eta^\nu_3$};
			\node[above] at (4,6) {$\cdots$};
			\node[above] at (5,6) {$\cdots$};   
            \node[above] at (6,6) {$\cdots$};   
		\end{tikzpicture}
    .\end{equation}
    Which we will show is equal to the following rational function formula, which is equivalent to \eqref{G nu sum expression thm eq} with inverted vertical spectral parameters
    \begin{multline}
		\label{G nu 1/y eq}
		\mathfrak{G}_\nu \left(x_1,\dots,x_L|\yalph^{-1}\right) = \sum_{K \subseteq [L] \atop |K|=n} Z_{L-n} (x_{\bar{k}_1},\dots,x_{\bar{k}_{L-n}}) \prod_{i\in K} h(x_i) \\
		\times \prod_{i\in K} \prod_{j\in \comp{K}} \left[\frac{x_j-qx_i}{x_j-x_i}\frac{1-x_i x_j}{1-q x_i x_j}\right] \prod_{\substack{1\leq i<j\leq L \\ i,j\in K}} \frac{1-x_{i} x_{j}}{1- q x_{i} x_{j}} \\
		\times \sum_{\sigma \in S_n} \prod_{1\leq i<j\leq n} \frac{x_{k_{\sigma(j)}}-q x_{k_{\sigma(i)}}}{x_{k_{\sigma(j)}}-x_{k_{\sigma(i)}}} \prod_{i=1}^n \left[  \frac{(1-q)x_{k_{\sigma(i)}}/y_{\nu_i}}{1-q x_{k_{\sigma(i)}}/y_{\nu_i}} \prod_{j=1}^{\nu_i-1}\frac{1- x_{k_{\sigma(i)}}/y_j}{1- q x_{k_{\sigma(i)}}/y_j}  \right].
	\end{multline}
     Before we present the proof of Theorem \ref{G nu thm}, we will need some important properties of the partition function which largely follow from Lemma \ref{G nu picture lemma}.
    \begin{lm}
        \label{G nu recursions lemma}
        The partition function from Definition \ref{G partition function defn} with empty initial condition satisfies the following properties. We note that for configuration $\nu\in\mathbb{W}$ with at least one occupation, The coordinate $\nu_1\in\mathbb{N}$ denotes the right-most occupation in $\nu$.
        \begin{enumerate}[label=\textbf{(\roman*)}]
            \item $G_\nu \left(x_1,\dots,x_L|\yalph^{-1}\right)$ is a meromorphic function in $y_{\nu_1}$. Its poles are all simple and occur at the points $y_{\nu_1} = q x_i$ for $1 \leq i \leq L$.
            \item $G_\nu \left(x_1,\dots,x_L|\yalph^{-1}\right)$ is symmetric in its alphabet $(x_1,\dots,x_L)$.
            \item The residue of $G_\nu \left(x_1,\dots,x_L|\yalph^{-1}\right)$ at its simple pole $y_{\nu_1} = qx_1$ is given by
            \begin{multline}
                \label{G nu 1/y residue eval}
                {\rm Res}_{y_{\nu_1}=qx_1} \Big[ G_\nu \left(x_1,\dots,x_L|\yalph^{-1}\right) \Big]
                =
                (1-q) x_1 h(x_1)
                \prod_{j=2}^{L}
                \frac{x_j-q x_1}{x_j-x_1}
                \frac{1-x_1 x_j}{1-q x_1 x_j}
                \\ \times
                \prod_{j=1}^{\nu_1-1}
                \frac{y_j-x_1}{y_j-q x_1}
                G_{(\nu_2,\dots,\nu_n)} \left(x_2,\dots,x_{L}|\yalph^{-1}\right).
            \end{multline}
            \item The limit in $y_{\nu_1}$ is 
            \begin{equation}
                \lim_{y_{\nu_1} \rightarrow \infty} G_\nu \left(x_1,\dots,x_L|\yalph^{-1}\right) = 0.
            \end{equation}
            \item When the coordinate is empty
            \begin{equation}
                G_\emptyset \left(x_1,\dots,x_L|\yalph^{-1}\right) = Z_L(x_1,\dots,x_L).
            \end{equation}
            
        \end{enumerate}
    \end{lm}
    \begin{proof}
        We will demonstrate the properties diagrammatically on the partition function\eqref{G nu picture 1/y eq}. 
        \begin{enumerate}[label=\textbf{(\roman*)}]
            \item From the diagram, the only dependence on $y_{\nu_1}$ is from the $\nu_1$'th column. The weights which contribute to the partition function from this column will be from \eqref{Stochastic 6VM weights table} with $z=x_i/y_{\nu_1}$, where $i$ corresponds to the rows $1\leq i\leq L$. All of the vertex configurations carry weights which are either entire functions of $y_{\nu_1}$ or are analytic except at the isolated point $y_{\nu_1}=qx_i$. Since these weights contribute the only dependence on $y_{\nu_1}$ we can conclude that $G_\nu\left(x_1,\dots,x_L|\yalph^{-1}\right)$ is a meromorphic function for all $y_{\nu_1}$ with possible singularities at the isolated points $y_{\nu_1}=qx_i$ for $1\leq i \leq L$.

            Each global path configuration on \eqref{G nu picture 1/y eq} is will feature a weight from each vertex in the $\nu_1$'th column at most once. Since each weight generates at most a simple pole at $y_{\nu_1}=qx_i$, we can conclude that the partition function will be a sum of rational functions with simple poles at $y_{\nu_1}=qx_i$. Therefore the poles at these points will be simple.
            
            \item This property follows from Corollary \ref{G symmetric cor}.
            \item Observing the boundary conditions, the vertex in the $\nu_1$'th column and first row has two possible vertex configurations. these are shown in the table below.
            \begin{align*}
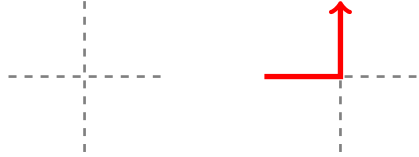
                	\begin{tabular}{cc}
                    	\qquad
                    	\begin{tikzpicture}
                    	\draw[gray,dashed,line width=1pt,-] (-1,0) -- (1,0);
                    	\draw[gray,dashed,line width=1pt,-] (0,-1) -- (0,1);
                    	\end{tikzpicture}
                    	\qquad
                    	&
                    	\qquad
                    	\begin{tikzpicture}
                    	\draw[gray,dashed,line width=1pt,-] (0,-1) -- (0,0) -- (1,0);
                    	\draw[red,line width=2pt,->] (-1,0) -- (0,0) -- (0,1);
                    	\end{tikzpicture}
                	\qquad
                	\end{tabular} 
                \end{align*}
            Since the all-empty configuration on the left has weight 1, lattice configurations where this vertex is empty will have a partition function contribution which are analytic at $y_{\nu_1}=qx_1$. When the other weight is involved the contribution will have a simple pole at $y_{\nu_1}=qx_1$. By taking the residue of the whole partition function at the point $y_{\nu_1}=qx_1$ we isolate contributions where this vertex is non-empty. Such configurations are depicted in the following diagram 
            
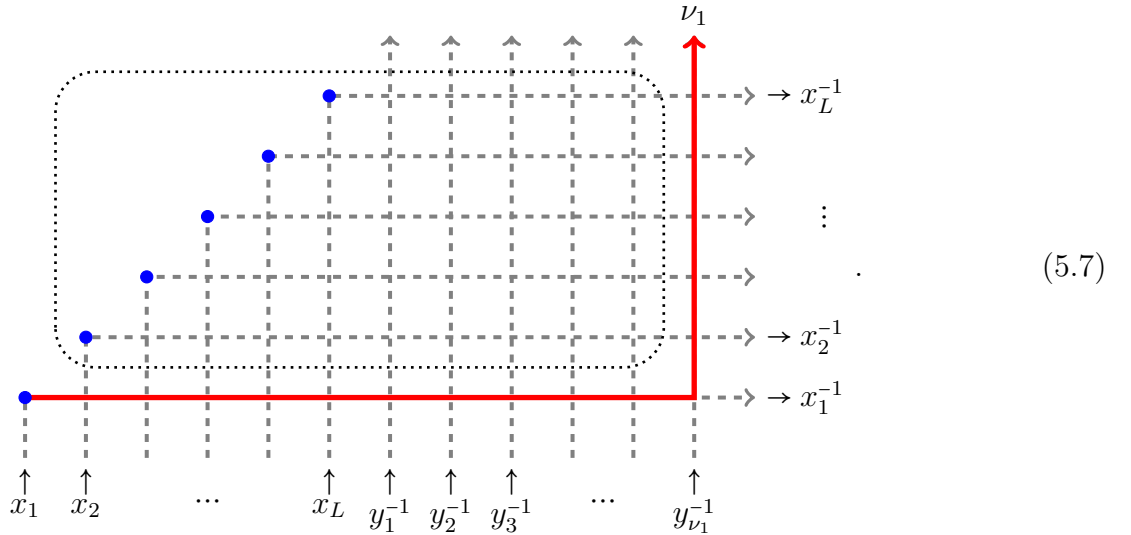
\begin{equation}
                \label{G nu pic recursions proof diagram}
                \begin{tikzpicture}[baseline={([yshift=-.5ex]current bounding box.center)},scale=0.8]
                    \draw[gray,dashed,line width=1.5pt,->] (-5,-1) -- (-5,0) -- (7,0);
        			\draw[gray,dashed,line width=1.5pt,->] (-4,-1) -- (-4,1) -- (7,1);
        			\draw[gray,dashed,line width=1.5pt,->] (-3,-1) -- (-3,2) -- (7,2);
        			\draw[gray,dashed,line width=1.5pt,->] (-2,-1) -- (-2,3) -- (7,3);
        			\draw[gray,dashed,line width=1.5pt,->] (-1,-1) -- (-1,4) -- (7,4);
        			\draw[gray,dashed,line width=1.5pt,->] (0,-1) -- (0,5) -- (7,5);

                    \foreach \x in {1,...,6}
        			{\draw[gray,dashed,line width=1.5pt,->] (\x,-1) -- (\x,6);
                    }
        
                    \draw[red,line width=2pt,->] (-5,0) -- (6,0) -- (6,6);
        
                    \draw[blue,fill=blue] (-5,0) circle (0.1cm);
        			\draw[blue,fill=blue] (-4,1) circle (0.1cm);
        			\draw[blue,fill=blue] (-3,2) circle (0.1cm);
        			\draw[blue,fill=blue] (-2,3) circle (0.1cm);
        			\draw[blue,fill=blue] (-1,4) circle (0.1cm);
        			\draw[blue,fill=blue] (0,5) circle (0.1cm);
        
                     \draw[dotted,line width=1pt,rounded corners=15pt]  (-4.5,0.5) rectangle ++(10,4.9);
        
                    \node[below] at (-5,-1) {$\uparrow$};
                    \node[below] at (-4,-1) {$\uparrow$};
                    \node[below] at (0,-1) {$\uparrow$};
                    
                    \node[below] at (-5,-1.5) {$x_1$};
                    \node[below] at (-4,-1.5) {$x_2$};
                    \node[below] at (-2,-1.5) {$\cdots$};
                    \node[below] at (0,-1.5) {$x_{L}$};
        
                    \node[below] at (1,-1) {$\uparrow$};
                    \node[below] at (2,-1) {$\uparrow$};
                    \node[below] at (3,-1) {$\uparrow$};
                    \node[below] at (6,-1) {$\uparrow$};
        
                    \node[below] at (1,-1.5) {$y_1^{-1}$};
                    \node[below] at (2,-1.5) {$y_2^{-1}$};
                    \node[below] at (3,-1.5) {$y_3^{-1}$};
                    \node[below] at (6,-1.5) {$y_{\nu_1}^{-1}$};
        
                    \node[below] at (4.5,-1.5) {$\cdots$};
        
                    \node[right] at (7,0) {$\rightarrow x_1^{-1}$};
                    \node[right] at (7,1) {$\rightarrow x_2^{-1}$};
                    \node[right] at (7,5) {$\rightarrow x_{L}^{-1}$};
        
                    \node[right] at (7.9,3) {$\vdots$};
        
                    \node[above] at (6,6) {$\nu_1$};   
        		\end{tikzpicture}.
            \end{equation}
            By fixing the configuration at this vertex, we really freeze the contribution along the whole line associated with spectral parameter $x_1$ and the line associated with vertical spectral parameter $y_{\nu_1}$. This freezing passes on the empty boundary conditions below the first line the below the second line. Likewise it enforces the empty conditions on the right of the $\nu_1$'th column to the $(\nu_1-1)$'th column.

            After removing the frozen contribution when taking the residue, what is left in the rectangle in \eqref{G nu pic recursions proof diagram} is that of the same partition function \eqref{G nu picture 1/y eq} with a $n-1$ coordinates $(\nu_2,\dots,\nu_n)$ and $L-1$ rows with spectral parameters $(x_2,\dots,x_L)$. 

            Taking this residue can be written as
            \begin{multline}
                {\rm Res}_{y_{\nu_1}=qx_1} \Big[ G_\nu \left(x_1,\dots,x_L|\yalph^{-1}\right) \Big]
                = h(x_1) \prod_{j=2}^n \frac{1-x_1x_j}{1-qx_1x_j}\prod_{j=1}^{\nu_1-1} \frac{y_j-x_1}{y_j-qx_1}  G_{(\nu_2,\dots,\nu_n)} \left(x_2,\dots,x_{L}|\yalph^{-1}\right)\\
                \times\lim_{y_{\nu_1}\to qx_1} \left[ (y_{\nu_1}-q x_1)\frac{x_1(1-q)}{y_{\nu_1}-qx_1}\prod_{j=2}^n \frac{q(y_{\nu_1}-x_j)}{y_{\nu_1}-qx_j}\right],
            \end{multline}
            which can be easily manipulated to take the form of \eqref{G nu 1/y residue eval}. We note here we can include the entire inverted alphabet $\yalph^{-1}$ and remove it from the evaluation of the limit since the smaller partition function will only have explicit dependence on $y_j$ for $1\leq j\leq \nu_2$.
            \item From the table of weights \eqref{Stochastic 6VM weights table}, the configuration
            
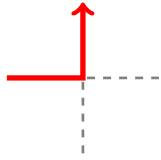
\[ \begin{tikzpicture}
                \draw[gray,dashed,line width=1pt,-] (0,-1) -- (0,0) -- (1,0);
	           \draw[red,line width=2pt,->] (-1,0) -- (0,0) -- (0,1);
            \end{tikzpicture}\]
            has weight $(1-q)x_i/(y_{\nu_1}-qx_i)$ when the horizontal and vertical spectral parameters are $x_i^{-1}$ and $y_{\nu_1}^{-1}$ respectively. In the limit $y_{\nu_1}\to\infty$ this weight is equal to zero and will not occur with the $\nu_1$'th column of the partition function. 
            
            However when observing the boundary conditions of \eqref{G nu picture 1/y eq}, we note that there is a path exit through the top of the $\nu_1$'th column while there are no occupations on the bottom entry or to the right of this column. Consequently, this weight must appear in the $\nu_1$'th column exaclty once for any configuration to provide a non-zero contribution. 
            
            Meanwhile, all other vertex configurations within \eqref{Stochastic 6VM weights table} will not diverge under the same the limit. This is sufficient to conclude that the limit of the whole partition function will evaluate to zero.
            
            \item This property follows from Theorem \ref{G=Z empty set thm}. 
        \end{enumerate}
    \end{proof}
    We will now proceed with the proof of the theorem.
    \begin{proof}[Proof of Theorem \ref{G nu thm}]
        Let us write \eqref{G nu 1/y eq} in the more compact form
        \begin{align}
        \label{compact}
        \mathfrak{G}_\nu \left(x_1,\dots,x_L|\yalph^{-1}\right)
        =
        \sum_{K \subseteq [L] \atop |K|=n}
        Z_{L-n}\left(x_{\comp{K}}\right)
        h(x_K)
        \Delta(x_K|x_{\comp{K}})
        \Phi_{\nu}\left(x_K| \yalph^{-1}\right),
        \end{align}
        where we have defined
        \begin{align*}
        h(x_K) & = \prod_{i\in K} h(x_i), \\
        \Delta(x_K|x_{\comp{K}}) &
        =
        \prod_{i\in K} \prod_{j\in \comp{K}} \left[\frac{x_j-qx_i}{x_j-x_i}\frac{1-x_i x_j}{1-q x_i x_j}\right] \prod_{\substack{1\leq i<j\leq L \\ i,j\in K}} \frac{1-x_{i} x_{j}}{1- q x_{i} x_{j}},
        \\ \\
        \Phi_{\nu}\left(x_K|\yalph^{-1}\right) &
        =
        \sum_{\sigma \in S_n} 
        \prod_{1\leq i<j\leq n} 
        \frac{x_{k_{\sigma(j)}}-q x_{k_{\sigma(i)}}}{x_{k_{\sigma(j)}}-x_{k_{\sigma(i)}}} 
        \prod_{i=1}^n \left[  \frac{(1-q)x_{k_{\sigma(i)}}/y_{\nu_i}}{1-q x_{k_{\sigma(i)}}/y_{\nu_i}} \prod_{j=1}^{\nu_i-1}\frac{1- x_{k_{\sigma(i)}}/y_j}{1- q x_{k_{\sigma(i)}}/y_j}  \right].
        \end{align*}
        We shall begin by proving that \eqref{compact} obeys the same set of properties as $G_\nu \left(x_1,\dots,x_L|\yalph^{-1}\right)$ from \eqref{G nu picture 1/y eq}. These are the properties in Lemma \ref{G nu recursions lemma} which provide a recursive construction for the formula \eqref{G nu 1/y eq} with an initial condition, and so completely define the formula for \eqref{G nu picture 1/y eq} through an inductive argument.
        
        \begin{enumerate}[label=\textbf{(\roman*)}]
        
        \item $\mathfrak{G}_\nu \left(x_1,\dots,x_L|\yalph^{-1}\right)$ is a meromorphic function in $y_{\nu_1}$. Its poles are all simple and occur at the points $y_{\nu_1} = q x_i$, $1 \leq i \leq L$. This property is immediate from the formula for $\Phi_{\nu}\left(x_K|\yalph^{-1}\right)$, which is the only place where $\mathfrak{G}_\nu \left(x_1,\dots,x_L|\yalph^{-1}\right)$ has dependence on the family $\yalph$.
        
        \item $\mathfrak{G}_\nu \left(x_1,\dots,x_L|\yalph^{-1}\right)$ is symmetric in its alphabet $(x_1,\dots,x_L)$. This is manifest from the form \eqref{compact} of $\mathfrak{G}_\nu \left(x_1,\dots,x_L|\yalph^{-1}\right)$.
        
        \item The residue of $\mathfrak{G}_\nu \left(x_1,\dots,x_L|\yalph^{-1}\right)$ at its simple pole $y_{\nu_1} = qx_1$ is given by
        \begin{multline*}
        {\rm Res}_{y_{\nu_1}=qx_1} \Big[ \mathfrak{G}_\nu \left(x_1,\dots,x_L|\yalph^{-1}\right) \Big]
        =
        (1-q) x_1 h(x_1)
        \prod_{j=2}^{L}
        \frac{x_j-q x_1}{x_j-x_1}
        \frac{1-x_1 x_j}{1-q x_1 x_j}
        \\ \times
        \prod_{j=1}^{\nu_1-1}
        \frac{y_j-x_1}{y_j-q x_1}
        \mathfrak{G}_{(\nu_2,\dots,\nu_n)} \left(x_2,\dots,x_{L}|\yalph^{-1}\right).
        \end{multline*}
        This is easily seen by computing
        \begin{align*}
            {\rm Res}_{y_{\nu_1}=qx_1} \Big[ \Phi_{\nu}\left(x_K|\yalph^{-1}\right) \Big]
            =
            \begin{cases}
                (1-q)x_1 \displaystyle{\prod_{j \in K \backslash \{1\}}}
                \frac{x_j - qx_1}{x_j - x_1}
                \displaystyle{\prod_{j=1}^{\nu_1-1}}
                \frac{y_j-x_1}{y_j-q x_1}
                \Phi_{(\nu_2,\dots,\nu_n)}\left(x_{K\backslash \{1\}}|\yalph^{-1}\right) & \text{ if } 1\in K\\
                \quad
                \\
                0 & \text{ if } 1\notin K
            \end{cases},
        \end{align*}
        and noting that for $1 \in K$, we have
        \begin{align*}
        \Delta(x_K|x_{\comp{K}})
        =
        \prod_{j \in \comp{K}}
        \frac{x_j - qx_1}{x_j - x_1}
        \prod_{j=2}^{L}
        \frac{1-x_1 x_j}{1-q x_1 x_j}
        \Delta(x_{K\backslash \{1\}} |x_{\comp{K}}),
        \qquad\
        h(x_K) = h(x_1) h(x_{K \backslash\{1\}}).
        \end{align*}
        
        \item $\mathfrak{G}_\nu \left(x_1,\dots,x_L|\yalph^{-1}\right) \rightarrow 0$ as $y_{\nu_1} \rightarrow \infty$. This follows by computing this limit directly on $\Phi_{\nu}\left(x_K|\yalph^{-1}\right)$.
        
        \item $\mathfrak{G}_{\varnothing} \left(x_1,\dots,x_L|\yalph^{-1}\right) = Z_L(x_1,\dots,x_L)$. This is simply the $n=0$ case of the formula \eqref{compact}.
        
        \end{enumerate}
        
        We have shown that $\mathfrak{G}_\nu \left(x_1,\dots,x_L|\yalph^{-1}\right)$ obeys the same set of properties as $G_\nu \left(x_1,\dots,x_L|\yalph^{-1}\right)$ does according to Lemma \ref{G nu recursions lemma}. It remains to show that these properties imply the equality of the two objects; we do this by induction on the length of $\nu$. To that end, define the function
        \begin{align*}
        \mathfrak{Z}_{\nu}(x_1,\dots,x_L|\yalph) = \mathfrak{G}_\nu \left(x_1,\dots,x_L|\yalph^{-1}\right) - G_\nu \left(x_1,\dots,x_L|\yalph^{-1}\right).
        \end{align*}
        By construction, $\mathfrak{Z}_{\varnothing}(x_1,\dots,x_L|\yalph)=0$. It follows that there exists an integer $m \geq 0$ such that $\mathfrak{Z}_{\mu}(x_1,\dots,x_L|\yalph)=0$ for all strict partitions $\mu = (\mu_1>\cdots>\mu_m)$ of length $m$ (with $L$ being arbitrary); this is our inductive hypothesis. 
        
        Now let $\lambda = (\lambda_1>\cdots>\lambda_{m+1})$ be a strict partition of length $m+1$. We know that $\mathfrak{Z}_{\lambda}(x_1,\dots,x_L|\yalph)$ is a meromorphic function in $y_{\lambda_1}$, its poles are all simple, and it vanishes as $y_{\lambda_1} \rightarrow \infty$. However, from the recursion relation obeyed by $\mathfrak{G}_{\lambda}\left(x_1,\dots,x_L|\yalph^{-1}\right)$ and $G_{\lambda}\left(x_1,\dots,x_L|\yalph^{-1}\right)$, as well as the inductive hypothesis, all poles have vanishing residue. This means that $\mathfrak{Z}_{\lambda}(x_1,\dots,x_L|\yalph)$ is entire and bounded in $y_{\lambda_1}$ and therefore constant. This constant must be zero in view of the known $y_{\lambda_1} \rightarrow \infty$ behaviour. It follows that $\mathfrak{Z}_{\lambda}(x_1,\dots,x_L|\yalph) = 0$ for all strict partitions $\lambda = (\lambda_1>\cdots>\lambda_{m+1})$ of length $m+1$, and the inductive step of the proof is complete.
    \end{proof}
\subsection{Integral formula}
Here we will present the sum over subset expression \eqref{G nu sum expression thm eq} as an equivalent nested integral formula. 
\begin{defn}
    \label{defn: nested contours}
    Fix an alphabet $(x_1,\dots,x_L)\in\mathbb{C}^L$. We denote by $\mathcal{C}_1,\dots,\mathcal{C}_n$ a collection of positively oriented closed complex contours satisfying 
    \begin{itemize}
        \item For all $1\leq i<j\leq n$, we have that $q\mathcal{C}_j$ lies completely outside the interior of $\mathcal{C}_i$, where $q\mathcal{C}_j$ denotes the image of $\mathcal{C}_j$ under multiplication by $q$. In addition, if $1$ is within the interior of $\mathcal{C}_i$ we also require that $\mathcal{C}_i$ is completely contained in the interior of $\mathcal{C}_j$.
        \item For all $1\leq i\leq L$, the contour $\mathcal{C}_i$ surrounds all points $x_j$ and does not surround the points $qx_j,q^{-1}x_j^{-1},q^{-1}y_k^{-1},a,c$ for all $1\leq j\leq L$ and $k\in\mathbb{N}$. 
    \end{itemize}
\end{defn}
Examples of contours satisfying the conditions of Definition \ref{defn: nested contours} are shown in Figure \ref{fig:contour diagram}. These conditions allow for some freedom with contour choice. In particular we may choose all contours to be equal, $\mathcal{C}_i=:\mathcal{C}$ for all $i$,  provided that the contours neither lie upon nor enclose $1\in\mathbb{C}$. 

Provided that the integrand considered has no singularity at 1, the contours may surround 1 if we choose that they are nested. That is, for all $1\leq i\leq n-1$ the contour $\mathcal{C}_i$ is completely contained within the interior of $\mathcal{C}_{i+1}$. This choice is convenient as it allows us to choose elements of our alphabet, $x_i$, to be arbitrarily close to 1.

\begin{thm}
    \label{thm: G nu inetgral formula}
	The partition function expression \eqref{G nu sum expression thm eq} can be expressed as the following $n$-fold integral:  
	\begin{multline}
        \label{eq: G nu integral forumla thm}
	    G_\nu(x_1,\dots,x_L|\yalph) = \oint_{\mathcal{C}_1}\frac{\dd w_1}{2\pi\ii} \cdots \oint_{\mathcal{C}_n}\frac{\dd w_n}{2\pi\ii} Z_{L+n} \left(x_1,\dots,x_L,w_1^{-1},\dots,w_n^{-1}\right) \\
        \times\prod_{i=1}^n\prod_{j=1}^L \left[\frac{q w_i-x_j}{w_i-x_j}\frac{1-w_ix_j}{1-qw_ix_j}\right] \prod_{1\leq i<j\leq n} \left[\frac{w_j-w_i}{qw_j-w_i}\frac{1-qw_iw_j}{1- w_iw_j}\right] \\ \times \prod_{i=1}^n \left[\frac{ac\left(q w_i^2-1\right)}{(w_i-a)(w_i-c)}\frac{y_{\nu_i}}{1-q w_i y_{\nu_i}}\prod_{j-1}^{\nu_i-1} \frac{1-w_iy_j}{1-qw_iy_j}\right],
	\end{multline}
    where the contours $\mathcal{C}_1,\dots,\mathcal{C}_n$ satisfy the conditions of Definition \ref{defn: nested contours}.
\end{thm}
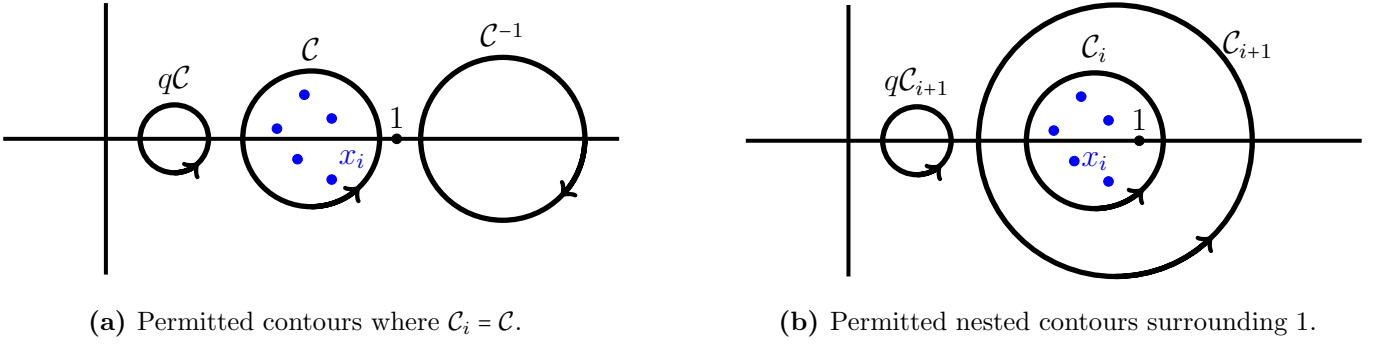
\begin{figure}
    \centering
    \begin{subfigure}{0.45\textwidth}
        \[\begin{tikzpicture}[baseline={([yshift=-.5ex]current bounding box.center)},scale=0.9]
            \draw[black,line width=1.5pt] (-1.5,0) -- (7.5,0);
            \draw[black,line width=1.5pt] (0,-2) -- (0,2);
            \draw[black,line width=2pt,->] (3,-1) arc (-90:315:1);
            \node[above] at (3,1) {$\textcolor{black}{\mathcal{C}}$};
            \draw[black,line width=2pt,->] (1,-0.5) arc (-90:315:0.5);
            \node[above] at (1,0.5) {$\textcolor{black}{q\mathcal{C}}$};
            \draw[black,line width=2pt,->] (7,0) arc (360:-45:1.2);
            \node[above] at (5.8,1.2) {$\textcolor{black}{\mathcal{C}^{-1}}$};
            \draw[black,fill=black] (4.25,0) circle (0.07);
            \node[above] at (4.25,0) {$1$};
            \draw[blue,fill=blue] (3.3,0.3) circle (0.07);
            \draw[blue,fill=blue] (2.5,0.15) circle (0.07);
            \draw[blue,fill=blue] (2.9,0.65) circle (0.07);
            \draw[blue,fill=blue] (2.8,-0.3) circle (0.07);
            \draw[blue,fill=blue] (3.3,-0.6) circle (0.07);
            \node at (3.6,-0.3) {\textcolor{blue}{$x_i$}};
        \end{tikzpicture}\]
        \label{fig: homog contour}
        \caption{Permitted contours where $\mathcal{C}_i=\mathcal{C}$.}
    \end{subfigure}
    \hspace{\fill}
    \begin{subfigure}{0.45\textwidth}
        \[\begin{tikzpicture}[baseline={([yshift=-.5ex]current bounding box.center)},scale=0.9]
            \draw[black,line width=1.5pt] (-1.5,0) -- (7.5,0);
            \draw[black,line width=1.5pt] (0,-2) -- (0,2);
            \draw[black,line width=2pt,->] (3.6,-1) arc (-90:315:1);
            \node[above] at (3.6,1) {$\textcolor{black}{\mathcal{C}_i}$};
            \draw[black,line width=2pt,->] (1,-0.5) arc (-90:315:0.5);
            \node[above] at (1,0.5) {$\textcolor{black}{q\mathcal{C}_{i+1}}$};
            \draw[black,line width=2pt,->] (3.9,-2) arc (-90:315:2);
            \node[above,right] at (5.314,1.414) {$\textcolor{black}{\mathcal{C}_{i+1}}$};
            \draw[black,fill=black] (4.25,0) circle (0.07);
            \node[above] at (4.25,0) {$1$};
            \draw[blue,fill=blue] (3.8,0.3) circle (0.07);
            \draw[blue,fill=blue] (3,0.15) circle (0.07);
            \draw[blue,fill=blue] (3.4,0.65) circle (0.07);
            \draw[blue,fill=blue] (3.3,-0.3) circle (0.07);
            \draw[blue,fill=blue] (3.8,-0.6) circle (0.07);
            \node at (3.6,-0.3) {\textcolor{blue}{$x_i$}};
        \end{tikzpicture}\]
        \caption{Permitted nested contours surrounding 1.}
        \label{fig: inhomog contour}
    \end{subfigure}
    \caption{Diagrams depicting arrangements of contours allowed by Defintion \ref{defn: nested contours}.}
    \label{fig:contour diagram}
\end{figure}
\begin{proof} 
    The idea of the proof is to replace each $x_i$ for $i\in K$ in the sum in \eqref{G nu sum expression thm eq} by an auxiliary variable $w_i$ that will be integrated over a contour surrounding simple poles at all $x_1,\ldots x_L$. The sum over $K$ then dictates which $n$ of the $L$ possible residues are evaluated, whilst the sum over $\sigma\in S_n$ dictates the order in which the residues are evaluated for a given $K$.

    For this to work, only the residues at $w_i=x_j$ in \eqref{eq: G nu integral forumla thm} should be evaluated. All other poles in \eqref{eq: G nu integral forumla thm} therefore need to be excluded from the contours, and this is guaranteed by Definition~\ref{defn: nested contours}. Firstly, it is obvious that the explicit poles in \eqref{eq: G nu integral forumla thm} at $a$, $c$, $q^{-1}y_j^{-1}$ and $q^{-1}x_j^{-1}$ need to lie outside each contour. Secondly, the poles at $q x_j$ also need to be avoided because these cause singularities in $Z_{L+n}$, see below. Moreover, for $1\le i < j \le n$, the factors of the form
    \[
\frac{w_j-w_i}{qw_j-w_i} \frac{1-qw_iw_j}{1-w_iw_j}, 
    \]
    produce potential residues which will be avoided when $q\mathcal{C}_j$ lies outisde the interior of $\mathcal{C}_i$.
    
    In order to reproduce the triangular partition function $Z_{L-n}$ that appears in the summand of \eqref{G nu sum expression thm eq} we extend this function to $Z_{L+n}$ in the combined alphabet $\left(x_1,\ldots,x_L,w_1^{-1},\ldots w_n^{-1}\right)$. During the evaluation of the residue of the simple pole at each $w_i=x_j$, the recursion relation \eqref{Z recur x1=1/x2} ensures that we re-obtain $Z_{L-n}$ in the complement alphabet of $x$-variables of \eqref{G nu sum expression thm eq}. Furthermore, according to Proposition~\ref{Z polynomial prop}, the rational function $Z_{L+n} \left(x_1,\dots,x_L,w_1^{-1},\dots,w_n^{-1}\right)$ has poles at $w_i=a^{-1}$, $w_i=c^{-1}$ and $w_i=qx_j$ and so has no singularities at $w_i=x_j$ that could affect the residue evaluation. 
    
    In order to proceed, we rewrite the following factor that occurs in \eqref{G nu sum expression thm eq},
    \begin{multline}
        \label{G nu integral proof 1}
        \prod_{i\in K} \prod_{j\in \comp{K}} \left[\frac{x_j-qx_i}{x_j-x_i}\frac{1-x_i x_j}{1-q x_i x_j}\right] = \prod_{i\in K} \prod_{j\in \comp{K}}\frac{1}{x_i-x_j}\prod_{i\in K}\prod_{j=1}^L \frac{(qx_i-x_j)(1-x_ix_j)}{1-qx_ix_j} \\
        \times \prod_{\substack{i\neq j \\ i,j\in K}} \frac{1-qx_ix_j}{(qx_i-x_j)(1-x_ix_j)} \prod_{i\in K} \frac{qx_i^2-1}{(1-q)x_i\left(1-x_i^2\right)}.
    \end{multline}
    Incorporating this we notice that for each $K$ the second line in \eqref{G nu sum expression thm eq} is manifestly symmetric in the variables $x_i$ for $i \in K$. This allows us to replace each $x_i$ for $i\in K$ with $w_i$, i.e. the right hand side of \eqref{G nu integral proof 1} is replaced by
    \be
    \prod_{i,j=1 \atop i\neq j}^n (w_i-w_j) \prod_{i=1}^n\prod_{j=1}^L \frac{(qw_i-x_j)(1-w_ix_j)}{(w_i-x_j)(1-qw_ix_j)}
        \prod_{i,j=1 \atop i\neq j}^n \frac{1-qw_iw_j}{(qw_i-w_j)(1-w_iw_j)} \prod_{i=1}^n \frac{qw_i^2-1}{(1-q)w_i\left(1-w_i^2\right)},
    \ee    
    and each $x_{\sigma(i)}$ and  $x_{\sigma(j)}$ in \eqref{G nu sum expression thm eq} is similarly replaced by $w_i$ and $w_j$ respectively. Simplifying and cancelling common factors we thus obtain the integrand of \eqref{eq: G nu integral forumla thm}. Finally we note that the factors $w_j-w_i$ in \eqref{eq: G nu integral forumla thm} ensure that after evaluating the residue of $w_i$ at $x_k$ the singularity at $w_j=x_k$ is removable for all $j\neq i$, and hence that the residue for each $w_i$ is evaluated at a different simple pole.
\end{proof}

\subsection{An orthogonality conjecture}

In this section we examine an interesting implication of the Cauchy identity \eqref{cauchy-cor} and integral formula \eqref{eq: G nu integral forumla thm} when $c \to \infty$. Our starting point is the observation made in Corollary \ref{cor:triangle-freeze} regarding the $c \to \infty$ limit of the triangular partition function \eqref{Z triangular partition function defn eq} which we recall now:
\begin{align}
\label{c to inf triangle 2}
\lim_{c \rightarrow \infty}
Z_m(x_1,\dots,x_m)
=
\lim_{c \rightarrow \infty}
\prod_{i=1}^m
(1-h(x_i))
=
\prod_{i=1}^m
\frac{x_i(1-ax_i)}{x_i-a}.
\end{align}
Making use of Corollary \ref{cor:triangle-freeze} in the Cauchy identity \eqref{cauchy-cor}, it reads
\begin{multline}
    \label{cauchy-cor-cinf}
    \lim_{c \to \infty}
        \sum_{\kappa} G_\kappa(x_1,\dots,x_L) F_{\kappa}(z_1,\dots,z_M) = \prod_{i=1}^M \frac{a (1-z_i^2)}{a-z_i} \prod_{i=1}^M\prod_{j=1}^L\left[\frac{x_j-qz_i}{x_j-z_i}\frac{1-z_ix_j}{1-qz_ix_j}\right] \prod_{i=1}^L
\frac{x_i(1-ax_i)}{x_i-a}.
    \end{multline}
    A similar application of Corollary \ref{cor:triangle-freeze} inside the integral formula \eqref{eq: G nu integral forumla thm} yields, after redistribution of factors in the integrand,
    \begin{multline}
        \label{eq: G nu integral forumla thm-cinf}
	    \lim_{c \to\infty} G_\nu(x_1,\dots,x_L) = \prod_{i=1}^L
\frac{x_i(1-ax_i)}{x_i-a} \oint_{\mathcal{C}_1}\frac{\dd w_1}{2\pi\ii} \cdots \oint_{\mathcal{C}_n}\frac{\dd w_n}{2\pi\ii} \prod_{i=1}^n\prod_{j=1}^L \left[\frac{q w_i-x_j}{w_i-x_j}\frac{1-w_ix_j}{1-qw_ix_j}\right] \\
        \times \prod_{1\leq i<j\leq n} \left[\frac{w_j-w_i}{qw_j-w_i}\frac{1-qw_iw_j}{1- w_iw_j}\right] \prod_{i=1}^n \left[\frac{a\left(1-q w_i^2\right)}{w_i(1-a w_i)}\frac{y_{\nu_i}}{1-q w_i y_{\nu_i}}\prod_{j-1}^{\nu_i-1} \frac{1-w_iy_j}{1-qw_iy_j}\right].
	\end{multline}
One may then recognize the right hand side of \eqref{cauchy-cor-cinf} (under the substitutions $M \mapsto n$ and $z_i \mapsto w_i$) as being embedded within the integrand of \eqref{eq: G nu integral forumla thm-cinf}, leading to the equality
\begin{multline}
        \label{eq: G nu integral forumla thm-cinf2}
	    \lim_{c \to\infty} G_\nu(x_1,\dots,x_L) = \sum_{\kappa}
     \lim_{c \to\infty} G_\kappa(x_1,\dots,x_L) \oint_{\mathcal{C}_1}\frac{\dd w_1}{2\pi\ii} \cdots \oint_{\mathcal{C}_n}\frac{\dd w_n}{2\pi\ii} \lim_{c \to\infty} F_{\kappa}(w_1,\dots,w_n) \\
        \times \prod_{1\leq i<j\leq n} \left[\frac{w_j-w_i}{qw_j-w_i}\frac{1-qw_iw_j}{1- w_iw_j}\right] \prod_{i=1}^n \left[\frac{\left(w_i-a\right) \left(1-q w_i^2\right)}{w_i(1-a w_i)(1-w_i^2)} \frac{y_{\nu_i}}{1-q w_i y_{\nu_i}}\prod_{j-1}^{\nu_i-1} \frac{1-w_iy_j}{1-qw_iy_j}\right],
	\end{multline}
where it is necessary to assume the convergence constraints
\begin{equation}
		\label{Cauchy identity condition-2}
		\abs{\frac{1-x_iy_k}{1-qx_iy_k}\frac{q(1-w_j/y_k)}{1-qw_j/y_k}}\leq\rho<1, \hspace{0.5cm} \abs{\frac{1-x_iy_k}{1-qx_iy_k}\frac{1-qw_jy_k}{1-w_jy_k}}\leq\rho<1
	\end{equation}
in order to introduce the infinite sum over $\kappa$\footnote{More precisely, we have used the fact that this convergence is uniform with $(x_1,\dots,x_L)$ and $(w_1,\dots,w_n)$ ranging over compact subsets of $\mathbb{C}$, which is necessary to be able to switch the order of integration and summation.}. The contours $\mathcal{C}_1,\dots,\mathcal{C}_n$ are as previously, but we must now examine which poles of the function $\lim_{c\to\infty} F_{\kappa}(w_1,\dots,w_n|Y)$ they enclose (noting that the points $w_j=x_i$ are no longer singularities of the integrand). 

To deduce this, we make some assumptions concerning the parameters $q$ and $Y = (y_1,y_2,\dots)$\footnote{It is later possible to relax these constraints, since we ultimately derive \eqref{F-int-orthog-2}, which is an identity of rational functions in $q$ and $Y$ that holds when these parameters take values in certain compact subsets of $\mathbb{C}$; it must therefore hold generally.}. We note that one way to satisfy the constraints \eqref{Cauchy identity condition-2} is to assume that $|q|$ is arbitrarily small. This renders the first constraint in \eqref{Cauchy identity condition-2} trivial, while the second one becomes equivalent to establishing the bound $|x_i-y_k^{-1}| < \rho |w_j -y_k^{-1}|$ for all $1 \leq i \leq L$, $1 \leq j \leq n$, $k \in \mathbb{N}$; the latter constraint is satisfied if the points $y_k^{-1}$ are arbitrarily close to the points $x_i$, and are thereby enclosed by the integration contours. 

In summary, we may replace the contours $\mathcal{C}_1,\dots,\mathcal{C}_n$ in \eqref{eq: G nu integral forumla thm-cinf2} by a single contour $\mathcal{C}$ that encloses all points $Y^{-1} = (y_1^{-1},y_2^{-1},\dots)$ and no other singularities of the integrand, and such that $q\cdot \mathcal{C}$ is disjoint from the interior of $\mathcal{C}$. Reading off the coefficient of $\lim_{c \to\infty} G_\nu(x_1,\dots,x_L)$ on both sides of the resulting equation, we arrive at the following conjecture:
\begin{conj}
\label{conj:F-orth}
Fix a finite subset $\nu = \{\nu_1 > \cdots > \nu_n \geq 1\}$, and a second finite subset $\kappa$ whose cardinality satisfies $|\kappa| \leq n$. Then one has that
\begin{multline}
\label{F-int-orthog-2}
\oint_{\mathcal{C}}\frac{\dd w_1}{2\pi\ii} 
\cdots 
\oint_{\mathcal{C}}\frac{\dd w_n}{2\pi\ii}
\prod_{1\leq i<j\leq n}
\left[\frac{w_j-w_i}{qw_j-w_i}\frac{1-qw_iw_j}{1- w_iw_j}\right]
\\
\times
\prod_{i=1}^{n}
\left[
\frac{w_i-a}{w_i(1-a w_i)}
\frac{1-q w_i^2}{1-w_i^2}
\frac{y_{\nu_i}}{1-q w_i y_{\nu_i}}\prod_{j-1}^{\nu_i-1} \frac{1-w_iy_j}{1-qw_iy_j}
\right]
\lim_{c\to\infty}
F_{\kappa}(w_1,\dots,w_n)
=
\delta_{\kappa,\nu},
\end{multline}
where the contour $\mathcal{C}$ is a small, positively oriented circle surrounding the points $y_j^{-1}$, $j \geq 1$ and no other singularities of the integrand, such that $q\cdot \mathcal{C}$ is disjoint from the interior of $\mathcal{C}$.
\end{conj}
Although we have given an essentially complete formulation of \eqref{F-int-orthog-2}, it remains conjectural as we have not established that the functions $\lim_{c \to\infty} G_\kappa(x_1,\dots,x_L)$ are linearly independent. Conjecture \ref{conj:F-orth} has been extensively tested, and we plan to return to its proof in a later text.

\section{Open ASEP on the half-line}
\label{ASEP construction section}
In this section we demonstrate an important reduction of the partition function $G_{\nu/\mu}$ defined by \eqref{G partition function defn A eq}. In section \ref{sec:row-ops + sym func} it was demonstrated that $G_{\nu/\mu}$ can describe a discrete-time Markov process of interacting particles on a half-line with both creation and annihilation occurring at the origin. A continuous-time limit of this propagator will recover the dynamics of the \emph{asymmetric simple exclusion process} (ASEP) on the half-line with open boundary conditions.

\subsection{Markov generator}
We will consider the continuous-time asymmetric simple exclusion process on the semi-infinite half-line $\mathbb{N}:=\mathbb{Z}_{>0}$. As with the vertex model, we only consider configurations with a finite yet varying number of particles. Our conventions for ASEP coordinates follow that of Section \ref{sec:row-operators}. These configurations are indexed by the random variable $\nu = (\nu_1(t),\dots,\nu_n(t)) \in \mathbb{W}$ where $\nu_1 > \nu_2 > \cdots > \nu_n$ for some finite $n \geq 0$. We may also regard the number of particles $n$ as a random variable. 

These configurations evolve according to the following bulk transition rates 
\begin{align*}
	\nu \mapsto (\nu_1,\dots,\nu_i+1,\dots,\nu_n) & \text{ at rate $1$ if } \nu_{i-1}>\nu_i +1, \\
	\nu \mapsto (\nu_1,\dots,\nu_i-1,\dots,\nu_n) & \text{ at rate $q$ if } \nu_{i+1}<\nu_i -1.
\end{align*}
If we have the case where $\nu_{i-1} = \nu_i +1$ then this process is excluded (occurs at rate zero). We also have the boundary transition rates 
\begin{align*}
(\nu_1,\dots,\nu_n) \mapsto (\nu_1,\dots,\nu_n,1) & \text{ at rate $\alpha$ if } \nu_{n}>1 \text{ or } n=0, \\(\nu_1,\dots,\nu_n) \mapsto (\nu_1,\dots,\nu_{n-1}) & \text{ at rate $\gamma$ if } \nu_{n}=1.
\end{align*}
The bulk and boundary dynamics are depicted in Figure \ref{ASEP dynamics figure}.
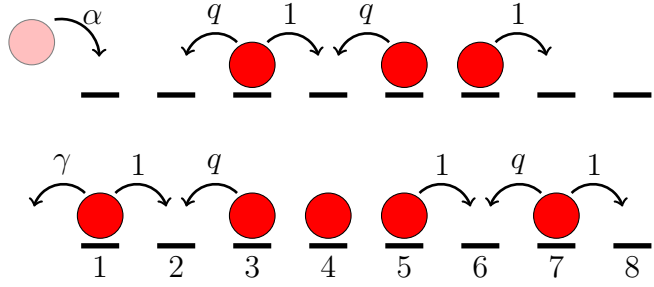
\begin{figure}
	\begin{center}
		\begin{tikzpicture}
			\foreach \x in {1,...,8}
			{\draw[line width=2pt] (\x-1,2) -- (\x-0.5,2);}

			\draw[black,fill=red] (2.25,2.4) circle (0.3);
			\draw[black,fill=red] (4.25,2.4) circle (0.3);
			\draw[black,fill=red] (5.25,2.4) circle (0.3);
			\draw[gray,fill=pink] (-0.65,2.7) circle (0.3);
			\draw[line width=1pt,->] (-0.35,3) arc (100:0:0.5);
			\node[above] at (0.15,2.8) {$\alpha$};
			\draw[line width=1pt,->] (2.05,2.7) arc (45:165:0.4);
			\node[above] at (1.75,2.8) {$q$};
			\draw[line width=1pt,->] (2.45,2.7) arc (135:15:0.4);
			\node[above] at (2.75,2.8) {$1$};
			\draw[line width=1pt,->] (4.05,2.7) arc (45:165:0.4);
			\node[above] at (3.75,2.8) {$q$};
			\draw[line width=1pt,->] (5.45,2.7) arc (135:15:0.4);
			\node[above] at (5.75,2.8) {$1$};

			\foreach \x in {1,...,8}
			{\draw[line width=2pt] (\x-1,0) -- (\x-0.5,0);
				\node[below] at (\x-0.75,0) {$\x$};}
			\draw[black,fill=red] (0.25,0.4) circle (0.3);
			\draw[black,fill=red] (2.25,0.4) circle (0.3);
			\draw[black,fill=red] (3.25,0.4) circle (0.3);
			\draw[black,fill=red] (4.25,0.4) circle (0.3);
			\draw[black,fill=red] (6.25,0.4) circle (0.3);
			\draw[line width=1pt,->] (0.05,0.7) arc (45:165:0.4);
			\node[above] at (-0.25,0.8) {$\gamma$};
			\draw[line width=1pt,->] (0.45,0.7) arc (135:15:0.4);
			\node[above] at (0.75,0.8) {$1$};
			\draw[line width=1pt,->] (2.05,0.7) arc (45:165:0.4);
			\node[above] at (1.75,0.8) {$q$};
			\draw[line width=1pt,->] (4.45,0.7) arc (135:15:0.4);
			\node[above] at (4.75,0.8) {$1$};
			\draw[line width=1pt,->] (6.05,0.7) arc (45:165:0.4);
			\node[above] at (5.75,0.8) {$q$};
			\draw[line width=1pt,->] (6.45,0.7) arc (135:15:0.4);
			\node[above] at (6.75,0.8) {$1$};
		\end{tikzpicture}
	\end{center}
	\caption{Dynamics of the ASEP on the half-line $\mathbb{N}$.}
    \label{ASEP dynamics figure}
\end{figure}
For fixed $\nu\in\mathbb{W}$, the Markov generator for this process is expressed as an operator $\mathscr{L}$ which acts on functions $f:\mathbb{W}\to\mathbb{C}$ given by
\begin{multline}
    \label{Half-line ASEP generator}
    \mathscr{L}[f](\nu) = (1-\eta^\nu_1)\left(\gamma f(\nu\cup\{1\})-\alpha f(\nu)\right) + \eta^\nu_1\left(\alpha f(\nu\setminus\{1\})-\gamma f(\nu)\right) \\
    + \sum_{s=1}^\infty \left[\eta^\nu_s(1-\eta^\nu_{s+1})+q\eta^\nu_{s+1}(1-\eta^\nu_{s})\right]\left(f(\nu^{s,s+1})-f(\nu)\right),
\end{multline}
where $\nu^{s,s+1}\in\mathbb{W}$ is the configuration obtained by the interchange of occupations at sites $s$ and $s+1$. We have also used the notation $\eta^\nu_j$ as the occupation of $\nu$ at site $j$ defined in Definition \ref{Occupation notation defn}. Using the Markov generator we can arrive at the evolution equation of the ASEP on the half-line
\begin{equation}
	\label{Half-line ASEP evolution equation}
	\frac{\dd}{\dd t} \psi_t(\nu) = \mathscr{L}[\psi_t](\nu),
\end{equation}
whose solution, $\psi_t:\mathbb{W} \to \mathbb{R}$, is an eigenfunction of the Markov generator indexed by continuous-time parameter $t\geq 0$. The rest of this section is devoted to the presentation a solution to \eqref{Half-line ASEP evolution equation} for an initially empty configuration with partially open boundary parameters.

\subsection{Transition probability from the vertex model}
\label{ASEP formulas section}
In this section we outline a method to obtain solutions to the half-line ASEP equations of motion \eqref{Half-line ASEP evolution equation} via a reduction of the symmetric function \eqref{G partition function defn A eq}. An explicit evaluation is provided for the case of an empty initial configuration of particles following from the results of Section \ref{Vertex propagator section}.

Let $\mu,\nu\in \mathbb{W}$ be arbitrary particle configurations. We will denote $\mathbb{P}_t(\mu\to\nu)$ as a solution to \eqref{Half-line ASEP evolution equation}, indexed by a time parameter $t\geq0$, subject to the initial condition 
\[\mathbb{P}_0(\mu\to\nu) = \delta_{\mu,\nu}.\]
That is, a probability of being in configuration $\nu$ at time $t$ after being initially in $\mu$. We refer to this as the \emph{transition probability}. The transition probability can be expressed as a formal solution to the evolution equation in terms of the Markov generator. This is given as
\begin{equation}
    \label{eq: transition prob exponential}
    \mathbb{P}_t(\mu\to\nu) = \bra{\mu}\exp(t\mathscr{L})\ket{\nu},
\end{equation}
where the exponential is regarded as the formal operator exponential. We now proceed by outlining the specialization required to recover the ASEP transition probability from the half-space six-vertex model.
\begin{prop}
    \label{ASEP cts time limit thm}
    Fix configurations $\mu,\nu\in\mathbb{W}$ and time $t\geq0$. The ASEP half-line transition probability is recovered as the limit of the symmetric function
	\begin{equation}
        \label{ASEP cts time limit}
		\mathbb{P}_t(\mu\to\nu) = \lim_{L\to \infty} G_{\nu/\mu}(x_1\dots,x_L|\yalph)\Big|_{y_j=1,\ x_i=1-(1-q)t/(2L)},
	\end{equation}
    where we have specified the spectral parameters $x_i = 1-(1-q)t/(2L)$ and vertical parameters $y_j=1$ for all $1\leq i\leq L$ and $j\in\mathbb{N}$ prior to taking the limit $L\to\infty$. The limit \eqref{ASEP cts time limit} holds provided that we choose boundary Markov rates as
	\be
        \alpha = \frac{a c(1-q)}{(1-a)(1-c)}, \qquad \gamma = - \frac{1-q}{(1-a)(1-c)},
        \label{eq:a2alpha}
    \ee
	where we are free to choose $a,c$ so that both $\alpha,\gamma \geq 0$ while we restrict $q\geq 0$. 
\end{prop}
We will first demonstrate a degeneration of the double-row 
\begin{lm}
    Fix $\epsilon>0$ let $x=1-(1-q)\epsilon$ and $y_j=1$ for all $j\in\mathbb{N}$. Then the double-row operator \eqref{AB row operator defn} is given by
    \begin{equation}
        \label{eq: A row-op ASEP limit}
        A(x|Y)\Big|_{x=1-(1-q)\epsilon,y_j=1} = 1 + 2\epsilon \mathscr{L} + O\left(\epsilon^2\right).
    \end{equation}
\end{lm}
\begin{proof}
    To demonstrate this, we choose $x_i = 1-(1-q)\epsilon,\ y_j=1$ and observe that the weights from \eqref{Stochastic 6VM weights table} with horizontal parameters $x_i^{-1}$ and vertical parameters $y_j$ acquire the following form. 
\begin{align}
    \label{epsilon bulk weights}
	\begin{tabular}{ccc}
		\qquad
		\begin{tikzpicture}
			\draw[gray,dashed,line width=1pt,-] (-1,0) -- (1,0);
			\draw[gray,dashed,line width=1pt,-] (0,-1) -- (0,1);
		\end{tikzpicture}
		\qquad
		&
		\qquad
		\begin{tikzpicture}
			\draw[gray,dashed,line width=1pt,-] (-1,0) -- (1,0);
			\draw[red,line width=2pt,->] (0,-1) -- (0,1);
		\end{tikzpicture}
		\qquad
		&
		\qquad
		\begin{tikzpicture}
			\draw[gray,dashed,line width=1pt,-] (0,1) -- (0,0) -- (-1,0);
			\draw[red,line width=2pt,->] (0,-1) -- (0,0) -- (1,0);
		\end{tikzpicture}
		\qquad
		\\
		\qquad
		$1$
		\qquad
		& 
		\qquad
		$q\epsilon + O\left(\epsilon^2\right)$
		\qquad
		& 
		\qquad
		$1-q\epsilon + O\left(\epsilon^2\right)$
		\qquad
		\\
        \quad
        \\
		\qquad
		\begin{tikzpicture}
			\draw[gray,dashed,line width=1pt,-] (-1,0) -- (1,0);
			\draw[gray,dashed,line width=1pt,-] (0,-1) -- (0,1);
			\draw[red,line width=2pt,->] (-1,0) -- (-0.2,0) -- (0,0.2) -- (0,1);
			\draw[red,line width=2pt,<-] (1,0) -- (0.2,0) -- (0,-0.2) -- (0,-1);
		\end{tikzpicture}
		\qquad
		&
		\qquad
		\begin{tikzpicture}
			\draw[red,line width=2pt,->] (-1,0) -- (1,0);
			\draw[gray,dashed,line width=1pt,-] (0,-1) -- (0,1);
		\end{tikzpicture}
		\qquad
		&
		\qquad
		\begin{tikzpicture}
			\draw[gray,dashed,line width=1pt,-] (0,-1) -- (0,0) -- (1,0);
			\draw[red,line width=2pt,->] (-1,0) -- (0,0) -- (0,1);
		\end{tikzpicture}
		\qquad
		\\
		\qquad
		$1$
		\qquad
		& 
		\qquad
		$\epsilon + O\left(\epsilon^2\right)$
		\qquad
		&
		\qquad
		$1-\epsilon + O\left(\epsilon^2\right)$
		\qquad
		\\
	\end{tabular} 
\end{align}
This table represents the weights of the upper row in the double row of \eqref{Row-operator A defn}. The weights of the lower row can be obtained by rotating these vertices. The boundary weights have a similar form.
\begin{align}
	\label{epsilon K weights}
	\begin{tabular}{cccc}
		\qquad
		\begin{tikzpicture}[baseline={([yshift=-.5ex]current bounding box.center)}]
			\draw[gray,dashed,line width=1pt,-] (1,0) -- (0,1) -- (1,2);
			\draw[blue,fill=blue] (0,1) circle (0.1cm);
		\end{tikzpicture}
		\qquad
		&
		\qquad
		\begin{tikzpicture}[baseline={([yshift=-.5ex]current bounding box.center)}]
			\draw[gray,dashed,line width=1pt,-] (1,0) -- (0,1);
			\draw[red,line width=2pt,->] (0,1) -- (1,2);
			\draw[blue,fill=blue] (0,1) circle (0.1cm);
		\end{tikzpicture}
		\qquad
		&
		\qquad
		\begin{tikzpicture}[baseline={([yshift=-.5ex]current bounding box.center)}]
			\draw[gray,dashed,line width=1pt,-] (0,1) -- (1,2);
			\draw[red,line width=2pt,-] (1,0) -- (0,1);
			\draw[red,line width=2pt,->] (1,0) -- (0.5,0.5);
			\draw[blue,fill=blue] (0,1) circle (0.1cm);
		\end{tikzpicture}
		\qquad
		&
		\qquad
		\begin{tikzpicture}[baseline={([yshift=-.5ex]current bounding box.center)}]
			\draw[gray,dashed,line width=1pt,-] (0,1) -- (1,2);
			\draw[red,line width=2pt,->] (1,0) -- (0,1) -- (1,2);
			\draw[red,line width=2pt,->] (1,0) -- (0.5,0.5);
			\draw[blue,fill=blue] (0,1) circle (0.1cm);
		\end{tikzpicture}
		\qquad
        \\
        \quad
		\\
		\qquad
		$1-2\alpha \epsilon +  O\left(\epsilon^2\right)$
		\quad
		&
		\qquad
		$2\alpha \epsilon +  O\left(\epsilon^2\right)$
		\qquad
		&
		\qquad
		$2\gamma \epsilon +  O\left(\epsilon^2\right)$
		\qquad
		&
		\qquad
		$1-2\gamma \epsilon +  O\left(\epsilon^2\right)$
		\qquad
		\\
	\end{tabular}
\end{align}
The double-row operator \eqref{eq: A row-op ASEP limit} can be determined by calculating for specific configurations $\mu,\nu\in\mathbb{W}$ the partition function
using the weights in \eqref{epsilon bulk weights},\eqref{epsilon K weights} and matching with the action of the ASEP generator \eqref{Half-line ASEP generator}.
\end{proof}
\begin{proof}[Proof of Proposition \ref{ASEP cts time limit thm}]
Using the double-row operator definition of the partition \eqref{G partition function defn A eq} and the scaling \eqref{eq: A row-op ASEP limit} with $\epsilon=t/(2L)$, we can arrive at the following expression for the partition function
\begin{equation}
    G_{\nu/\mu}(x_1\dots,x_L|\yalph)\Big|_{y_j=1,\ x_i=1-(1-q)t/(2L)} = \bra{\mu} \left(1+\frac{t}{L}\mathscr{L}+ O\left(\frac{1}{L^2}\right)\right)^L \ket{\nu}.
\end{equation}
In order to take the limit we use the definition of the operator exponential. This yields
\begin{equation}
    \lim_{L\to\infty} \left(1+\frac{t}{L}\mathscr{L}+O\left(\frac{1}{L^2}\right)\right)^L = \lim_{L\to\infty} \left(1+\frac{t}{L}\mathscr{L}\right)^L = \exp(t\mathscr{L}).
\end{equation}
This recovers the formal solution for ASEP transition probability \eqref{eq: transition prob exponential}.
\end{proof}

\subsection{Particle injection case}
In this section we provide an explicit expression for the ASEP transition probability \eqref{eq: transition prob exponential} under the specialization $\gamma=0$, whereby particles may enter from the boundary at rate $\alpha$ but may not exit the system.
\begin{defn}
    \label{defn: ASEP contours}
    Fix an alphabet $(x_1,\dots,x_L)$ and let $\{\mathcal{C}_1,\dots,\mathcal{C}_n\}$ be a set of contours satisfying the conditions of Definition \ref{defn: nested contours} whilst also surrounding 1. We define $\{\mathcal{D}_1,\dots,\mathcal{D}_n\}$ to be these nested contours satisfying the conditions of Definition \ref{defn: nested contours} having taken $x_i=1$. 
\end{defn}
\begin{thm}
    Under the limit $c\to\infty$ ($\gamma\to0$) for $\alpha+q\neq 1$, the ASEP transition probability is given by
    \begin{multline}
	    \mathbb{P}_t(\emptyset\to\nu) = \alpha^n\mathrm{e}^{-\alpha t}\oint_{\mathcal{D}_1}\frac{\dd w_1}{2\pi\ii} \cdots \oint_{\mathcal{D}_n}\frac{\dd w_n}{2\pi\ii} \prod_{1\leq i<j\leq n} \left[\frac{w_j-w_i}{qw_j-w_i}\frac{1-qw_iw_j}{1- w_iw_j}\right] \\
        \times\prod_{i=1}^n \left[\frac{1-q w_i^2}{w_i(q+\alpha-1-\alpha w_i)(1-qw_i)} \left(\frac{1-w_i}{1-qw_i}\right)^{\nu_i-1} \exp(\frac{(1-q)^2 w_i t}{(1-w_i)(1-qw_i)})\right],
    \end{multline}
     where the contours satisfy the conditions of Definition \ref{defn: ASEP contours}. That is, they surround the essential singularity at $w_i=1$.
\end{thm}
\begin{proof}
Consider the integral formula \eqref{eq: G nu integral forumla thm} with contours $\{\mathcal{C}_1,\dots,\mathcal{C}_n\}$ taken to surround 1 as well as the points $x_1\dots,x_L$. The conditions of the contours from Definition \ref{defn: nested contours} also means that the contours must be nested whilst they do not intersect. 

In the limit $c\to\infty$, only the term $r=0$ in the sum over subset formula \eqref{eq:Zsubset} for $Z_L$ survives, see Corollary~\ref{cor:triangle-freeze}. Hence $Z_L$ trivializes and completely factorizes into factors of the form $1-h(x_i)$. 

Upon substitution in \eqref{eq: G nu integral forumla thm} and setting $x_i=1-(1-q)t/(2L)$, the limit $L\to\infty$ can be taken on the integrand in a straightforward manner. Note that in this limit we have written  $a=\alpha/(\alpha+q-1)$ for $\alpha+q\neq 1$. We have also calculated under the ASEP limit
\[\lim_{L\to\infty}\left.\prod_{j=1}^L \left[\frac{q w_i-x_j}{w_i-x_j}\frac{1-w_ix_j}{1-qw_ix_j}\right]\right|_{x_j=1-(1-q)t/(2L)} = \exp(\frac{(1-q)^2 w_i t}{(1-w_i)(1-qw_i)}).\]

In order to obtain the result of the theorem, the $L\to\infty$ limit must must also be simultaneously applied to the contours. This will deform the contours $\mathcal{C}_i\mapsto\mathcal{D}_i$ using the contours from Definitions \ref{defn: nested contours} and \ref{defn: ASEP contours}. This deformation occurs without crossing over any other singularities of the integrand, thus yielding the result.

\end{proof}

\section*{Compliance with ethical standards}

The authors are not aware of any potential conflict of interest.

\appendix
\section{Properties of Pfaffians}
A Pfaffian is taken over a $2m\times2m$-dimensional skew-symmetric matrix. It is defined as 
\begin{equation}
	\pf A = \frac{1}{2^m m!} \sum_{\sigma\in S_m} (-1)^\abs{\sigma} \prod_{i=1}^{m} a_{\sigma(2i-1),\sigma(2i)}.
\end{equation}
Importantly, the determinant of skew-symmetric matrix is the square of a polynomial in its entries. The Pfaffian can be identified as this polynomial. That is, so long as $A$ is an even-dimensional skew-symmetric matrix then
\begin{equation}
	\det A = \left(\pf A\right)^2.
\end{equation}
The following identity appears in \cite{stembridge_nonintersecting_1990}.
\begin{lm}
\label{lm:PFsum}
    Let $A$ and $B$ be skew-symmetric $m\times m$ matrices. Then we have 
\begin{align}
    \pf \left( A+ B\right) = 
    \sum_{r=0}^m(-1)^{r/2}\sum_{S\subseteq [m]\atop \abs{S}=r} (-1)^{\sum_i S_i} 
    \pf \left( A_S\right)
    \pf \left( B_{\comp{S}}\right),
    \label{eq:PFsum}
\end{align}
where $\comp{S}$ denotes the set which is the complement to the set $S$ w.r.t. $[m]$.
\end{lm}

\bibliography{Library}{}
\bibliographystyle{alphamod}

\end{document}